\documentclass{amsart}

\usepackage{fullpage}
\usepackage{amssymb}
\usepackage[all]{xy}

\DeclareMathOperator{\Gal}{Gal}
\DeclareMathOperator{\Spec}{Spec}
\DeclareMathOperator{\coker}{coker}
\DeclareMathOperator{\Eul}{Eul}
\DeclareMathOperator{\Hom}{Hom}

\DeclareMathOperator{\id}{id}
\DeclareMathOperator{\Det}{Det}
\DeclareMathOperator{\Ind}{Ind}

\newtheorem{theorem}[equation]{Theorem}
\newtheorem{conjecture}[equation]{Conjecture}
\newtheorem{notation}[equation]{Notation}
\newtheorem{assumption}[equation]{Assumption}
\newtheorem{proposition}[equation]{Proposition}
\newtheorem{lemma}[equation]{Lemma}
\newtheorem{corollary}[equation]{Corollary}
\theoremstyle{definition}
\newtheorem{definition}[equation]{Definition}
\newtheorem{remark}[equation]{Remark}
\newtheorem{example}[equation]{Example}

\numberwithin{equation}{section} 
\numberwithin{figure}{section}

\title[\'Etale cohomology, cofinite generation, $p$-adic $L$-functions]{\'Etale cohomology, cofinite generation, and $p$-adic $L$-functions}

\author{Rob de Jeu}
\address{Faculteit der Exacte Wetenschappen\\Afdeling Wiskunde\\VU University Amsterdam\\De Boelelaan 1081a\\1081 HV Amsterdam\\The Netherlands}

\author{Tejaswi Navilarekallu}
\address{Faculteit der Exacte Wetenschappen\\Afdeling Wiskunde\\VU University Amsterdam\\De Boelelaan 1081a\\1081 HV Amsterdam\\The Netherlands}
\curraddr{Optiver\\ Strawinskylaan 3095\\ 1077 ZX Amsterdam\\ The Netherlands}

\begin{document}

\begin{abstract}
Let $p$ be a prime number.
We study certain \'etale cohomology groups with coefficients
associated to a $p$-adic Artin representation of the Galois group of a number field~$k$.
These coefficients are equipped with a modified Tate twist involving a $p$-adic index.
The groups are cofinitely generated, and we determine the
additive Euler characteristic. If $k$ is totally real and the
representation is even, we study the relation between the behaviour
or the value
of the $p$-adic $L$-function at the point $e$ in its domain,
and the cohomology groups with $p$-adic twist $1-e$.
In certain cases this gives short proofs of a
conjecture by Coates and Lichtenbaum, and the equivariant Tamagawa
number conjecture for classical $L$-functions. For $p=2$ our results
involving $p$-adic $L$-functions depend on a conjecture in Iwasawa
theory.
\end{abstract}



\maketitle

\def\rightiso{\buildrel{\sim}\over{\rightarrow}}
\newcommand{\dual}[1]{#1^{\vee}}

\def\Sig{\Sigma_\infty}
\def\OKP{\Omega_{K,P}}

\newcommand{\kbar}{\overline{k}}
\def\dEul#1{\textup{Eul}_{#1}^{\langle\rangle}}
\def\Eul#1{\textup{Eul}_{#1}}
\def\Eulpol#1{F_{#1}}

\def\EC_#1{\widehat{EC}_{p,#1}}
\def\O{\mathcal{O}}
\newcommand{\Ochi}{\mathcal{O}(\chi)}
\newcommand{\Oeta}{\mathcal{O}(\eta)}
\newcommand{\Opsi}{\mathcal{O}(\psi)}
\newcommand{\Oe}{\mathcal{O}_{E}}
\newcommand{\Oep}{\mathcal{O}_{E'}}

\def\o{\omega}
\def\op{\o_p}
\def\psio{\psi_p^{\langle\rangle}}
\def\psip{\psi_p}

\newcommand{\B}{\mathfrak{B}}
\newcommand{\M}{\mathcal{M}}
\def\H{\overline{H}}
\def\tM{\widetilde{\mathcal{M}}}

\def\Fr{\textup{Fr}}
\def\Nm{\text{Nm}}
\def\rank{\mathrm{rank}}
\def\corank{\mathrm{corank}}

\def\tT{\widetilde T}
\def\mtt#1{\langle#1\rangle}
\def\cork#1#2#3{r_{#1}(#2,#3)}
\def\sec#1#2{r(#1,#2)}

\newcommand{\lat}[2]{M(#1,#2)}
\newcommand{\klat}[3]{W(#2,#3)[#1]}
\newcommand{\qlat}[2]{W(#1,#2)}
\newcommand{\vlat}[2]{V(#1,#2)}
\newcommand{\slat}[1]{M(#1)}
\newcommand{\sklat}[2]{W(#2)[#1]}
\newcommand{\sqlat}[1]{W(#1)}
\newcommand{\tqlat}[2]{W(#1,#2)\langle 1-e \rangle}
\newcommand{\latmtt}[3]{M(#1,#2)\langle #3 \rangle}
\newcommand{\klatmtt}[4]{W(#2,#3)[#1]\langle #4 \rangle}
\newcommand{\qlatmtt}[3]{W(#1,#2)\langle #3 \rangle}
\newcommand{\vlatmtt}[3]{V(#1,#2)\langle #3 \rangle}
\newcommand{\slatmtt}[2]{M(#1)\langle #2 \rangle}
\newcommand{\sklatmtt}[3]{M_{#1}(#2)\langle #3 \rangle}
\newcommand{\sqlatmtt}[2]{W(#1)\langle #2 \rangle}

\def\kbar{\overline k}
\newcommand{\C}{\mathbb C}
\newcommand{\F}{\mathbb F}
\newcommand{\Fbar}{\overline{\mathbb{F}}}
\newcommand{\Q}{\mathbb Q}
\newcommand{\Qbar}{\overline{\Q}}
\newcommand{\Qpbar}{\overline{\Q}_p}
\newcommand{\R}{\mathbb R}
\newcommand{\Z}{\mathbb Z}

\def\tgammanought{\widetilde{\gamma}_0}
\def\tg{\tilde g}
\def\tPol{\widetilde P}
\def\sp{s'}

\def\tor{\textup{tor}}
\def\ord{\mathrm{ord}}

\section{Introduction}

Let $k$ be a number field, $p$ a prime number, $ E $ a finite extension
of $ \Q_p $ with valuation ring $\Oe$,
and $ \eta : G_k \to E $ an Artin character with dual character
$ \eta^\vee $, that is, the character
of an Artin representation of $G_k = \Gal(\kbar/k)$.
If $ S $ is a finite set of finite primes of $ k $, $ m \le 0 $ an integer,
and $ \sigma : E \to \C $ an embedding,
then the value $ L_S(m,\sigma \circ \eta^\vee, k) $ of the classical truncated Artin $ L $-function
is in $ \sigma(E) $, and if we let $L_S^\ast(m,\eta^\vee,k) = \sigma^{-1} ( L_S(m,\sigma \circ \eta^\vee, k) )$
in $ E $ then this is independent of $ \sigma $ (see Section~\ref{pL-functions}).

We call $ \eta $  realizable over $E$ if the corresponding representation can
be defined over~$ E $.
This representation can then be obtained as
$\lat{E}{\eta} \otimes_{\Oe}E \simeq \lat{E}{\eta} \otimes_{\Z_p} \Q_p$
for some finitely generated torsion-free
$\Oe$-module $\lat{E}{\eta} $ on which $G_k$ acts (we shall call $\lat{E}{\eta}$ an $\Oe$-lattice for $\eta$).
If $ S $ includes all the finite primes of $ k $ at which $ \eta $ is ramified, and $ \O_{k,S} $ is obtained
from the ring of algebraic integers $ \O_k $ of $ k $ by inverting all primes in $ S $, then we may view
$ \lat{E}{\eta} $ and $\lat{E}{\eta} \otimes_{\Oe} E/\Oe \simeq \lat{E}{\eta} \otimes_{\Z_p} \Q_p/\Z_p$
as sheaves for the \'etale topology on the open subscheme $\Spec \O_{k,S}$ of $\Spec \O_k$.
We let $ \alpha : \Spec \O_{k,S} \to \Spec \O_k $ be the inclusion, but in \'etale cohomology groups
we shall write $ \O_k $ instead of $ \Spec \O_k $ and similarly for $ \O_{k,S} $.

In the special case that $ p $ is odd, $ E=\Q_p $, $ m <0 $, and $L_S^\ast(m,\eta^\vee,k) \neq 0$,
according to Conjecture~1 of~\cite{Co-Li} we should have that
the \'etale cohomology groups
$H^i(\O_k,\alpha_! (\lat{E}{\eta} \otimes_{\Z_p} \Q_p/\Z_p(m)))$
are finite for all $i \geq 0$, trivial for $i > 3$, and that
\begin{equation} \label{CL-statement}
|L_S^\ast(m,\eta^\vee,k)|_p 
=
 \prod_{i=0}^3 \# H^i(\O_k, \alpha_! (\lat{E}{\eta} \otimes_{\Z_p} \Q_p/\Z_p(m)))^{(-1)^i}
\,.
\end{equation}
(Note that on page~502 of loc.\ cit.\ the inverse of the arithmetic Frobenius is used in the definition
of the $ L $-function for $ \eta $, resulting in the standard
$ L $-function for $ \eta^\vee $.)
We observe here that by \cite[VII Theorem~12.6]{Neu99} and the definition of the completed $ L $-series,
the non-vanishing of the $ L $-value is equivalent with
$ k $ being totally real, and $ \eta(c) = (-1)^{m-1}\eta(\id_{\kbar}) $
for all complex conjugations $ c $ in~$ G_k $.

B\'ayer and Neukirch proved this conjecture for the trivial character \cite[Theorem~6.1]{Ba-Ne}
assuming the main conjecture of Iwasawa theory for this character
(later proved by Wiles in far greater generality; see \cite[Theorem~1.2 and~1.3]{Wiles90}).
In this case the conjecture is, in fact, equivalent to an earlier conjecture of Lichtenbaum
(see \cite[Conjecture~9.1]{Licht72} and
\cite[Conjecture~3.1]{Co-Li})
because $H^i(\O_k,\alpha_! (\lat{E}{\eta} \otimes_{\Z_p} \Q_p/\Z_p(m)))$ is dual to
$H^{2-i}(\O_{k,S}, \lat{E}{\eta^\vee} \otimes_{\Z_p} \Q_p/\Z_p(1-m))$ for some lattice
$\lat{E}{\eta^\vee}$ (see Remark~\ref{introremark}).
Using this duality the proof of B\'ayer-Neukirch relates the right-hand side of~\eqref{CL-statement} to the
$ p $-adic absolute value of the value of a certain $ p $-adic $ L $-function at $ m $, which equals the left-hand
side of~\eqref{CL-statement} by an interpolation formula (see~\eqref{introinterpol}).

Again letting $ p $ be any prime number,
one may therefore expect the $p$-adic absolute value of the value of a $p$-adic $L$-function at
an integer $ m $ to be related to the multiplicative Euler characteristic of certain \'etale cohomology groups with an
$ m $-th Tate twist. 
In order to generalize this interpretation from integers $ m $ to (almost)
every point in the much larger $ p $-adic domain of definition
of the $ p $-adic $ L $-function, we now introduce modified Tate twists
indexed by suitable $ p $-adic numbers.
If $ k $ is any number field and $ A $ any $ \Z_p[G_k] $-module, then $ A(m) $ is obtained from $ A $
by multiplying the action  of $ G_k $ by the $ m $-th power of the $ p $-cyclotomic character
$ \psip : G_k \to \Z_p^\times $.  For $ g $ in $ G_k $ write 
\begin{equation}\label{psi0def}
 \psip(g) = \op(g) \psio(g) 
\end{equation}
with $ \psio(g) $ in $ 1+2p\Z_p $ and $ \op = \o_{p,k} : G_k \to \mu_{\phi(2p)} \subset \Z_p^\times $ the
Teichm\"uller character of $ G_k $ for $ p $.
Let $ k_\infty/k $ be the cyclotomic $ \Z_p $-extension of $ k $, and $ \tgammanought $
in $ G_k $ 
a lift of a topological generator $ \gamma_0 $ of $ \Gal(k_\infty/k) $.  Let
$ q_k = \psip(\tgammanought) = \psio(\tgammanought) $, and for a finite extension $ E $
of $ \Q_p $ (always with $ p $-adic absolute value $ |\cdot|_p $ normalized by $ |p|_p = 1/p $), put
\begin{equation} \label{BEdef}
 \B_k(E) = \{ e \text{ in } E \text{ with } |e|_p < |q_k-1|_p^{-1} p^{-1/(p-1)}\} 
\,,
\end{equation}
i.e., those $ e $ in $ E $ where $ q_k^e $ and $ \psio(g)^e $ for $ g $ in $ G_k $ converge.
Then $ \B_k(E) $ contains $ \Oe $ but it can be much larger.
For any $ \Oe[G_k] $-module $ A $ and $ e $ in $ \B_k(E) $, we let
$ A\mtt{e}$ be $ A $ with the action of $ G_k $ multiplied by $ (\psio)^e $.
Note that for an integer $ m $, $ A(m) $ can be obtained from $ A\mtt{m} $ by twisting the action of
$ G_k $ with $ \op^m $.
Since $ \psio(g) $ is often denoted as $\langle\psip(g)\rangle$, we think of this as a 
`diamond' twist
and use notation and terminology accordingly.
Using this twist systematically leads to a more general result and simpler proofs:
it removes the need to adjoin the $ 2p $-th roots of unity (thus
avoiding many technical complications when $ p=2 $).

With $ \qlat{E}{\eta} = \lat{E}{\eta} \otimes_{\Z_p} \Q_p/\Z_p $ as before,
Theorem~\ref{cogentheorem}
concerns the structure of the \'etale cohomology groups $H^i(\O_{k,S},\qlatmtt{E}{\eta}{1-e}) $.
(We state our results for \'etale cohomology groups with
torsion coefficients. For some reasons why we prefer these over alternatives we refer
to Remark~\ref{introremark} below.)
For $i \ge 3$, it is well-known that this group is trivial if $ p \ne 2 $,
and is finite and easily computed if $ p=2 $ (see Remark~\ref{coh_dim}), so
we only consider $ i=0,1,2 $.  For its statement, let us 
call an $\Oe$-module $A$ cofinitely generated if its Pontryagin dual $A^\vee = \Hom_{\Z_p}(A,\Q_p/\Z_p)$,
on which $\Oe$ acts through its action on $A$, is a finitely generated $\Oe$-module.  
In this case we define $\corank_{\Oe} A = \rank_{\Oe} A^\vee$.
We note that then also the natural map $ A \to (A^\vee)^\vee $ is an isomorphism
(see Remark~\ref{exactness}).
We write $ \B_\eta(E) $ for $ \B_k(E) $ if $ \eta $ does not contain the trivial character, and for
$ \B_k(E) \setminus\{1\} $ if it does.  Finally, if $k$ is totally real then we call an Artin character $\eta$
of $G_k$ even if $ \eta(c) = \eta(\id_{\kbar}) $ for every complex conjugation $ c $ in $ G_k $.

We first extend results on cofinite generation that already have a
long history.  For example, the equivalent statement for Galois cohomology
of the first part of Theorem~\ref{cogentheorem}(1) below follows
already by combining \cite[Theorems~2.1 and~3.1]{Tate62}
with the Corollary on page~260 of \cite{Tate76};
a succinct general overview is given in Appendix~A.1 of~\cite{Per00}.
(For the relation between those cohomology groups and the groups
we consider, we refer to Remark~\ref{introremark} or Remark~\ref{duality-remark}.)
However, the uniform bound across all $e$ and $E$ in the second part of Theorem~\ref{cogentheorem}(1) is new.

\begin{theorem} \label{cogentheorem}
Let $ k $ be a number field, $ p $ a prime number, $ E $ a finite extension of $ \Q_p $,
and $ \eta : G_k \to E $ an Artin character realizable over~$ E $.
Let $ \lat{E}{\eta} $ be an $\Oe$-lattice for $\eta$ and let  $ \qlat{E}{\eta} = \lat{E}{\eta} \otimes_{\Z_p} \Q_p/\Z_p $.
Assume that $ S $ is a finite set of finite primes of $ k $ containing the primes above $ p $
as well as the finite primes at which $ \eta $ is ramified.
Then for $ e $ in $ \B_k(E) $ the following hold.
\begin{enumerate}
\item
$ H^i(\O_{k,S},\qlatmtt{E}{\eta}{1-e}) $ for $ i \ge 0 $ is cofinitely generated.
There is a constant $D = D(S,\eta,k)$ independent of $e$, $E$ and the choice of $\lat{E}{\eta}$
such that each of these groups can be cogenerated by at most $D$ elements.

\item
Let
$ r_i = \cork{i,S}{1-e}{\eta}=\corank_{\Oe} H^i(\O_{k,S},\qlatmtt{E}{\eta}{1-e}) $
for $ i=0,1,2 $.  Then $ r_i $ is independent of the
choice of $ \lat{E}{\eta} $, and the Euler characteristic
$  r_0 - r_1 + r_2 $ equals
$$
\qquad\quad
 -[k:\Q]\cdot\eta(\id_{\kbar}) 
+ \sum_{v \in\Sig}\corank_{\Oe}H^0(\Gal(\overline{k_v}/k_v),\qlat{E}{\eta})
\,,
$$  
where $\Sig$ is the set of all infinite places of $k$ and $ k_v $
the completion of $ k $ at $ v $.
This quantity is independent of $ e $,
non-positive, and is zero if and only if $k$ is totally real and $\eta$ is even.
Moreover, $ r_i $ is independent of $ S $ for $i = 0$, and also for $i = 1,2$ if $ e\ne 0 $.

\item
$ \cork{0,S}{1-e}{\eta} = 0 $ if $ e \ne 1 $, and 
$ \cork{0,S}{0}{\eta} $ equals the multiplicity of the trivial character in $ \eta $.
Moreover, the size of $H^0(\O_{k,S},\qlatmtt{E}{\eta}{1-e}) $ is a locally
constant function for $e$ in $ \B_\eta(E) $.

\item
If $ p\ne 2 $ then
$H^2(\O_{k,S},\qlatmtt{E}{\eta}{1-e})$ is $p$-divisible.
If $ p=2 $ then
\begin{alignat*}{1}
\quad\qquad
&
H^2(\O_{k,S},\qlatmtt{E}{\eta}{1-e})
\simeq 
\\
&
2H^2(\O_{k,S},\qlatmtt{E}{\eta}{1-e}) 
\bigoplus \oplus_{v\in\Sig} H^0(\Gal(\overline{k_v}/k_v),\qlat{E}{\eta})/2
\,,
\end{alignat*}
and $2H^2(\O_{k,S}, \qlatmtt{E}{\eta}{1-e})$ is $2$-divisible.
\end{enumerate}
\end{theorem}

We mention that if $H^i(\O_{k,S},\qlatmtt{E}{\eta}{1-e}) $ is finite for $i =0$ or~$1$, then
Proposition~\ref{contprop}, Remark~\ref{directsumremark} and
Remark~\ref{countremark} express its size and structure
using cohomology groups with finite coefficients.
By Theorem~\ref{cogentheorem}(2), this can
happen for $i=1$ only when $k$ is totally real and $\eta$ is even,
in which case this group is finite for $ i=0 $ and trivial for $ i=2 $
by that theorem.
Of independent interest is that for such $ k $ and $ \eta $
a short exact sequence of coefficients gives rise to
a nine term exact sequence involving only $H^i(\O_{k,S},\cdot)$ for $i =0,1,2$ (see Lemma~\ref{9lemma}).

\begin{remark} \label{introremark}
We formulated Theorem~\ref{cogentheorem} (and Theorem~\ref{introtheorem} below)
for the groups $ H^i(\O_{k,S},\qlatmtt{E}{\eta}{1-e}) $
because this gives a uniform point of view for all primes and all
modified Tate twists, unlike the groups $ H^i(\O_k, \alpha_! \qlatmtt{E}{\eta}{e}) $
and $ H_{cts}^i(\O_{k,S},\latmtt{E}{\eta}{1-e}) $ discussed
below. On the other hand, the groups $ \varprojlim_n H_c^i(\O_k,\alpha_! (\latmtt{E}{\eta}{e}/p^n)) $
below also provide a uniform approach but they are technically more
difficult to handle. For example, 
in the proof of Theorem~\ref{ebn} the only cohomology groups
that we have to analyse in detail in our approach are of low degree, and those can be
described fairly explicitly.

We refer to Remark~\ref{duality-remark} and the proof
of Proposition~\ref{cofingen} for details of the following discussion.

Let $k$ be any number field and $ \eta : G_k \to E $ any Artin character.
Fix a finite set $S$ of finite primes of $k$ containing all primes at which $\eta$ is
ramified as well as all 
primes of $k$ lying above $p$, and let $\alpha: \Spec \O_{k,S} \rightarrow \Spec \O_k$ be the natural inclusion.
Then for an appropriate choice of lattices we have
\begin{alignat*}{1}
    \varprojlim_n H_c^i(\O_k,\alpha_! (\latmtt{E}{\eta}{e}/p^n))
& \simeq
    \varprojlim_n H_c^i(\O_{k,S},\latmtt{E}{\eta}{e}/p^n) 
\\
& \simeq
   \left( H^{3-i}(\O_{k,S},\qlatmtt{E}{\eta^\vee \op}{1-e}) \right)^\vee
\end{alignat*}
for $e$ in $\B_k(E)$ and $i = 0,\dots,3$, where $H_c^i$ denotes cohomology with compact support.
If $p \ne 2$ and all $H^i(\O_{k,S},\qlatmtt{E}{\eta^\vee \op}{1-e})$ are finite then also
\begin{equation} \label{introdual}
H^i(\O_k, \alpha_! \qlatmtt{E}{\eta}{e}) \simeq H^{2-i}(\O_{k,S}, \qlatmtt{E}{\eta^\vee \op}{1-e})
\end{equation}
for $e$ in $\B_k(E)$, $i = 0, 1, 2$, the group on the left being trivial for $i \ge 3$.

For $ p $ again an arbitrary prime,
let $\Omega_S$ be the maximal extension of $k$ that is
unramified outside of $S$ and the infinite primes of $k$, and let $G_S = \Gal(\Omega_S/k)$.
Then $H^i(\O_{k,S},\qlatmtt{E}{\eta}{1-e}) \simeq H^i(G_S,\qlatmtt{E}{\eta}{1-e})$
where the right-hand side denotes continuous group cohomology.

Finally,
for continuous \'etale cohomology groups $ H_{cts}^i $ with
$\latmtt{E}{\eta}{1-e} \simeq \varprojlim_n \latmtt{E}{\eta}{1-e}/p^n$
as coefficients, we have a long exact sequence
\begin{alignat*}{1}
&  \cdots \to H_{cts}^i(\O_{k,S},\latmtt{E}{\eta}{1-e}) \to H_{cts}^i(\O_{k,S},\latmtt{E}{\eta}{1-e})\otimes_{\Z_p} \Q_p
\\
& \to
 H^i(\O_{k,S},\qlatmtt{E}{\eta}{1-e}) \to H_{cts}^{i+1}(\O_{k,S},\latmtt{E}{\eta}{1-e}) 
\to \cdots
\,,
\end{alignat*}
where $H_{cts}^i(\O_{k,S},\latmtt{E}{\eta}{1-e}) \simeq \varprojlim_n H^i(\O_{k,S},\latmtt{E}{\eta}{1-e}/p^n)$
is a finitely generated $\Oe$-module.
Hence $\rank_{\Oe} H_{cts}^i(\O_{k,S},\latmtt{E}{\eta}{1-e}) = \corank_{\Oe} H^i(\O_{k,S},\qlatmtt{E}{\eta}{1-e})$.
If $ H^j(\O_{k,S},\qlatmtt{E}{\eta}{1-e}) $ is finite for $ j=i $ and $ i+1 $,
then
$ H^i(\O_{k,S},\qlatmtt{E}{\eta}{1-e}) \simeq H_{cts}^{i+1}(\O_{k,S},\latmtt{E}{\eta}{1-e})$
as well.
\end{remark}

Our second main result Theorem~\ref{introtheorem} is a relation between the cohomology groups
$H^i(\O_{k,S},\qlatmtt{E}{\eta}{1-e})$
and a certain $p$-adic $L$-function $L_{p,S}(e,\eta,k)$.  For such a $p$-adic $L$-function to be defined
and non-trivial, 
we must have $k$ totally real and $\eta$ even.
By Theorem~\ref{cogentheorem}(2),
this is also the only case where those cohomology groups can
simultaneously be finite for $ i=0,1 $ and~$ 2 $.
For the sake of clarity we write $\chi$ for an even Artin character when $k$ is a totally real number
field, and $\eta$ for an arbitrary Artin character when $k$ is an arbitrary number field.

The conjecture of Coates and Lichtenbaum will be deduced from the purely $ p $-adic result
Theorem~\ref{introtheorem} for $ \chi = \eta^\vee\op^{1-m} $ and $ e=m $, by using~\eqref{introdual}
and two small miracles that occur only at integers.
The first is the relation between $ A\mtt{m} $ and $ A(m) $ for any integer $ m $,
so that we may assume
$ \qlatmtt{E}{\chi}{1-m} =\qlat{E}{\chi \op^{m-1}}(1-m)$.
The other is the interpolation formula~\eqref{introinterpol} for negative integers
(and possibly $ 0 $).
It is surely no coincidence that the exponent of $ \op $ is the same in both miracles.

In order to state Theorem~\ref{introtheorem} we now discuss $ p $-adic $ L $-functions (see Section~\ref{pL-functions}
for details).
Let $ k $ be a totally real number field, $ p $ a prime number, $ E $ a finite extension of $ \Q_p $ with valuation ring
$ \Oe $, and $ \chi : G_k \to E $ an even Artin character
realizable over~$ E $.
If $S$ is a finite set of finite primes of $k$ containing the primes above $p$,
then on $ \B_k(E) $ there exists a meromorphic $ E $-valued $p$-adic $L$-function $L_{p,S}(s,\chi,k)$, defined at all negative
integers~$ m $, where it satisfies the interpolation formula
\begin{equation} \label{introinterpol}
 L_{p,S} (m,\chi, k) = L_S^\ast(m,\chi\op^{m-1},k)
\,.
\end{equation}
The same statement sometimes also holds when $ m=0 $.
The usual $ p $-adic $ L $-function $ L_p (m,\chi, k) $ is
obtained when $ S $ consists only of the finite primes of $ k $
lying above~$ p $.

A consequence of \cite[Proposition 5]{greenberg83} is
that the main conjecture of Iwasawa theory for~$ p $
(i.e., the statements in Theorem~\ref{mainc} and Remark~\ref{notSremark}
for $p$) implies that $ L_{p,S}(s,\chi,k) $ is analytic on $ \B_k(E) $ if $\chi$ does not contain the trivial character,
and meromorphic on $\B_k(E)$ with at most a pole at $s=1$ if it does.
If $ p \ne 2 $ we could use this by \cite[Theorems~1.2 and~1.3]{Wiles90},
but our results imply this consequence of Greenberg's stronger result
when $ S $ also contains the finite primes of $ k $ at which $ \chi $
is ramified, and we remove this restriction by proving a lower bound
for the corank of $ H^1(\O_{k,S},\qlatmtt{E}{\chi}{1}) $
in Proposition~\ref{OPprop}.
In fact, working directly with Theorem~\ref{mainc} if $ p\ne2 $ or Assumption~\ref{2-mainc}
if $ p=2 $, i.e., working with a weaker version of the main conjecture of Iwasawa theory,
we get a uniform approach for all~$ p $.

A zero of $ L_{p,S}(s,\chi,k) $ always lies in $ \B_k(E') $ for some finite extension $ E' $
of $ E $ (see Remark~\ref{zero_remark}), so will be interpreted by Theorem~\ref{introtheorem}.
Note that $ \B_\chi(E) $ is the expected domain for $ L_p(s,\chi,k) $
and $ L_{p,S}(s,\chi,k) $, as the Leopoldt conjecture implies that $ L_{p,S}(s,\chi,k) $
should not be defined at $ 1 $ if $ \chi $ contains the trivial character.
By Theorem~\ref{cogentheorem}(3) it is also the set for which $H^0(\O_{k,S},\qlatmtt{E}{\chi}{1-e}) $ is finite,
and where Theorem~\ref{introtheorem}(3) gives
a precise correspondence between non-vanishing of the $ p $-adic $ L $-function and the finiteness of
the cohomology groups.

We can now state our second main result.

\begin{theorem} \label{introtheorem}
If in Theorem~\ref{cogentheorem} $k$ is totally real and $\eta = \chi$ is even, then
$H^i(\O_{k,S},\qlatmtt{E}{\chi}{1-e})$ for $i \ge 0$ is finite for all but finitely many $ e $ in $ \B_k(E) $.
Moreover, the following hold, where for $p = 2$ we make Assumption~\ref{2-mainc}.
\begin{enumerate}
\item
Let $ \sum_{i=0}^2 (-1)^i i \cdot r_i = r_2 - r_0 $ be the secondary Euler characteristic
and let $ \nu=\nu_S(1-e,\chi) = \ord_{s=e} L_{p,S}(e,\chi,k) $.  Then 
we have that
$ \min(1-r_0,\nu) \le r_2-r_0 \le \nu $.

\item
$L_{p,S}(s,\chi,k)$ and $L_p(s,\chi,k)$ are
meromorphic on $ \B_k(E) $; the only possible pole is at $s=1$, with
order at most the multiplicity of the trivial character in $\chi$.

\item
For $ e $ in $ \B_\chi(E) $ the following are equivalent:
\begin{enumerate}
\item
$L_{p,S}(e,\chi , k) \neq 0$;

\item
$H^1(\O_{k,S},\qlatmtt{E}{\chi}{1-e}) $ is finite;

\item
$H^2(\O_{k,S},\qlatmtt{E}{\chi}{1-e}) $ is finite;

\item
$H^2(\O_{k,S},\qlatmtt{E}{\chi}{1-e}) $ is trivial.

\end{enumerate}
If this is the case, then
\begin{equation*}
|L_{p,S}(e,\chi , k)|_p = 
\left( \frac{ \# H^0(\O_{k,S},\qlatmtt{E}{\chi}{1-e} ) }{ \# H^1(\O_{k,S},\qlatmtt{E}{\chi}{1-e} ) }
\right)^{1/[E:\Q_p]}
\,. 
\end{equation*}
\end{enumerate} 
\end{theorem}

The assumption for $ p=2 $ in Theorem~\ref{introtheorem} is fulfilled 
if $ k=\Q $ and $ \chi $ is a multiple of the trivial character
(see Remark~\ref{itholds}).
We also remark that part of Theorem~\ref{introtheorem}
(for $ p $ odd, $ e \ne 1 $ an odd integer, and $ \chi = \op^{1-e} $)
is contained in \cite[Theorem~6.1]{Ba-Ne}, 
under an assumption in Iwasawa theory since proved by Wiles.
A similar partial result
(for $ p $ odd, $ e $ a negative odd integer, and $ \chi $ an even Artin character)
is outlined in the proof of \cite[Proposition~6.15]{ckps}.
The case $ p=2 $ is not discussed in either paper.

Theorem~\ref{introtheorem} would be even more complete if the upper bounds
in 
part~(1) and part~(2)
were sharp.
The equality $ \nu_S(0,\chi) = - \cork{0,S}{0}{\chi} $ in part~(2)
for all $ k $
is equivalent with the Leopoldt conjecture for all $ k $
(see Remark~\ref{sisone}).
Equivalently, we always have
$ \min(\cork{0,S}{0}{\chi},\cork{2,S}{0}{\chi}) = 0 $.
The other equality, in part~(1), we formulate as a conjecture,
which is itself implied by some folklore conjectures in Iwasawa
theory (see Conjecture~\ref{conjecture}).

\begin{conjecture} \label{introconjecture}
We have 
$ \nu_S(1-e,\chi) = \cork{2,S}{1-e}{\chi} - \cork{0,S}{1-e}{\chi} $
in Theorem~\ref{introtheorem}(1).
\end{conjecture}

In this paper we give two applications of the $ p $-adic statements in Theorem~\ref{introtheorem}
to classical $ L $-functions.
One of them is contained in Section~\ref{etnc_section}.  It consists of a short proof of the
equivariant Tamagawa number conjecture at negative integers
for Artin motives of the right parity over a totally real number field $k$
with coefficients in a maximal order.
(In this case the conjecture is equivalent with the Bloch-Kato
conjecture~\cite[\S5]{Blo-Kat90},
and the corresponding result for a Dirichlet motive over $ \Q $ with
$ p $ odd and $ p^2 $ not dividing the conductor of the underlying
Dirichlet character, was already stated in \S10.1(b) of ~\cite{Fon92}.)
For Dirichlet motives over $\Q$ with $p$ odd,
Huber and Kings proved this conjecture at all integers
\cite[Theorem~1.3.1]{Hu-Ki03}, whereas in the same situation Burns and Greither
proved the stronger statement where
the maximal order is replaced by the group ring \cite[Corollary~8.1]{BuGr}.
(See Remark~\ref{finalremark}(2) for those and other statements.)

The other application we give here, namely a proof of a generalization of the conjecture of Coates and Lichtenbaum
discussed above.  Let $m$ be a negative integer, $ \eta : G_k \to E $ an Artin character that is
realizable over $ E $ with $L_S^\ast(m,\eta^\vee,k) \neq 0$,
where $ E $ is a finite extension of $ \Q_p $.  Then $k$ is totally real and $\eta\op^{m-1}$ is even.
As in the paragraph containing~\eqref{CL-statement}, let $ S $ contain the finite primes of $ k $ at which $ \eta $
is ramified.  Since the conjecture is independent of the choice of $S$ by (the proof of)
\cite[Proposition~3.4]{Co-Li}, we may assume
$ S $ also contains the finite primes of $ k $ lying above $ p $, hence all finite primes
of $ k $ at which the even Artin character $ \chi = \eta^\vee\op^{1-m} $ is ramified.
We see from ~\eqref{introinterpol} that
$ | L_{p,S}(m,\chi,k) |_p = | L_{S}^\ast(m,\eta^\vee,k) |_p \ne 0$, so by Theorem~\ref{introtheorem}(3)
the groups $H^i(\O_{k,S},\qlatmtt{E}{\chi}{1-m})$ are all finite.  By Remark~\ref{introremark},
with $\eta$ replaced with $\eta \op^m$, if $p \ne 2$ then for $ i=0 $, $ 1 $ and $ 2 $,
\begin{alignat*}{1}
 H^i(\O_k,\alpha_! \qlat{E}{\eta} (m))
& =
 H^i(\O_k,\alpha_! \qlatmtt{E}{\eta \op^m}{m})
\\
& \simeq
 H^{2-i}(\O_{k,S},\qlatmtt{E}{\chi}{1-m})^\vee 
\end{alignat*}
with appropriate choice of lattices (see Remark~\ref{duality-remark}),
and for $ i=3 $ the left-hand side is trivial.
Also, the left-hand side is trivial for $i > 3$ by Theorem~II.3.1 and Proposition~II.2.3(d) of~\cite{Mil86}.
By Theorem~\ref{introtheorem}(3) it then follows that
\begin{alignat*}{1}
 |L_S^\ast(m,\eta^\vee,k)|_p ^{[E:\Q_p]}
& =
 |L_{p,S}(m,\chi,k)|_p ^{[E:\Q_p]}
=
 \frac{\# H^0(\O_{k,S}, \qlatmtt{E}{\chi}{1-m} )}{\# H^1(\O_{k,S}, \qlatmtt{E}{\chi}{1-m} )} 
\\
& =
\frac{\#H^2(\O_k,\alpha_!\qlat{E}{\eta}(m))}{\#H^1(\O_k,\alpha_!\qlat{E}{\eta}(m))} 
\,.
\end{alignat*}
Similarly, if
$k$ is totally real, $ \eta^\vee\op $ even, $S$ contains the primes lying above $p$,
$L_S^\ast(0,\eta^\vee,k) \ne 0$,
and~\eqref{introinterpol} is true with $m = 0$ and $ \chi = \eta^\vee\op $,
then those equalities also hold for $m=0$.

If $ p=2 $ then the first two equalities hold for $ m\leq0 $
under the same conditions, where for the second we make
Assumption~\ref{2-mainc}.
The third equality holds for all primes $ p $ 
provided that
we replace the last quotient with
$ \# \H_c^3(\O_{k,S},\lat{E}{\eta}(m)) / \# \H_c^2(\O_{k,S},\lat{E}{\eta}(m)) $
for certain cohomology groups $\H_c^j$ with compact support (see
Section~\ref{etnc_section}).
In fact, for $ m<0 $, the resulting equality 
\begin{equation*}
  |L_S^\ast(m,\eta^\vee,k)|_p^{[E:\Q_p]}
= 
  \frac{\# \H_c^3(\O_{k,S},\lat{E}{\eta}(m))}{\# \H_c^2(\O_{k,S},\lat{E}{\eta}(m))}
\end{equation*}
is part of the equivariant Tamagawa number conjecture, a generalization
of the conjecture of Coates and Lichtenbaum.

The paper is organized as follows.
In Section~\ref{cofinite} we prove Theorem~\ref{cogentheorem} and
discuss how the cohomology groups depend on the choice of
the lattice, the field $E$, and on the finite set of primes~$S$.  If the
cohomology groups are finite then we describe them in terms of cohomology with finite coefficients.
In Section~\ref{pL-functions} we give the definition of the $p$-adic $L$-function
and discuss the interpolation formula relating it to the classical $L$-function.  In Section~\ref{iwasawa-theory}
we discuss the main conjecture of Iwasawa theory, including the case $p = 2$.  
In Section~\ref{EC} we restrict ourselves to $k$ totally real and $\chi$ even and prove
Theorem~\ref{introtheorem} in four steps.  We first prove most of it for a $ 1 $-dimensional 
character of order not divisible by $p$, then for any $ 1 $-dimensional
character, followed by the case of all characters.
The fourth step then strengthens this to Theorem~\ref{introtheorem}
by varying the twist.  In this section we also discuss some conjectures.
In Section~\ref{secexamples} we discuss examples based on computations
by X.-F.~Roblot for certain Galois extensions of $ \Q $ with dihedral group $ G $
of order~8 and cyclic subgroup $ H $ of order~4.
They include the existence of two $ (H\times (1+\Z_5)) \ltimes \Z_5 $-extensions
of $ \Q(\sqrt{41}) $ and two $ (H\times (1+\Z_5)) \ltimes \Z_5[\sqrt5]$-extensions of $\Q(\sqrt{793})$,
which are inside
a $ (G\times (1+\Z_5)) \ltimes \Z_5^2 $-extension of $ \Q $ and
a $ (G\times (1+\Z_5)) \ltimes \Z_5[\sqrt5]^2 $-extension of $ \Q $ respectively;
see Remarks~\ref{funny} and~\ref{funny2}.
Finally, in Section~\ref{etnc_section}, we 
give the first application mentioned after Conjecture~\ref{introconjecture},
i.e., we show how the results from the preceding sections
imply certain cases of the equivariant Tamagawa number conjecture.

It is a pleasure to thank
Denis Benois,
Spencer Bloch, 
David Burns, John Coates, Matthias Flach, Ralph Greenberg,
Steve Lichtenbaum,
Thong Nguyen Quang Do,
Bernadette Perrin-Riou,
Xavier-Fran\c cois Roblot,
Alex\-ander Schmidt,
Michael Spiess,
and Sujatha for helpful discussions and correspondence.
We also wish to thank Xavier-Fran\c cois Roblot for performing the calculations for the examples
in Section~\ref{secexamples}.
Finally, we want to thank the referees for the detailed comments.

\section{Cofinite generation and additive Euler characteristics} \label{cofinite}

In this section we prove Theorem~\ref{cogentheorem}.  In the process, we also discuss some of the properties
of the cohomology groups that we shall use in later sections.  We note that the cohomology groups in the later
sections have coefficients with modified Tate twist indexed by $1-e$
instead of $ e $, hence use the twist $ 1-e $ also in this section
for the sake of consistency.

\begin{notation} \label{notation}

Unless stated otherwise, throughout this section $ k $ denotes a number field, $ p $ a prime number, and
$ E $ a finite extension of $ \Q_p $ with valuation ring $ \Oe $.
For $ \eta : G_k \to E $ an Artin character realizable over $ E $
we let $ \vlat{E}{\eta} $ be an Artin representation of $ G_k $
over $ E $ with character $ \eta $.
Then there exists an $\Oe$-lattice for $\eta$, i.e., a finitely generated torsion-free
$\Oe$-module $\lat{E}{\eta}$ 
with $ G_k $-action such that $\lat{E}{\eta} \otimes_{\Z_p} \Q_p \simeq \vlat{E}{\eta}$ as $ E[G_k]$-modules.
Let $\qlat{E}{\eta} = \lat{E}{\eta} \otimes_{\Z_p} \Q_p/\Z_p $. 

For a number field $F$ we write $F_\infty$ for the cyclotomic $\Z_p$-extension.
Fix a topological generator $\gamma_0$ of $\Gamma_0 = Gal(k_\infty/k)$.
Let $\op : G_k \to \mu_{2p}$ and $\psio : G_k \to 1 + 2p\Z_p$ be as in~\eqref{psi0def}, and
let $q_k = \psio(\tgammanought)$ in $1 + 2p\Z_p$, where $\tgammanought$ is a lift of $\gamma_0$ to $G_k$.
As defined around~\eqref{BEdef}, for any $ \Oe $-module $ A $ on which $G_k$ acts and 
$ e $ in $ \B_k(E) = \{ s \text{ in } E \text{ with } |s|_p < |q_k-1|_p^{-1} p^{-1/(p-1)}\} $,
we let $ A\mtt{e}$ be $ A $ with the action of $ G_k $ multiplied by the character $ (\psio)^e $.

Let $\Sig$ be the set of infinite places of $k$,
and $S$ a finite set of finite primes of $k$ containing the set
$ P $ of primes above $ p $, as well as all finite primes at which $\eta$ is ramified.  
For any subfield $ k' $ of an algebraic closure $ \kbar $ of $ k $ containing $ k $, we denote by
$ \O_{k',S} $ the ring of integers in $ \O_{k'} $ with all primes of $ \O_{k'} $ lying above primes
in $ S $ inverted.  For $ e $ in $ \B_k(E) $ we view
$ \qlatmtt{E}{\eta}{e}$ as a sheaf for the \'etale topology of $ \O_{k,S} $.
As in Remark~\ref{introremark}, we
let $\Omega_S$ denote the maximal extension of $k$ that is
unramified outside of $S \cup \Sig$ and let $G_S = \Gal(\Omega_S/k)$.

Finally, for any finite or infinite prime $ v $ of $ k $, we
let $ k_v $ denote the completion of $ k $ at $ v $.
\end{notation}

\begin{remark}\label{coeffremark}
(1)
If $ \eta $ is 1-dimensional then any $ \lat{E}{\eta} $ is isomorphic
as $ \Oe[G_k] $-module to $ \Oe $ with $ g $ in $ G_k $ acting as multiplication by $ \eta(g) $.
Hence in this case $ \qlat{E}{\eta} $ is always isomorphic to
$ \Oe \otimes_{\Z_p} \Q_p/\Z_p \simeq E/\Oe $ with this action of~$ G_k $.

(2)
Let $ K/k $ be a Galois extension with $ G = \Gal(K/k) \simeq \Z/p\Z $, $ M = \Oe[G] $,
$ M_1 = \Oe \sum_{g \in G} g \subset M $,
and $ M_2 = M/M_1 $, so that $ M_1 \oplus M_2 $ and $ M $ are
lattices for the same Artin character of $ G_k $.  If $ K \cap k_\infty =  k $
then $ W\mtt{1-e}^{G_k} $ is isomorphic with $ \Oe/(q_k^{1-e}-1) $
if $ e \ne 1 $ and $ E/\Oe $ if $ e=1 $,
$ W_1\mtt{1-e}^{G_k} \simeq W\mtt{1-e}^{G_k} $ under the natural
map, and $ W_2\mtt{1-e}^{G_k} \simeq \Oe/(p,q_k^{1-e}-1) $.
This shows the coefficients $ W $ are not necessarily  unique up to isomorphism
for Abelian characters that are not 1-dimensional.

(3)
Assume $ p \ne 2 $, $ E $ contains a primitive $ p $-th root of unity $ \xi_p $, and fix $ e $ in
$ \B_k(E) $.  Let $ K/k $ be a Galois extension with $ G = \Gal(K/k) $
isomorphic to the dihedral group of order $2p$.
Fix $ r $ and $ s $ in $ G $ with orders $ p $ and $ 2 $ respectively.
Consider the two actions of $ G $ on $ M=\Oe^2 $ by letting $ r $ act as either 
$ R_1 = \left({\xi_p \atop 0}{0\atop \xi_p^{-1}}\right)  $
or $ R_2 = \left({0 \atop 1}{-1\atop \xi_p+\xi_p^{-1}}\right)  $,
and $ s $ as $ \left({0\atop1}{1\atop0}\right) $.
Noting that $K \cap k_\infty = k$,
one easily checks that $ W\mtt{1-e}^{G_k} \simeq \Oe/(\xi_p-1) $  
for the first action but is trivial for the second.  
Therefore the coefficients $ W $ are not necessarily  unique up to isomorphism
for non-Abelian characters that are irreducible over $ \Qpbar $.
\end{remark}

\begin{remark}
If $\chi : G_k \to E$ is a 1-dimensional even Artin character of order prime to $p$,
then we shall see in the proof of Theorem~\ref{ebn}
that
$H^1(\O_{k,S},\qlatmtt{E}{\chi}{1-e})$ is isomorphic to the Pontryagin dual of $\oplus_j \Oe/(g_j(q_k^{1-e} - 1))$
for some distinguished polynomials $g_j(T)$ in $\Oe[T]$.
We remark here that in general even $H^0(\O_{k,S},\qlatmtt{E}{\chi}{1-e})$ cannot be isomorphic to the
Pontryagin dual of $ \oplus_j \Oe/(h_j(q_k^e))$ for non-zero polynomials $h_j(T)$ in $\Oe[T]$ because
either $h_j(q_k^e)$ is zero for some $e$ in
$\B_k(E)$ or $|h_j(q_k^e)|_p$ is a constant function on $\B_k(E)$.
This would contradict the example $ W_2\mtt{1-e}^{G_k} \simeq \Oe/(p, q_k^{1-e}-1) $
in Remark~\ref{coeffremark}(2).
\end{remark}

Recall that in the introduction we made the following definition
to describe the structure of the cohomology groups $H^i(\O_{k,S},\qlatmtt{E}{\eta}{1-e})$ for $i = 0,1,2$,
which, in general are not even finitely generated. 

\begin{definition} 
Let $\O$ be the valuation ring in a finite extension of $\Q_p$.
For an $\O$-module $A$ we denote 
its Pontryagin dual $\Hom_{\Z_p}(A,\Q_p/\Z_p)$ by $A^\vee$. Then $\O$ acts on $ A^\vee $ via its action on $A$.
We say that $A$ is cofinitely generated if $A^\vee$ is
a finitely generated $\O$-module, and in this case we denote by $\corank_{\O} A$ the $\O$-rank of $A^\vee$.
\end{definition}

\begin{remark} \label{exactness}
Note that $\Hom_{\Z_p}(\cdot, \Q_p/\Z_p)$ is an exact functor
on $ \Z_p $-modules, hence submodules and quotients of cofinitely generated modules
are cofinitely generated.
In particular, if $A$ is cofinitely generated then so is $H^0(G_S,A)$.
Moreover, if $A$ is finitely generated or cofinitely generated,
then by Pontryagin duality \cite[Theorem~1.1.11]{NSW2nded}
the natural inclusion of $ A $ into $ (A^\vee)^\vee $ is an isomorphism
because the $ \Q_p/\Z_p $-dual of its cokernel is trivial, hence
the cokernel is trivial.
\end{remark}

\begin{remark} \label{duality-remark}
Let $ e $ in $ \B_k(E) $.  We shall mostly use the \'etale cohomology groups
$ H^i(\O_{k,S}, \qlatmtt{E}{\eta}{1-e}) $, but they are
isomorphic to various other groups as we now discuss.

For $ m $ in $ \Z $, we may assume $ \vlat{E}{\eta} = \vlat{E}{\eta\op^m} $
as $E$-vectorspaces but with different $G_k$-actions.
Then choosing $ \lat{E}{\eta\op^m} = \lat{E}{\eta} $ we 
get $ \latmtt{E}{\eta}{e}(m) = \latmtt{E}{\eta\op^m}{e+m} $
and similarly for $ W $ instead of $M$.
We let $\lat{E}{\eta^\vee} = \Hom_{\Z_p}(\lat{E}{\eta},\Z_p)$ on which $G_k$
acts via the inverse of its action on $\lat{E}{\eta}$, so that
on $ \lat{E}{\eta^\vee} \otimes_{\Z_p} \Q_p $ this gives an Artin representation
of $ G_k $ with character $ \eta^\vee $.  We then have
$\qlat{E}{\eta^\vee} = \lat{E}{\eta^\vee} \otimes_{\Z_p} \Q_p/\Z_p \simeq \Hom_{\Z_p}(\lat{E}{\eta},\Q_p/\Z_p)$.

By \cite[Proposition~II.2.9]{Mil86} we have an isomorphism between
$H^i(\O_{k,S},X) $ and the continuous group cohomology group
$ H^i(G_S,X)$ for any finite $\Z_p$-module $X$ with continuous $G_S$-action.
This isomorphism is natural, so taking direct limits gives
$H^i(\O_{k,S},\qlatmtt{E}{\eta}{1-e}) \simeq H^i(G_S,\qlatmtt{E}{\eta}{1-e})$.
By \cite[Theorem~8.3.20(i)]{NSW2nded}
the $H^i(G_S,X)$ are finite for any such
$X$ and $i \ge 0$, so by \cite[Corollary, p.261]{Tate76} we have
$H^i(G_S, \latmtt{E}{\eta}{1-e}) \simeq \varprojlim_n H^i(G_S, \latmtt{E}{\eta}{1-e}/p^n)$
for $i \ge 1$, and for $i = 0$ this is obvious.
Similarly
$\varprojlim_n^1 H^1(\O_{k,S},\latmtt{E}{\eta}{1-e}/p^n) \simeq \varprojlim_n^1 H^1(G_S,\latmtt{E}{\eta}{1-e}/p^n)$
is trivial, so if $H_{cts}^i$ denotes Jannsen's continuous \'etale cohomology \cite{Jan88}, then
it follows from (3.1) in loc.\ cit., that for all $i$,
\begin{alignat*}{1}
 H_{cts}^i(\O_{k,S}, \latmtt{E}{\eta}{1-e})
& \simeq
 \varprojlim_n H^i(\O_{k,S}, \latmtt{E}{\eta}{1-e}/p^n)
\\
& \simeq
 H^i(G_S, \latmtt{E}{\eta}{1-e})
\,.
\end{alignat*}

We shall see in Proposition~\ref{cofingen} that
$ H^{3-i}(\O_{k,S},\qlatmtt{E}{\eta^\vee \op}{1-e}) $ is cofinitely
generated.
Let $H_c^i(\O_{k,S},\cdot)$ denote the $i$-th cohomology with compact support as
in \cite[Section~II \S~2, p.203]{Mil86},
so that $H_c^i(\O_{k,S},\cdot) \simeq H_{cts}^i(\O_{k,S}, \cdot) $ when $ p\ne 2 $.
Let $\alpha: \Spec \O_{k,S} \rightarrow \Spec \O_k$ be the natural inclusion.
Then by \cite[II, Proposition~2.3(d) and Corollary~3.3]{Mil86} and \cite[Corollary, p.261]{Tate76},
we have
\begin{equation} \label{alpha_duality}
\begin{aligned}
       & \, \varprojlim_n H_c^i(\O_k,\alpha_! (\latmtt{E}{\eta}{e}/p^n)) \\
\simeq & \, \varprojlim_n H_c^i(\O_{k,S},\latmtt{E}{\eta}{e}/p^n) \\
\simeq & \, \varprojlim_n H^{3-i}(\O_{k,S},\Hom(\latmtt{E}{\eta}{e}/p^n, \mu_{p^\infty}))^\vee \\
\simeq & \, H^{3-i}(\O_{k,S},\varinjlim_n\Hom(\latmtt{E}{\eta}{e}/p^n, \mu_{p^\infty}))^\vee \\
\simeq & \, H^{3-i}(\O_{k,S},\Hom(\latmtt{E}{\eta}{e}, \mu_{p^\infty}))^\vee \\
\simeq & \, H^{3-i}(\O_{k,S},\qlatmtt{E}{\eta^\vee \op}{1-e})^\vee
\,,
\end{aligned}
\end{equation}
for $e$ in $\B_k(E)$, $i = 0$, $1$, $2$ and $3$.

If 
$ H^j(\O_{k,S},\qlatmtt{E}{\eta}{1-e}) $ is finite for $ j=i,i+1 $,
then it will follow from the proof of Proposition~\ref{cofingen}
that
$ H^i(\O_{k,S},\qlatmtt{E}{\eta}{1-e}) \simeq H_{cts}^{i+1}(\O_{k,S},\latmtt{E}{\eta}{1-e}) $.
So, if $ p \ne 2 $ and all $H^i(\O_{k,S},\qlatmtt{E}{\eta}{e})$ are finite,
then
$H^i(\O_k, \alpha_! \qlatmtt{E}{\eta}{e}) \simeq H^{2-i}(\O_{k,S}, \qlatmtt{E}{\eta^\vee \op}{1-e})^\vee$
for $i = 0,1,2$ and both sides are trivial for $i \ge 3$ (cf. \cite[Theorem~3.2]{Co-Li}).
\end{remark}

\begin{remark} \label{coh_dim}
If $p \neq 2$ then
by \cite[Proposition~8.3.18]{NSW2nded}
we have $H^i(\O_{k,S},X) \simeq H^i(G_S,X) = 0$
for $i \geq 3$ and any finite $G_S$-module $X$.  By taking filtered direct limits it follows that
$H^i(\O_{k,S},\qlatmtt{E}{\eta}{1-e})$
is trivial for $i \geq 3$. 
But for $ p=2 $ and $ i \ge 3 $ we have, by \cite[8.6.10(ii)]{NSW2nded}
and a direct limit argument, that
\begin{alignat*}{1}
 H^i(\O_{k,S},\qlatmtt{E}{\eta}{1-e}) & \simeq H^i(G_S,\qlatmtt{E}{\eta}{1-e})
\\
&  \simeq \bigoplus_{v\in\Sig} H^i(k_v,\qlatmtt{E}{\eta}{1-e})
\,.
\end{alignat*}
\end{remark}

We now start the proof of Theorem~\ref{cogentheorem}.
In the proof of Proposition~\ref{cofingen}
below we shall use results by Tate \cite{Tate62}, \cite{Tate76} and Jannsen \cite{Jan88},
although we shall use \cite{NSW2nded} as a reference instead of \cite{Tate62}.

\begin{lemma} \label{Gcofin}
Let $ \O $ be the valuation ring of a finite extension of $ \Q_p $.
If $ A $ is a (co)finitely generated $ \O $-module on which a
finite group $ G $ acts, then each $ H^i(G,A) $ is (co)finitely generated and the number of (co)generators
needed
can be bounded in terms of $\# G$, $i$ and the number of (co)generators of $A$.
In particular, if $i > 0$ then $ \# H^i(G,A) $ is bounded in terms of the same quantities.
\end{lemma}

\begin{proof}
Since $A$ is (co)finitely generated and $G$ is finite it follows that
each $\Hom_{\O[G]}(\O[G^{i+1}],A)$ is
(co)finitely generated with the number of (co)generators needed
bounded by
a constant depending only on $\#G$, $i$ and the number of (co)generators of $A$.
The lemma now follows immediately from the
definition of $H^i(G,A)$ because $\# G$ annihilates $H^i(G,A)$ if $i > 0$.
\end{proof}

\begin{definition}
For any $a \ne 0 $ in $\Oe$ we let
\begin{equation*}
 \klat{a}{E}{\eta} =  \ker(\qlat{E}{\eta} \stackrel{a}{\to} \qlat{E}{\eta})
\,,
\end{equation*}
so that for $ e $ in $ \B_k(E) $ we have a short exact sequence
\begin{equation}\label{texact}
 0 \to \klatmtt{a}{E}{\eta}{1-e} \to \qlatmtt{E}{\eta}{1-e} \overset{a}{\to} \qlatmtt{E}{\eta}{1-e} \to 0 
\,.
\end{equation}
\end{definition}

\begin{proposition} \label{cofingen}
Let $ k $ be a number field, $ p $ a prime number, $ E $ a finite extension of $ \Q_p $
with valuation ring $ \Oe $,
$ \eta : G_k \to E $ an Artin character, and $ e $ in $ \B_k(E) $.
Assume that $ S $ contains $ P $ as well as all the finite primes of $ k $ at which $ \eta $ is ramified.
Then the following hold.
\begin{enumerate}
\item For $i \ge 0$, $H^i(\O_{k,S},\qlatmtt{E}{\eta}{1-e})$ is a cofinitely generated $\Oe$-module.
\item There is some constant $D = D(S,\eta,k)$ depending only on $S$, $\eta$ and $k$, but not on $e$, $E$, or
the choice of the lattice $\lat{E}{\eta}$, such that each
$H^i(\O_{k,S},\qlatmtt{E}{\eta}{1-e})$ can be cogenerated by at most $D$ elements.
\end{enumerate}
\end{proposition}

\begin{proof}
Write $M$ and $W$ for $\lat{E}{\eta}$ and $\qlat{E}{\eta}$ respectively.
From the long exact sequence associated to $0 \to M \mtt{1-e} \stackrel{p}{\to} M \mtt{1-e} \to M \mtt{1-e}/pM \mtt{1-e} \to 0$ we see that,
for $i \ge 0$, $H^i(G_S,M \mtt{1-e})/p$ is finite since $H^i(G_S,M \mtt{1-e}/pM \mtt{1-e})$ is finite by
\cite[Theorem~8.3.20(i)]{NSW2nded}.
Hence $H_{cts}^i(\O_{k,S},M \mtt{1-e}) \simeq H^i(G_S,M \mtt{1-e})$ is finitely generated for $i \ge 0$ by 
\cite[Corollary, p.260]{Tate76}.  
From~\cite[Theorem~5.14(a)]{Jan88} we have an exact sequence
\begin{alignat*}{1}
\cdots & \to H_{cts}^i(\O_{k,S},M \mtt{1-e}) \to H_{cts}^i(\O_{k,S},M \mtt{1-e}) \otimes_{\Z_p}\Q_p
\\
& \to H^i(\O_{k,S},W \mtt{1-e}) 
\to H_{cts}^{i+1}(\O_{k,S},M \mtt{1-e}) \to \cdots \,,
\end{alignat*}
so it follows that $H^i(\O_{k,S},W \mtt{1-e})$ is cofinitely generated for $i \ge 0$.  This proves part~(1).

Consider the long exact sequence associated to \eqref{texact} with $a=p$, so that
$H^i(\O_{k,S},W \mtt{1-e})[p]$ is a quotient of
$H^i(\O_{k,S},W [p] \mtt{1-e})$ for $i \ge 0$.
Since $p$ kills $W [p] \mtt{1-e}$, there exists
an $ m $ depending only on $ k $ such that,
with $K=k_\eta(\mu_{p^m}) \subset \Omega_S$,
$W [p] \mtt{1-e} \simeq (\Oe/(p))^{\eta(\id_{\kbar})}$
as $\Gal(\Omega_S/K)$-modules for all $e$ in $\B_k(E)$.
Note that this isomorphism is independent of the choice of $M$.
By \cite[Theorem~8.3.20(i)]{NSW2nded},
each $H^j(\Gal(\Omega_S/K), (\Oe/(p))^{\eta(\id_{\kbar})})$
is finite and so the number of its generators can be bounded by a
constant that depends only on $S$, $k$ and $\eta(\id_{\kbar})$ and not on $e$ or $E$.
By Lemma~\ref{Gcofin} the number of generators of $H^i(\Gal(K/k),H^j(\Gal(\Omega_S/K),W [p] \mtt{1-e}))$
can be bounded by a constant that depends only on $S$, $\eta$, $i$, $j$ and $k$.
By the Hochschild-Serre spectral sequence and Remark~\ref{coh_dim} it follows that
the number of generators of $H^i(\O_{k,S},W [p] \mtt{1-e}) \simeq H^i(G_S,W [p] \mtt{1-e})$
can be bounded by a constant that depends only on $S$, $\eta$ and $k$.
This shows that there is a constant $D = D(S,\eta,k)$ depending only on $S$, $\eta$ and $k$ such that each
$H^i(\O_{k,S},W \mtt{1-e})[p]$, and hence each $H^i(\O_{k,S},W \mtt{1-e})$ can be cogenerated by at most $D$ elements.
\end{proof}

This proves part~(1) of Theorem~\ref{cogentheorem}.
We now start the proof of the remaining parts,
obtaining various results for later use along the way.

\begin{remark} \label{Mremark}
(1)
If $M$ and $M'$ are $\Oe$-lattices for $\eta$, then we can find an injection from $ M $ to $ M' $ with
finite cokernel.  Tensoring with $\Q_p/\Z_p$, we get a short exact sequence
\begin{equation}\label{sos}
 0 \to X_E \to M \otimes_{\Z_p} \Q_p/\Z_p \to M' \otimes_{\Z_p} \Q_p/\Z_p \to 0
\end{equation}
of sheaves for the \'etale topology on $\Spec(\O_{k,S})$  with $X_E$ finite.
Each $H^i(\O_{k,S},X_E \mtt{1-e})$ is finite by
\cite[Theorem~8.3.20(i)]{NSW2nded},
so the groups
$H^i(\O_{k,S}, M \otimes_{\Z_p} \Q_p/\Z_p\mtt{1-e})$ and
$H^i(\O_{k,S}, M' \otimes_{\Z_p} \Q_p/\Z_p\mtt{1-e})$ have the
same corank.

(2)
Suppose that $E'/E$ is a finite extension and let $\Oep$ denote the valuation ring of $E'$.  Then
$M' = M \otimes_{\Oe} \Oep$ is an $\Oep$-module of finite $\Oep$-rank on which $G_k$ acts
such that $M' \otimes_{\Oep} E' \simeq V\otimes_E E'$.  
Because $ M' \simeq M^{[E':E]} $ as $ \Oe[G] $-modules, we have,
with $ W = M \otimes_{\Z_p} \Q_p/\Z_p $ and $ W' = M' \otimes_{\Z_p} \Q_p/\Z_p \simeq W^{[E':E]}$,
\begin{equation*}
H^i(\O_{k,S}, W'\mtt{1-e}) \simeq
H^i(\O_{k,S}, W\mtt{1-e}) ^{[E':E]}
\end{equation*}
for all $e$ in $\B_k(E)$.

(3)
Consider two Artin characters $\eta_1, \eta_2 : G_k \to E$ that
are realizable over $E$, and let
$M_j$ for $ j=1,2 $ be corresponding torsion-free $\Oe$-lattices.
Then $M_1 \oplus M_2$ corresponds to the character $\eta_1 + \eta_2$,
and with $W_j = M_j \otimes_{\Z_p} \Q_p/\Z_p$ for $j=1,2$, we
have
\begin{equation*}
 H^i(\O_{k,S}, (W_1 \oplus W_2)\mtt{1-e}) \simeq H^i(\O_{k,S}, W_1\mtt{1-e}) \times H^i(\O_{k,S}, W_2\mtt{1-e})
\,.
\end{equation*}

(4) Let $k'/k$ be a finite extension, $\eta'$ an Artin character of $G_{k'}$ that is realizable over $E$, and
$M'$ an $\Oe$-lattice for $\eta'$.  Then
$M = \Oe[G_k] \otimes_{\Oe[G_{k'}]} M'$ is an
$\Oe$-lattice for the Artin character $\eta=\Ind_{k'}^k (\eta')$ of $G_k$.  If $ S $ is so that
$ \eta $ (and hence $ \eta' $) is unramified outside of $ S $, then
by \cite[Proposition~1.6.4]{NSW2nded}
we have that
$H^i(\O_{k,S}, M \otimes_{\Z_p} \Q_p/\Z_p\mtt{1-e}) \simeq
H^i(\O_{k',S},  M' \otimes_{\Z_p} \Q_p/\Z_p\mtt{1-e})$.

(5)
Suppose that $S'$ is a finite set of primes of $k$ containing $S$.
Taking $\lat{E}{\eta \op^{-1}} = \lat{E}{\eta} $ as $ \Oe $-modules but
with different $ G_k $-action, localization with respect to
$Z = \Spec(\O_{k,S}) \setminus \Spec(\O_{k,S'})$ gives an exact Gysin sequence
(cf.\ the proof of VI Corollary 5.3, VI Remark~5.4(b) of \cite{Mil80}, and 
\cite[Expos\'e~I, Theorem~5.1]{SGA5}),
giving
an isomorphism
$ H^0(\O_{k,S},\qlatmtt{E}{\eta}{1-e})$  $ \simeq H^0(\O_{k,S'},\qlatmtt{E}{\eta}{1-e}) $
and an exact sequence
\begin{equation}\label{gysin}
\begin{aligned}
0 & \to  H^1(\O_{k,S},\qlatmtt{E}{\eta}{1-e}) \to
H^1(\O_{k,S'},\qlatmtt{E}{\eta}{1-e})
\\
& \to H^0(Z,\qlatmtt{E}{\eta \op^{-1}}{-e})
  \to  H^2(\O_{k,S},\qlatmtt{E}{\eta}{1-e})
\\
& \to
H^2(\O_{k,S'},\qlatmtt{E}{\eta}{1-e}) \to H^1(Z,\qlatmtt{E}{\eta \op^{-1}}
{-e}) \to 0
\,.
\end{aligned}
\end{equation}
Note that $ H^3(\O_{k,S},\qlatmtt{E}{\eta}{1-e}) \to H^3(\O_{k,S'},\qlatmtt{E}{\eta}{1-e})$ is an
isomorphism by Remark~\ref{coh_dim}.

We observe that
\begin{equation*}
H^i(Z,\qlatmtt{E}{\eta \op^{-1}}{-e}) \simeq \bigoplus_{v \in S' \setminus S} \, H^i(\Gal(\Fbar_v/\F_v),\qlatmtt{E}{\eta \op^{-1}}{-e})
\,,
\end{equation*}
where $\F_v$ is the residue field at $v$.
Because $\Gal(\Fbar_v/\F_v)$ is topologically generated by the Frobenius $\Fr_v$, we have
\begin{alignat*}{1}
 H^0(\Gal(\Fbar_v/\F_v),\qlatmtt{E}{\eta \op^{-1}}{-e})
& \simeq
 \ker(1-\Fr_v | \qlatmtt{E}{\eta \op^{-1}}{-e})
\\
\intertext{and}
 H^1(\Gal(\Fbar_v/\F_v),\qlatmtt{E}{\eta \op^{-1}}{-e})
& \simeq
 \coker(1-\Fr_v | \qlatmtt{E}{\eta \op^{-1}}{-e})
\,.
\end{alignat*}
But
$\ker(1-\Fr_v | \vlatmtt{E}{\eta \op^{-1}}{-e}) \simeq \ker(1-\mtt{\Nm(v)}^{-e} \Fr_v | \vlat{E}{\eta \op^{-1}})$
because $\eta \op^{-1}$ is unramified at all $v$ in $Z$ and $ \psip(\Fr_v) = \Nm(v) $ in $ \Z_p^\times $,
with $ \psip $ as in~\eqref{psi0def}.  This group is trivial if $e \ne 0$ because all the eigenvalues
of $\Fr_v$ are roots of unity and $\mtt{\Nm(v)}^e$ is a root of unity only when $e = 0$.
Applying the snake lemma to $1 - \Fr_v$ acting on the short exact sequence
\begin{equation*}
 0
\to
 \latmtt{E}{\eta \op^{-1}}{-e}
\to
 \vlatmtt{E}{\eta \op^{-1}}{-e}
\to
 \qlatmtt{E}{\eta \op^{-1}}{-e}
\to
 0
\,,
\end{equation*}
we find that $\corank_{\Oe} H^i(\Gal(\Fbar_v/\F_v),\qlatmtt{E}{\eta \op^{-1}}{-e})$, for $i=0,1$,
is trivial when $e \ne 0$, and equals $\dim_E \vlat{E}{\eta \op^{-1}}^{\Fr_v = 1}$ otherwise.
If this corank is trivial then $1-\Fr_v$ is an isomorphism on $\vlatmtt{E}{\eta \op^{-1}}{-e} $,
so $ H^1(\Gal(\Fbar_v/\F_v),\qlatmtt{E}{\eta \op^{-1}}{-e}) $ is trivial and
\begin{equation} \label{h0eul}
\begin{aligned}
& \# H^0(\Gal(\Fbar_v/\F_v),\qlatmtt{E}{\eta \op^{-1}}{-e})
\\
= & \,
 |\det(1- \mtt{\Nm(v)}^{-e} \Fr_v |  V(E,\eta \op^{-1}))|_p^{-[E:\Q_p]}
\,.
\end{aligned}
\end{equation}
\end{remark}

In Section~\ref{EC} we shall use the following proposition for Conjecture~\ref{conjecture}.
Here we let $\cork{i,S}{1-e}{\eta} = \corank_{\Oe} H^i(\O_{k,S},\qlatmtt{E}{\eta}{1-e})$,
just as in Theorem~\ref{introtheorem}.

\begin{proposition} \label{gysinprop}
Let $k$ be a number field, $\eta$ an Artin character, $S$ and $S'$ finite set of primes of $k$ such that
$S \subseteq S'$, and $e$ in $\B_k(E)$.
Then
\begin{enumerate}
\item
the map $ H^2(\O_{k,S},\qlatmtt{E}{\eta}{1-e}) \to  H^2(\O_{k,S'},\qlatmtt{E}{\eta}{1-e}) $
in~\eqref{gysin} has finite kernel;

\item
$\cork{1,S'}{1-e}{\eta} - \cork{1,S}{1-e}{\eta}=\cork{2,S'}{1-e}{\eta} - \cork{2,S}{1-e}{\eta}$,
which equals
$\sum_{v \in S' \setminus S} \dim_E \vlat{E}{\eta \op^{-1}}^{\Fr_v=1}$ if
$ e=0 $, and $ 0 $ if $ e\ne0 $.

\end{enumerate}
\end{proposition}

\begin{proof}
From Remark~\ref{Mremark}(5) we see this holds for $e \ne 0$, and that for $e = 0$ it is sufficient to prove part~(1).
We may assume $ \qlatmtt{E}{\eta}{1} = \qlat{E}{\eta \op^{-1}}(1)$, therefore
we only need to show that the kernel of the map
$ H^2(\O_{k,S},\qlat{E}{\eta \op^{-1}}(1)) \to H^2(G_k,\qlat{E}{\eta \op^{-1}}(1)) $
is finite.
We let $ K $ be a Galois extension of $ k $ such that the restriction
of ${\eta\op^{-1}}$ to $G_K$ is a multiple of the trivial character,
and let $G = \Gal(K/k)$.
Because the $ H^i(\O_{K,S},\qlat{E}{\eta \op^{-1}}(1)) $
are cofinitely generated by Theorem~\ref{cogentheorem}(1), we see from Lemma~\ref{Gcofin} and
the spectral sequence
$$ H^p(G,H^q(\O_{K,S},\qlat{E}{\eta \op^{-1}}(1))) \Longrightarrow H^{p+q}(\O_{k,S},\qlat{E}{\eta\op^{-1}}(1)) $$
that it suffices to show that the map
\begin{equation*}
 H^2(\O_{K,S},\qlat{E}{\eta \op^{-1}}(1))^G \to H^2(G_K,\qlat{E}{\eta \op^{-1}}(1))^G 
\end{equation*}
has finite kernel.  By our choice of $K$ and Remark~\ref{Mremark}(1)
this follows if we show that the kernel of $ H^2(\O_{K,S},\Q_p/\Z_p(1)) \to H^2(G_K,\Q_p/\Z_p(1) ) $
is finite.

Consider the boundary map $H^1(G_K,\Q_p/\Z_p(1)) \to H^0(Z,\Q_p/\Z_p)$ where $Z = \Spec(\O_K) \setminus S$
(cf.~\eqref{gysin}).
This is obtained by taking the direct limit of the maps $H^1(G_K,\Z/p^n\Z(1)) \to H^0(Z,\Z/p^n\Z)$.
The size of the cokernel of this last map is bounded by the class number of $K$.
Taking the direct limit over $n$, it follows that $H^1(G_K,\Q_p/\Z_p(1)) \to H^0(Z,\Q_p/\Z_p)$
has finite cokernel, so $ H^2(\O_{K,S},\Q_p/\Z_p(1)) \to H^2(G_K,\Q_p/\Z_p(1) ) $ has finite kernel.
\end{proof}

We also prove the following result, which provides a lower bound for the corank of
$H^1(\O_{k,S},\qlatmtt{E}{\eta}{1})$.

\begin{proposition} \label{OPprop}
Let $ k $ be a number field, $ p $ a prime number, $ E $ a finite extension of $ \Q_p $,
$ \eta : G_k \to E $ an Artin character realizable over~$ E $,
$ \lat{E}{\eta} $ an $\Oe$-lattice for $\eta$, 
and
$ \qlat{E}{\eta} = \lat{E}{\eta} \otimes_{\Z_p} \Q_p/\Z_p $.
Write $\alpha: \Spec k \to \Spec \O_{k,P}$
for the natural map.  Then
$ H^i(\O_{P}, \alpha_* \qlatmtt{E}{\eta}{1}) $ is cofinitely generated
for $i \ge 0$,
and its corank is independent of the choice of $ \lat{E}{\eta} $.
Moreover,
\begin{equation} \label{h1lowerbound}
\begin{aligned}
& \, \corank_{\Oe} H^1(\O_{k,S}, \qlatmtt{E}{\eta}{1})
\\
= & \,
 \corank_{\Oe} H^1(\O_{P}, \alpha_* \qlatmtt{E}{\eta}{1})
+
\sum_{v \in S \setminus P} \dim_E \vlat{E}{\eta \op^{-1}}^{G_{w_v}}
\,,
\end{aligned}
\end{equation}
where $ w_v $ is a prime of $ \kbar $ above $ v $ with decomposition
group $ G_{w_v} $.
\end{proposition}

\begin{proof}
Write $W$ for $\qlatmtt{E}{\eta}{1}$ and $Z $ for
$ \Spec \O_{k,P} \setminus \Spec \O_{k,S}$.  Let $\beta : \Spec \O_{k,S} \to \Spec \O_{k,P}$ be
the natural map. For $ i\ge0 $, by \cite[II. Lemma~2.4]{Mil86} we have an exact sequence
\begin{equation} \label{smin}
\begin{aligned}
\cdots & \to
 \oplus_{v \in Z} H^{i-1}(\Gal(\overline{k}_v/k_v), W) 
\to
 H^i(\O_{k,P}, \beta_! W)
\\
&
\to
 H^i(\O_{k,S}, W)
\to
 \oplus_{v \in Z} H^i(\Gal(\overline{k}_v/k_v), W) 
\to\cdots
\,.
\end{aligned}
\end{equation}
For any finite prime $v$ of $k$ and a finite Galois extension $K$ of $k_v$
such that $\Gal(\kbar_v/K)$ acts trivially on $\qlat{E}{\eta \op^{-1}}$,
\cite[Theorem~7.1.8(i) and Proposition~7.3.10]{NSW2nded}
imply that $H^i(\Gal(\kbar_v/K),W)$ is cofinitely
generated, and using Lemma~\ref{Gcofin} one sees that the same
holds for $  \oplus_{v \in Z} H^i(\Gal(\overline{k}_v/k_v), W) $.
From Proposition~\ref{cofingen}(1) and~\eqref{smin}
it then follows that $H^i(\O_{k,P},\beta_!W)$ is cofinitely generated.
Since $H^i(\Gal(\kbar_v/k_v),X)$ is finite for finite $X$ by
\cite[Theorem~7.1.8(iii)]{NSW2nded},
it also follows from the long exact
sequence associated to~\eqref{sos}
that the corank of $H^i(\Gal(\overline{k}_v/k_v,W)$ does not depend on the choice of $\lat{E}{\eta}$.
Hence the same holds for the corank of $H^i(\O_{k,P},\beta_!W)$ by Remark~\ref{Mremark}(1) and~\eqref{smin}.

On the other hand, with $\gamma: Z \to \Spec \O_{k,P}$ the natural map,
by \cite[p.76]{Mil80} we have a short exact sequence
\begin{equation} \label{opsos}
0 \to \beta_! \beta^* \alpha_* W \to \alpha_* W \to \gamma_*\gamma^*\alpha_* W \to 0
\end{equation}
of sheaves for the \'etale topology on $\O_{k,P}$,
which gives a long exact sequence of cohomology groups.  We note that
\begin{equation} \label{gammalabel}
 H^i(\O_{k,P}, \gamma_*\gamma^*\alpha_* W)
\simeq
 \oplus_{v \in Z} H^i(\Gal(\overline{\F}_v/\F_v),W^{I_{w_v}})
\,,
\end{equation}
where $I_{w_v}$ denotes the decomposition group of a prime $ w_v $ in $\kbar$ lying above $v$.
These groups are cofinitely generated and using~\eqref{sos} their coranks are independent of the choice
of $\lat{E}{\eta}$ because $H^i(\Gal(\overline{\F}_v/\F_v),X)$ is finite if $X$ is
finite.  Since $H^i(\O_{k,P},\beta_!\beta^* \alpha_* W) = H^i(\O_{k,P},\beta_! W)$ it follows from
the long exact sequence associated to~\eqref{opsos} that
$H^i(\O_{k,P},\alpha_* W)$ is cofinitely generated and that its corank is independent of
the choice of $\lat{E}{\eta}$.
Moreover, $H^i(\O_{k,P},\alpha_* W)$ behaves well under extending
$ E $, as well as under induction and addition
of characters.  Combining Brauer induction with parts~(3) and~(4) of Remark~\ref{Mremark},
one sees that
it is enough to prove~\eqref{h1lowerbound} for $1$-dimensional characters $\eta$.
With this $ S $, if $\eta$ is unramified at a prime $v$ of $k$, then
$\vlat{E}{\eta \op^{-1}}^{G_{w_v}} = \vlat{E}{\eta \op^{-1}}^{\Fr_v = 1}$, so
by Proposition~\ref{gysinprop}, it is enough to prove the result for $S = S_{min}$, the
smallest set of primes containing the finite primes of $k$ lying above $p$ and the primes at
which $\eta$ is ramified.
In this case, if $v$ is in $Z$ then $\vlat{E}{\eta \op^{-1}}^{G_{w_v}}$ is
trivial since $\eta \op^{-1}$ is ramified at $v$.  Therefore,
we need to show that $\corank_{\Oe} H^1(\O_{k,S_{min}},W) = \corank_{\Oe} H^1(\O_{k,P}, \alpha_* W)$
for $1$-dimensional $\eta$.

The first term of~\eqref{smin}, with $i = 1$ and $S=S_{min}$, is finite since $\Gal(\overline{k}_v/k_v)$
acts non-trivially on $W$.  On the other hand, for any finite Abelian extension $K$ of $k_v$, the coranks of
$H^1(\Gal(\overline{K}/K),E/\Oe(1))$ and $\oplus_{\psi} \, H^1(\Gal(\kbar_v/k_v),\qlat{E}{\psi}(1))$ are
the same, where the sum runs over all characters $\psi : \Gal(\kbar_v/k_v) \to E$ that are irreducible
over $E$ and factor through $\Gal(K/k_v)$.  Since
\begin{equation*}
H^1(\Gal(\overline{K}/K),E/\Oe(1)) \simeq E/\Oe \simeq H^1(\Gal(\kbar_v/k),E/\Oe(1)) 
\end{equation*}
by \cite[Theorem 7.3.10(ii)]{NSW2nded},
$W \simeq \qlat{E}{\eta \op^{-1}}(1)$ by Remark~\ref{coeffremark}(1) and
$\eta \op^{-1}$ is non-trivial on $\Gal(\kbar_v/k_v)$ for $v$ in $Z$,
we see that
$H^1(\Gal(\kbar_v/k_v),W)$ is finite for all such $v$.  By~\eqref{smin}, with $i=1$ and $S=S_{min}$, it follows that
the coranks of $ H^1(\O_{k,P}, \beta_! W)$ and $H^1(\O_{k,S_{min}}, W)$ are equal.  Moreover, in this case,
the groups $H^i(\Gal(\overline{\F}_v/\F_v),W^{I_{w_v}})$ in~\eqref{gammalabel} are also finite since
$W^{I_{w_v}}$ is finite.
Therefore, from the long exact sequence associated to~\eqref{opsos} it follows that the coranks of
$ H^1(\O_{k,P}, \alpha_* W)$ and $H^1(\O_{k,P}, \beta_! W)$ are the
same.  This completes the proof.
\end{proof}

We need the following lemmas in order to prove Theorem~\ref{cogentheorem}(2), and also to show that
Definition~\ref{euler_char} is independent of the choice of lattice.  Here,
for a $G_k$-module $A$ and a place $v$ in $\Sig$, we let $A_{+,v} = H^0(k_v,A)$. 

\begin{lemma} \label{eul_ind}
Let $ A $ be a finite Abelian group with continuous $ G_S $-action
and $ \# A $ in $ \O_{k,S}^\times $.  Then
\begin{equation} \label{euler-poincare}
\frac{\# H^0(G_S,A) \cdot \# H^2(G_S,A)}{\# H^1(G_S,A)}
=
(\#A)^{-[k:\Q]} \cdot \prod_{v \in \Sig} (\# A_{+,v})
\,.
\end{equation}
For non-trivial $A$, this last quantity equals 1 if and only if
$k$ is totally real and every complex conjugation in $G_k$ acts trivially on $A$.
\end{lemma}

\begin{proof}
By \cite[Theorem~8.3.20(i)]{NSW2nded} the groups $ H^i(G_S,A) $ for $ i \ge 0 $ are finite,
and by the Global Euler-Poincar\'e Characteristic formula
\cite[8.7.4]{NSW2nded}
we have
\begin{equation*}
\frac{\# H^0(G_S,A) \cdot \# H^2(G_S,A)}{\# H^1(G_S,A)}
=
 \prod_{v \in \Sig} \frac{\# H^0 (k_v, A)}{||\# A||_v}
\,,
\end{equation*}
where $||n||_v = n$ if $v$ is real and $||n||_v = n^2$ if $v$ is complex.
This proves the first claim, and the second claim follows immediately.
\end{proof}

\begin{lemma}\label{9lemma}
Let $\O$ be the valuation ring of a finite extension of $\Q_p$.
Suppose that $  0 \to A_1 \to A_2 \to A_3 \to 0 $ is a short exact sequence of cofinitely generated
$\O$-modules with continuous $ G_S $-action, with $A_1$ finite.
Then
\begin{enumerate}
\item
the size of the image of the connecting map
$\delta: H^2(G_S,A_3) \to H^3(G_S,A_1)$ is bounded by a constant depending only
on the number of generators of $A_1$ and the number of real places
of $k$;

\item
$\delta$ is trivial if and only if $(A_2)_{+,v} \to (A_3)_{+,v}$ is surjective for
all $v$ in $\Sig$.  In this case we have a nine-term exact sequence
\begin{alignat*}{1}
 0 &
\to
 H^0(G_S,A_1) \to H^0(G_S,A_2) \to H^0(G_S,A_3)
\to 
\cdots
\\
\cdots &
\to 
 H^2(G_S,A_1) \to H^2(G_S,A_2) \to H^2(G_S,A_3)
\to 
 0
\,.
\end{alignat*}
\end{enumerate}
\end{lemma}

\begin{proof}
By \cite[8.6.10(ii)]{NSW2nded}
we have
$H^3(G_S,A_1) \simeq \prod_{v \in \Sig} H^3(k_v,A_1) \simeq \prod_{v \in \Sig} H^1(k_v,A_1)$,
the size of which is, by Lemma~\ref{Gcofin}, bounded by a constant depending only on the number of generators of $A_1$
and the number of real places of $k$.  This proves part~(1).
For part~(2), we see this way that
$\delta$ is trivial if and only if $H^1(k_v,A_1) \to H^1(k_v,A_2)$ is injective for all $v$ in $\Sig$,
and this is true if and only if $H^0(k_v,A_2) \to H^0(k_v,A_3)$ is surjective for all $v$ in $\Sig$.
\end{proof}

We shall now finish the proof of Theorem~\ref{cogentheorem}
by proving Propositions~\ref{divprop}, \ref{corankprop} and~\ref{h0finiteness}.

\begin{proposition}\label{divprop}
Fix $ e $ in $ \B_k(E) $.  
Then
$H^2(\O_{k,S},\qlatmtt{E}{\eta}{1-e})$ is $p$-divisible
for $ p\ne 2 $.
If $ p=2 $ then
\begin{equation*}
 H^2(\O_{k,S},\qlatmtt{E}{\eta}{1-e}) \simeq 2H^2(\O_{k,S},\qlatmtt{E}{\eta}{1-e}) \oplus 
\!\bigoplus_{v\in\Sig} \!\! \qlat{E}{\eta}_{+,v}/2
\,,
\end{equation*}
and $2H^2(\O_{k,S}, \qlatmtt{E}{\eta}{1-e})$ is $2$-divisible.
\end{proposition}

\begin{proof}
If $p \ne 2$ then this follows from Lemma~\ref{9lemma}(2) applied to~\eqref{texact} with $a=p$,
so let us take $p=2$.
By the structure of cofinitely generated $ \Oe $-modules we have
$H^2(\O_{k,S},\qlatmtt{E}{\eta}{1-e}) $ ${}\simeq A \oplus B$ with $A$ $2$-divisible and $B$ finite.
Take $n \ge 1$ such that $H^2(\O_{k,S},\qlatmtt{E}{\eta}{1-e})/2^n \simeq B$.
Note that
$H^3(\O_{k,S},\klatmtt{2^n}{E}{\eta}{1-e}) \simeq \oplus_{v \in \Sig} H^3(k_v, \klatmtt{2^n}{E}{\eta}{1-e})$
is killed by multiplication by $2$, so from the long exact sequence associated to \eqref{texact} with $a = 2^n$
we see $2B = 0$, hence $A = 2H^2(\O_{k,S},\qlatmtt{E}{\eta}{1-e})$.  On the other hand,
\begin{alignat*}{1}
B & \simeq H^2(\O_{k,S},\qlatmtt{E}{\eta}{1-e})/2 \\
& \simeq \ker \left( H^3(\O_{k,S},\klatmtt{2}{E}{\eta}{1-e}) \to H^3(\O_{k,S},\qlatmtt{E}{\eta}{1-e}) \right) \\
& \simeq \oplus_{v \in \Sig} \ker \left( H^3(k_v,\klatmtt{2}{E}{\eta}{1-e}) \to H^3(k_v,\qlatmtt{E}{\eta}{1-e}) \right) \\
& \simeq \oplus_{v \in \Sig} \ker \left( H^1(k_v,\klatmtt{2}{E}{\eta}{1-e}) \to H^1(k_v,\qlatmtt{E}{\eta}{1-e}) \right) \\
& \simeq \oplus_{v \in \Sig} \qlatmtt{E}{\eta}{1-e}_{+,v}/2 \simeq \oplus_{v \in \Sig} \qlat{E}{\eta}_{+,v}/2
\,,
\end{alignat*}
where we used that $ \qlatmtt{E}{\eta}{1-e}_{+,v} = \qlat{E}{\eta}_{+,v} $
since $ \psio $ is trivial on every complex conjugation.
\end{proof}

\begin{proposition} \label{corankprop}
For $e$ in $\B_k(E)$, with
\begin{equation*}
\cork{i,S}{1-e}{\eta} = \corank_{\Oe} H^i(G_S,\qlatmtt{E}{\eta}{1-e})
\end{equation*}
one has 
\begin{alignat*}{1}
 \sum_{i=0}^2 (-1)^{i} \cork{i,S}{1-e}{\eta}
= \, &
 -[k:\Q] \cdot \corank_{\Oe} \qlat{E}{\eta}
\\
& +
 \sum_{v \in \Sig} \corank_{\Oe} \qlat{E}{\eta}_{+,v}
\,, 
\end{alignat*}
which is 
independent of $ e $ and
non-positive.  It is zero if and only of $k$ is totally real and $\eta$ is even.
Moreover, $\cork{i,S}{1-e}{\eta}$ is independent of $S$ for $i = 0$, and also for $i = 1,2$ if $e \ne 0$.
\end{proposition}

\begin{proof}
Note that an $\Oe$-module $A$ is cofinitely generated as an $\Oe$-module if and only if the same holds
as a $\Z_p$-module, and in this case $ \corank_{\Z_p} A = [E:\Q_p] \cdot \corank_{\Oe} A $.

Write $ W $ for $ \qlatmtt{E}{\eta}{1-e} $, so that $W_{+,v} = \qlat{E}{\eta}_{+,v} $ as
in the proof of Proposition~\ref{divprop}.
For $i=0,1,2$, let $r_i$ be an integer and $X_i$ a finite $\Z_p$-module such
that $H^i(G_S,W) \simeq (\Q_p/\Z_p)^{r_i} \times X_i$ and let
$d = [k:\Q] \cdot \corank_{\Z_p} W - \sum_{v \in \Sig} \corank_{\Z_p} W_{+,v} $.
We need to show $ r_0 - r_1 + r_2 = -d$.

Let $n > 0$ be such that $p^n$ annihilates $X_i$ for $i = 0,1,2$, as well as
the torsion in $(W_{+,v})^\vee$ for every $v$ in $\Sig$.
Since $ (W[p^n])_{+,v} = W_{+,v}[p^n]$,
the right-hand side of~\eqref{euler-poincare} with $A = W[p^n]$ equals $p^{-nd + c}$ for some $c$
independent of $n$.

On the other hand, consider the long exact sequence associated to~\eqref{texact} with $a=p^n$.
We get an isomorphism $H^0(G_S,W[p^n]) \simeq H^0(G_S,W)[p^n]$ and the two short exact sequences
\begin{equation*}
0  \to H^0(G_S,W)/p^n \to H^1(G_S,W[p^n]) \to H^1(G_S,W)[p^n] \to 0
\end{equation*}
and
\begin{equation*}
0 \to H^1(G_S,W)/p^n \to H^2(G_S,W[p^n]) \to H^2(G_S,W)[p^n] \to 0
\,.
\end{equation*}
Therefore the left-hand side of~\eqref{euler-poincare} for $A = W[p^n]$ equals
\begin{alignat*}{1}
& \, \frac{\# H^0(G_S,W[p^n]) \cdot \# H^2(G_S,W[p^n])}{\# H^1(G_S,W[p^n])}
\\
= & \, 
 \frac{\# H^0(G_S,W)[p^n] \cdot \# H^2(G_S,W)[p^n] \cdot \# X_1}{\# H^1(G_S,W)[p^n] \cdot \# X_0}
\\
= & \,
 p^{(r_0 - r_1 + r_2)n + c'}
\end{alignat*}
for some $c'$ independent of $n$.  Hence $r_0 - r_1 + r_2 = -d$ by Lemma~\ref{eul_ind}, proving the displayed
equality. 
Independence of $ e $ is then obvious.
Comparing the coranks of $\qlat{E}{\eta}$ and $\qlat{E}{\eta}_{+,v}$ as $\Oe$-modules it is 
easy to see when the right-hand side is non-positive and when it is zero.
The last statement follows from Remark~\ref{Mremark}(5) and Proposition~\ref{gysinprop}(1).
\end{proof}

We next prove a result on $ H^0(\O_{k,S}, \qlatmtt{E}{\eta}{1-e}) $.
Recall from just before Theorem~\ref{cogentheorem} that
$ \B_\eta(E) $ equals $ \B_k(E) $ if $ \eta $ does not contain the trivial character, and equals
$ \B_k(E) \setminus \{1\} $ if it does.

\begin{proposition} \label{h0finiteness}
Let $ k $ be a number field, $ E $ a finite extension of $ \Q_p $, and
$ \eta : G_k \to E $ an Artin character realizable over $ E $.
Assume that $ S $ contains $ P $ as well as all the primes at which $ \eta $ is ramified.
Then the following hold.
\begin{enumerate}
\item
For $ e $ in $\B_\eta(E)$, $H^0(\O_{k,S},\qlatmtt{E}{\eta}{1-e})$ is finite
and its size is a locally constant function of~$ e $.

\item

The corank of $ H^0(\O_{k,S},\qlat{E}{\eta})$
equals the multiplicity of the trivial character in $\eta$.

\end{enumerate}

\end{proposition}

\begin{proof}
For part~(1), let $ \rho : G_k \to GL(V) $ realize $ \eta $ on a finite dimensional $ E $-vector space~$ V $ and
let $K$ be the fixed field of the kernel of $\rho$.  
Since $M = \lat{E}{\eta}$ is torsion-free, we have a short exact sequence
$  0 \to M \to V \to W \to 0 $ with $ W = \qlat{E}{\eta} $.
Let $ g_1,\dots,g_n $ in $ G_k $ be lifts of the elements in $\Gal(K/k)$, and let $ g_0 $ in $ G_k $
be a lift of a topological generator of $ \Gal(K_\infty/K) \subseteq \Gal(K_\infty/k)$.
Then we have a commutative diagram 
\begin{equation*}
 \xymatrix{
 0 \ar[r] 
&
  M \ar[r] \ar[d]^{\Phi_{M,e}}
&
  V \ar[r] \ar[d]^{\Phi_{V,e}}
&
  W \ar[r] \ar[d]^{\Phi_{W,e}}
& 0
\\
 0 \ar[r] 
&
  M^{n+1} \ar[r]
&
  V^{n+1} \ar[r] 
&
  W^{n+1} \ar[r] 
& 0
}
\end{equation*}
of $ \Oe[G_k] $-modules, 
where $ \Phi_{*,e} $ is given by
\begin{equation*}
 (\id- \psio(g_0)^{1-e}\id) \times (\id - \psio(g_1)^{1-e} \rho(g_1) ) \times \dots \times (\id - \psio(g_n)^{1-e} \rho(g_n) )
\end{equation*}
with $ \psio $ as in~\eqref{psi0def}, and
$ H^0(\O_{k,S}, \qlatmtt{E}{\eta}{1-e}) = \ker(\Phi_{W,e}) $
because the action of $ G_k $ factorizes through its quotient $ \Gal(K_\infty/k) $.
Note that $ \Phi_{V,e} $ is injective: if $ e \ne 1 $ then $ \id- \psio(g_0)^{1-e}\id $ is an automorphism
of $ \vlat{E}{\eta} $, and if $ e=1 $ and  $ \eta $ does not contain the trivial character
then it follows by considering the other components of $ \Phi_{V,1} $.
Hence $ H^0(\O_{k,S}, \qlatmtt{E}{\eta}{1-e}) = \ker(\Phi_{W,e}) $ is isomorphic as $ \Oe $-module
with the torsion in $ \coker(\Phi_{M,e}) $, hence is finite.
If we fix an isomorphism $ \lat{\eta}{E} \simeq \Oe^d $ of $ \Oe $-modules,
where $ d = \dim_E \vlat{E}{\eta} $,
then $ \Phi_{M,e} $ is identified with an $ (nd+d)\times d $-matrix $ A_e $ with entries in $ \Oe $.
The ideal $ I_e $ of $ \Oe $ generated by the determinants of the $ d\times d $-minors of $ A_e $ is
not the zero ideal because $ \Phi_{V,e} $ is injective, hence
$ \#  H^0(\O_{k,S}, \qlatmtt{E}{\eta}{1-e}) = \# \coker(\Phi_{M,e})_\tor = \# \Oe/I_e $.
Clearly, $ I_{e'} = I_e $ for all $ e' $ in $ \B_\eta(E) $ close enough to~$ e $.
This proves part~(1).

For part~(2), write $\eta = s\eta_0 + \eta'$, where $\eta'$ does not contain the trivial character
$ \eta_0 $, and is realizable over $ E $.
Then by Remark~\ref{Mremark}(1) we may choose $\lat{E}{\eta} = \lat{E}{\eta'} \oplus \Oe^s$.
Since $H^0(\O_{k,S},\qlat{E}{\eta'})$ is finite by part~(1) we find
the corank of $ H^0(\O_{k,S}, \qlat{E}{\eta}) $ equals that
of $ H^0(\O_{k,S},(E/\Oe)^s) \simeq (E/\Oe)^s $, which is~$ s $.
\end{proof}

We have now proved all the statements of Theorem~\ref{cogentheorem}:
part~(1) is Proposition~\ref{cofingen}, part~(2) follows from Remark~\ref{Mremark}(1) and
Proposition~\ref{corankprop},
part~(3) is Proposition~\ref{h0finiteness} and part~(4) is Proposition~\ref{divprop}.

To conclude this section, we describe $ H^i(\O_{k,S}, \qlatmtt{E}{\eta}{1-e}) $ for $i=0,1$
when it is finite by means of cohomology groups with finite coefficients, and discuss how
its structure varies
with $e$.  By Theorem~\ref{cogentheorem}(2), for $i = 1$ this can only apply when
$k$ is totally real and $\eta$ is even, when this group
is also finite for $ i=0 $, and, by Proposition~\ref{divprop}, trivial for $ i=2 $.

\begin{proposition} \label{contprop}
Let $ k $ be a number field, $ p $ a prime number, $ E $ a finite extension of $ \Q_p $,
and $ \eta : G_k \to E $ an Artin character realizable over~$ E $.

{\rm(1)}
Assume that $ H^0(\O_{k,S}, \qlatmtt{E}{\eta}{1-e}) $ is annihilated by $ a_0 \ne 0 $ in $ \Oe $.  Then
\begin{enumerate}
\item[(a)]
$ H^0(\O_{k,S}, \klatmtt{a_0}{E}{\eta}{1-e}) \simeq H^0(\O_{k,S}, \qlatmtt{E}{\eta}{1-e}) $
as $ \Oe $-modules under the natural map;

\item[(b)]
if, for $ e' $ in $ \B_\eta(E) $, $ a_0 $ 
also annihilates $ H^0(\O_{k,S},\qlatmtt{E}{\eta}{1-e'}) $,
and $ q_k^e - q_k^{e'} $ is in $ a_0\Oe $, then 
\begin{equation*}
 H^0(\O_{k,S},\qlatmtt{E}{\eta}{1-e'}) \simeq H^0(\O_{k,S},\qlatmtt{E}{\eta}{1-e}) 
\end{equation*}
as $ \Oe $-modules.
\end{enumerate}

{\rm(2)}
Assume that $ H^1(\O_{k,S}, \qlatmtt{E}{\eta}{1-e}) $ is annihilated by $ a_1 \ne 0 $ in $ \Oe $.  Then
\begin{enumerate}
\item[(a)]
$ 
\# H^1(\O_{k,S}, \qlatmtt{E}{\eta}{1-e}) 
 =
\dfrac{\# H^1(\O_{k,S}, \klatmtt{a_1}{E}{\eta}{1-e})}{\# H^0(\O_{k,S}, \klatmtt{a_1}{E}{\eta}{1-e})}
$;

\item[(b)]
if $ H^0(\O_{k,S}, \qlatmtt{E}{\eta}{1-e}) $ is annihilated by $ a_0 \ne 0 $ in $ \Oe $ then
$ H^1(\O_{k,S}, \qlatmtt{E}{\eta}{1-e}) \simeq \coker( f_{a_0,a_1} ) $ as $ \Oe $-modules, with
\begin{equation*}
\qquad\quad
f_{a_0,a_1} : H^0(\O_{k,S}, \klatmtt{a_0}{E}{\eta}{1-e}) \to H^1(\O_{k,S}, \klatmtt{a_1}{E}{\eta}{1-e}) 
\end{equation*} 
the boundary map in the long exact sequence corresponding to the short exact sequence
\begin{equation}\label{sshort}
0 \to \klatmtt{a_1}{E}{\eta}{1-e} \to \klatmtt{a_0a_1}{E}{\eta}{1-e} \overset{a_1}{\to} \klatmtt{a_0}{E}{\eta}{1-e} \to 0
\,,
\end{equation}
and $ f_{a_0,a_1} $ is injective if and only if
$ H^0(\O_{k,S},\qlatmtt{E}{\eta}{1-e}) $ is annihilated by $ a_1 $ as well.

\item[(c)]
If, for $ e' $ in $ \B_\eta(E) $, $ a_1 $ 
also annihilates $ H^1(\O_{k,S},\qlatmtt{E}{\eta}{1-e'}) $,
and $ q_k^e - q_k^{e'} $ is in $ a_0a_1\Oe $, then 
\begin{equation} \label{H1iso}
 H^1(\O_{k,S},\qlatmtt{E}{\eta}{1-e'}) \simeq H^1(\O_{k,S},\qlatmtt{E}{\eta}{1-e}) 
\end{equation}
as $ \Oe $-modules.
\end{enumerate}
\end{proposition}

\begin{proof}
Parts~(1)(a) and~(2)(a) are immediate from the long  exact sequence associated to~\eqref{texact} with $ a $ replaced with $ a_i $.

For~(2)(b), we compare the long exact sequences associated to~\eqref{sshort} and
to~\eqref{texact} with $ a $ replaced with $ a_1 $ via the natural map between them.
This gives the commutative diagram
(with $ H^i(\O_{k,S},\qlatmtt{E}{\eta}{1-e}) $  abbreviated to
$ H^i(W) $ and similarly for the other coefficients)
\begin{equation*}
\xymatrix{
\cdots \ar[r] &
H^0(W[a_0a_1]) \ar[r]^-{a_1}\ar[d]^-{\simeq} &
H^0(W[a_0]) \ar[r]^-{f_{{\scriptstyle a}_0,{\scriptstyle a}_1}}\ar[d]^-{\simeq} &
H^1(W[a_1]) \ar[r]\ar@{=}[d] &
H^1(W[a_0a_1]) \ar[r]^-{a_1}\ar[d] & \cdots
\\
\cdots \ar[r] &
H^0(W) \ar[r]^-{a_1} &
H^0(W) \ar[r] &
H^1(W[a_1]) \ar[r] &
H^1(W) \ar[r]^-{a_1} & \cdots
}
\end{equation*}
where we used (1)(a) for some vertical maps, and the result is immediate as multiplication on $ H^1(W) $
is the zero map by assumption.

Parts~(1)(b) and ~(2)(c) follow from~(1)(a) and~(2)(b) by using that $ \klatmtt{b}{E}{\eta}{1-e'}\simeq \klatmtt{b}{E}{\eta}{1-e}$
as $ \Oe[G_k] $-modules if $ b\ne 0$ lies in $ \Oe $ and $ q_k^e-q_k^{e'} $ lies in $ b\Oe $.
\end{proof}

\begin{remark} \label{directsumremark}
Let notation and assumptions be as in
part (2)(b) of Proposition~\ref{contprop}.
Comparing the long exact sequences associated to~\eqref{texact} for
$ a = a_1 $ and $ a=a_0a_1 $, one sees that 
$ H^1(\O_{k,S}, \klatmtt{a_0a_1}{E}{\eta}{1-e}) $
is the direct sum of the image of the natural map
\begin{equation*}
 H^1(\O_{k,S},\klatmtt{a_1}{E}{\eta}{1-e}) \to H^1(\O_{k,S},\klatmtt{a_0a_1}{E}{\eta}{1-e}) 
\end{equation*}
and that of the (injective) connecting map
$ H^0(\O_{k,S}, \qlatmtt{E}{\eta}{1-e}) \to H^1(\O_{k,S},
\klatmtt{a_0a_1}{E}{\eta}{1-e}) $.
Therefore $ H^1(\O_{k,S}, \qlatmtt{E}{\eta}{1-e}) $ can be obtained
as a direct summand, i.e.,
\begin{alignat*}{1}
 & H^1(\O_{k,S}, \klatmtt{a_0a_1}{E}{\eta}{1-e})
\\
\simeq
 & \,
   H^1(\O_{k,S}, \qlatmtt{E}{\eta}{1-e})
    \oplus
  H^0(\O_{k,S}, \qlatmtt{E}{\eta}{1-e})
\,.
\end{alignat*}
For $ b\ne0 $ in $\Oe $, the natural map
\begin{equation*}
 H^1(\O_{k,S}, \klatmtt{a_0a_1}{E}{\eta}{1-e}) \to H^1(\O_{k,S}, \klatmtt{ba_0a_1}{E}{\eta}{1-e}) 
\end{equation*}
 corresponds to multiplication by $ b $ on $ H^0(\O_{k,S}, \qlatmtt{E}{\eta}{1-e}) $
and the identity map on $ H^1(\O_{k,S}, \qlatmtt{E}{\eta}{1-e}) $.
\end{remark}

\begin{remark} \label{countremark}
Let the notation and assumptions be as in Proposition~\ref{contprop}.

(1)
Let $ a \ne 0 $ be in $ \Oe $ but not in $ \Oe^\times $, so that $ \klat{a}{E}{\eta} \ne 0 $.
Using Theorem~\ref{cogentheorem}(3) and the long exact sequence associated to~\eqref{texact}, one sees that
\begin{equation*}
 H^1(\O_{k,S}, \klatmtt{a}{E}{\eta}{1-e}) \to H^1(\O_{k,S}, \qlatmtt{E}{\eta}{1-e}) 
\end{equation*}
is injective if and only if $ H^0(\O_{k,S}, \qlatmtt{E}{\eta}{1-e}) = 0 $.

(2)
If $ a \ne 0 $ in $ \Oe $ annihilates $ H^i(\O_{k,S}, \qlatmtt{E}{\eta}{1-e}) $ for $ i=0 $ and 1, then
from parts~(1)(a) and~(2)(a) of Proposition~\ref{contprop} we obtain
\begin{equation*}
 \frac{\# H^0(\O_{k,S},\qlatmtt{E}{\eta}{1-e})}{\# H^1(\O_{k,S},\qlatmtt{E}{\eta}{1-e})}
=
 \frac{\# H^0(\O_{k,S},\klatmtt{a}{E}{\eta}{1-e})^2}{\# H^1(\O_{k,S},\klatmtt{a}{E}{\eta}{1-e})}
\,.
\end{equation*}

(3)
The isomorphism in~\eqref{H1iso} is natural, but some
isomorphism must already exist if $ q_k^e-q_k^{e'} $ is only
in the intersection $ a_0\Oe \cap a_1 \Oe $
(as opposed to the product $ a_0a_1 \Oe $).
Namely, from the long exact sequence associated to~\eqref{texact}
we find, for $ a \ne 0 $ in $ \Oe $, 
\begin{equation*}
 \# H^1(\O_{k,S}, \qlatmtt{E}{\eta}{1-e})[a]
= 
 \frac{\# H^1(\O_{k,S}, \klatmtt{a}{E}{\eta}{1-e})}{\# H^0(\O_{k,S}, \klatmtt{a}{E}{\eta}{1-e})}
\end{equation*}
and similarly with $ e $ replaced with $ e' $.
Applying this with $ a $ the powers of a uniformizer of $ \Oe $,
the result follows from the 
classification of finite $ \Oe $-modules as direct sums of cyclic modules.
\end{remark}

\section{$p$-adic $L$-functions}\label{pL-functions}

\begin{notation}
Throughout this section,
$ k $ is a number field, $ p $ a prime number, $ E $ a finite extension of $ \Q_p $ with valuation ring $ \Oe $,
and $ \eta : G_k \to E $ the character of an Artin representation.  We let $ P $ denote
the set of primes of $ k $ lying above~$ p $.
\end{notation}

First assume that $ \eta $ is $ 1 $-dimensional.
Fixing an embedding $\sigma: E \to \C$, it follows from \cite[VII Corollary 9.9]{Neu99}
that $L(m,\sigma \circ \eta, k)$ for an integer $m \le 0$ is in $E$ and that the value
\begin{equation} \label{Lvalue}
L^\ast(m,\eta,k) := \sigma^{-1}(L(m,\sigma \circ \eta,k))
\end{equation}
is independent of the choice of $\sigma$.

Now we want to show this for $ \eta $ of arbitrary dimension.
By what we have just proved we may assume that $\eta$ does not contain the trivial character.
If $L(m,\sigma \circ \eta, k) = 0$ for $m \le 0$ and some $\sigma$ then the same holds for 
every $\sigma$  because such a zero will be determined by the $\Gamma$-factors in the functional
equation which in turn are determined by $\sigma \circ \eta (\id_{\kbar})$ and the multiplicities of
the eigenvalues $\pm 1$ of the complex conjugations.  We may therefore assume that
$L(m,\sigma \circ \eta, k) \ne 0$ for every $ \sigma $.
Because $\eta$ does not contain the trivial character, we see as just
after~\eqref{CL-statement} that this is the case if and
only if $ k $ is totally real and
\begin{equation}\label{parity}
\sigma\circ\eta(c) = (-1)^{m-1}\sigma\circ\eta(\id_{\kbar})
\end{equation} 
for every complex conjugation in $ G_k $, which is independent of $ \sigma $.
Then $ \eta\op^{m-1} $ is even,
so it factorizes through some $ \Gal(K/k) $ with $ K/k $ totally
real and finite.  Using Brauer's theorem (see \cite[Theorems~16 and~19]{Ser77}) we can write
$ \eta\op^{m-1} = \sum_{i} a_i\Ind_{k_i}^k(\chi_i) $
for 1-dimensional Artin characters $ \chi_i : G_{k_i} \to E' $
and integers~$ a_i $, where $ k \subseteq k_i \subseteq K $, so $ \chi_i $ is even,
and $E'$ is a suitable finite Galois extension of $ E $.
Then $ \eta = \sum_{i} a_i\Ind_{k_i}^k( \eta_i) $ with
$ \eta_i = \chi_i\op^{1-m} $,
and~\eqref{parity} applies with $ \eta $ replaced with~$ \eta_i $
(cf.\ the more complicated version of Brauer's theorem used in the proof
of~\cite[Theorem~1.2]{Co-Li}).  
As induction of characters is compatible with applying $ \sigma: E' \to \C$, we find
$ \sigma^{-1}(L(m,\sigma\circ\Ind_{k_i}^k(\eta_i),k_i)) 
= \sigma^{-1}(L(m,\sigma\circ\eta_i,k_i)) $ is in $ E'^\times $.
Since this is independent of $ \sigma $,
the same holds for the value in~\eqref{Lvalue} for the current $ \eta $,
and this value lies in $ E^\times $ because
$\tau L^\ast(m,\eta,k) = L^\ast(m,\tau \eta,k) = L^\ast(m,\eta,k)$ for
all $\tau $ in $ \Gal(E'/E)$.

We also define, for arbitrary $ \eta $, $v$ a finite prime of $ k $, and $m $ in $ \Z $,
the reciprocal Euler factor $ \Eul{v}^\ast(m,\eta,k) $.
For this, let $ V $ be an Artin representation of $ G_k $ 
over a finite extension $ E' $ of $ E $ with character $ \eta $,
$ D_w $ the decomposition group in $G_k$ of a prime $w$ lying above $v$,
with  inertia subgroup $I_w$ and Frobenius $ \Fr_w $.
With $ \Eulpol{v}( t,\eta) $ the determinant of $ 1-\Fr_{w} t $
we let
\begin{equation} \label{recEulerfactor}
 \Eul{v}^\ast(m,\eta,k) = \Eulpol{v}(\Nm(v)^{-m},\eta) 
\,.
\end{equation}
Clearly $ \Eulpol{v}( t,\eta) $ has coefficients in $ \Oe $,
and is independent of the choice of $ V $, $ w $ and $ \Fr_{w} $.

As in Notation~\ref{notation}, we let $k_\infty$ denote  the cyclotomic $\Z_p$-extension
of $k$, $\gamma_0$ a topological generator of $\Gal(k_\infty/k)$, and
let $q_k = \psio(\tgammanought)$ in $1 + 2p\Z_p$, where $\tgammanought$ is a lift of $\gamma_0$ to $G_k$.

Now assume $ k $ is totally real and that $ \chi $ is even and $ 1 $-dimensional.
Then there exists a unique  $ \C_p $-valued function $L_p(s,\chi,k)$  on
\begin{equation*}
 \B_k = \{ s \text{ in } \C_p \text{ with } |s|_p < |q_k-1|_p^{-1} p^{-1/(p-1)}\} 
\end{equation*}
that is analytic if $ \chi $ is non-trivial and meromorphic with at most a pole of order $ 1 $ at $ s=1 $
if $ \chi  = 1 $ , such that, in $\C_p$,
\begin{equation}\label{interpol}
 L_p (m,\chi, k) = L^\ast(m,\chi\op^{m-1},k) \prod_{v \in P} \Eul{v}^\ast(m,\chi\op^{m-1},k)
\end{equation}
for all integers $m \leq 0$.
(For the existence and uniqueness of such functions we refer to the overview statement in
\cite[Theorem~2.9]{BBdJR}, and for the radius of convergence to \cite[p.82]{greenberg83}.)
In particular, $  L_p (s,\chi, k) $ is not identically zero because $ \chi $ is even, so that the right-hand
side of~\eqref{interpol} is non-zero for $ m<0 $.

If $ \chi $ factorizes through $\Gal(k_\infty/k)$, we let $h_{\chi}(T) = \chi(\tgammanought)(1+T)-1$
where $\tgammanought$ is a lift of $\gamma_0$ to $G_k$, and we let $h_{\chi}(T) = 1$ otherwise.
Then from \cite[Eq.~(3)]{greenberg83} we have an identity
\begin{equation}\label{pLpowerseries}
 L_p(s,\chi,k) = \frac{\pi^{m_\chi} \tg_{\chi}(q_k^{1-s} -1)u_{\chi}(q_k^{1-s} - 1)}{h_{\chi}(q_k^{1-s} - 1)} 
\end{equation}
for $s $ in $ \B_k$ if $\chi$ is not the trivial character and in $ \B_k\setminus\{1\} $ if it is,
where $\pi$ is a uniformizer of $\Oe$, $m_\chi$ a non-negative integer,
$\tg_{\chi}(T)$ a distinguished polynomial in $\Oe[T]$
and $u_{\chi}(T)$ a unit of $\Oe[[T]]$.

If $ \chi : G_k \to E $ is any even Artin character, with $k$ still totally real,
then we can write $ \chi $, as we did following~\eqref{parity},
in terms of even 1-dimensional characters of totally real number fields (enlarging $ E $
to $ E' $ if necessary).
This gives a meromorphic function $L_p(s,\chi,k)$ on $ \B_k $
such that, for all $ m<0 $ (and possibly for $m=0$), $ L_p(m,\chi,k) $ is defined, non-zero,
and~\eqref{interpol} holds in $E$. 
This interpolation property shows that $ L_p(s,\chi,k) $ is independent
of the way we write $ \chi $ in terms of induced characters,
hence $ L_p(s,\chi,k) $ is compatible with induction of characters.
Using the action of $ \Gal(\Qpbar/E) $ and the $ p $-adic Weierstra\ss\ preparation
theorem for $ \Oep[[T]] $ one sees easily that
\begin{equation}\label{LpIw}
 L_p(s,\chi,k) = \pi^{m_\chi} P_\chi(q_k^{1-s}-1) u_\chi(q_k^{1-s}-1)/Q_\chi(q_k^{1-s}-1)
\end{equation}
with $ \pi $ a uniformizer of $ \Oe $, $ m_\chi $ an integer, $ P_\chi(T) $
and $ Q_\chi(T) $
relatively prime distinguished polynomials in $ \Oe[T] $, and $ u_\chi(T) $ in $ \Oe[[T]]^\times $.
So if $ L_p (s,\chi, k) $ is defined at some $ s $ in $ \B_k(E) = \B_k \cap E $
(see~\eqref{BEdef}), then its value lies in~$ E $.

\begin{remark}
Greenberg has shown \cite[Proposition 5]{greenberg83} that the main conjecture of Iwasawa theory for
the prime $ p $ (i.e., the statements in Theorem~\ref{mainc} and Remark~\ref{notSremark}
for $p$) implies that in~\eqref{LpIw} one can take $ Q_{\chi}[T]$
to be a product of factors $ \zeta\cdot(1+T) -1 $ with $ \zeta $ a $ p $-th power root of unity.
Moreover, the number of factors with $ \zeta=1 $ is at most the multiplicity
$ m $ of the trivial character in $ \chi $, so that $ (s-1)^m L_p(s,\chi,k) $ is analytic on $ \B_k $.
We could use this for $ p \ne 2 $ by Theorem~\ref{mainc} and Remark~\ref{notSremark},
but as mentioned in the introduction, 
we take a uniform approach for all primes based on Theorem~\ref{mainc} and Assumption~\ref{2-mainc}.
\end{remark}

\begin{remark} \label{zero_remark}
Note that if $ G(t) $ is a polynomial with coefficients in a finite extension $ E $
of $ \Q_p $, $ q_k $ is in $ 1+2p\Z_p $, and $ G(q_k^s-1) = 0 $ for some $ s $ in $ \B_k $, then $ q_k^s $
lies in a finite extension $ E' $ of $ E $ inside $ \Qpbar $, hence $ s $ lies in $ \B_k(E') $.  Therefore
any zero or pole of $ L_p(s,\chi,k) $ lies in some $ \B_k(E') $ with $ E'/\Q_p $ finite.
\end{remark}

In order to extend~\eqref{interpol} to truncated $ L $-functions we
now introduce some more notation.
We shall denote the image of $ z $ in $ \Z_p^\times $ under the projection
$ \Z_p^\times = (1+2p\Z_p) \cdot \mu_{\varphi(2p)} \to 1+2p\Z_p $ by $\langle z\rangle $.
Then mapping $ s $ to $ \langle z\rangle^s $ defines
an analytic function on
$\{ s \in \C_p \text{ with } |s|_p < |\langle z \rangle -1 |_p^{-1} p^{-1/(p-1)} \}$,
and $ |\langle z\rangle^s-1|_p < p^{-1/(p-1)} $.
Now let $ k $ be any number field, and $ \eta : G_k \to E $ an
Artin character.
For a finite place $ v $ of $ k $ not in $P$,
we define a $ p $-adic `diamond reciprocal Euler factor' 
$ \dEul{v}(s,\eta,k) =  \Eulpol{v}(\langle \Nm(v)\rangle^{-s},\eta) $ (cf.~\eqref{h0eul}),
i.e., we use $\langle \Nm(v) \rangle$
instead of $ \Nm(v) $ in the definition of 
$ \Eul{v}^\ast(m,\eta,k) $ in~\eqref{recEulerfactor},
which allows us to replace the integer $ m $ with~$ s $.
With notation as just before~\eqref{recEulerfactor}, we observe that
twisting the representation of $ G_k $ on $ V $
with $\op^{-m}$ for $m$ in $\Z$ gives a $G_k$-action on $V$ with character $\eta \op^{-m}$,
and $V^{I_w}$ is the same for both actions because $\op$ is unramified at $v$.
Since $\psip(\Fr_w) = \Nm(v)$ we have $\op(\Fr_{w}) = \Nm(v)/\langle \Nm(v) \rangle$, hence
\begin{equation}\label{diamondinterpolation}
 \Eulpol{v}(\langle \Nm(v) \rangle^{-m} ,\eta\op^{-m}) =
\Eulpol{v}( \Nm(v)^{-m} ,\eta) 
\end{equation}
in $ E $.  Therefore $ \dEul{v}(s,\eta\op^{-1},k) $ is a
$ p $-adic analytic function interpolating the values $ \Eul{v}^\ast(m,\eta\op^{m-1},k) $ for $m$ in $\Z$.
It converges on $\B_k$ since $ \mtt{\Nm(v)}^{-s} = \psio(\Fr_w)^{-s}$ converges on this domain.

\begin{remark} \label{deulzero}
Since the eigenvalues of Frobenius are roots of unity,  and $ \langle \Nm(v)\rangle^{-s} $
is only a root of unity for $ s=0 $, $ \dEul{v}(s,\eta,k) \ne 0 $ for $ s \ne 0 $.  Moreover,
the order of vanishing of $\dEul{v}(s,\eta,k)$ at $s = 0$ equals the dimension of $(V^{I_w})^{\Fr_w = 1}$
where $V$ realizes $\eta$ (cf.\ Proposition~\ref{gysinprop}(2)).
\end{remark}

We now return to the case where $ k $ is totally real and $ \chi $ is even.
If $ S $ is a finite set of primes of $ k $ containing $ P $, then we
define a meromorphic $ p $-adic $ L $-function $ L_{p,S} $ on
$ \B_k $ by putting
\begin{equation}\label{pLS}
 L_{p,S}(s,\chi,k) = L_p(s,\chi,k) \prod_{v \in S \setminus P} \dEul{v}(s,\chi\op^{-1},k)
\end{equation}
(cf.~\cite[Section 3, p.147]{greenberg77}).
Note that $L_p(s,\chi,k) = L_{p,P}(s,\chi,k)$, the latter being a better notation.
Finally, for an integer $ m<0 $ (and possibly for $ m=0 $),
$ L_{p,S}(s,\chi,k) $ is defined at $ m $, and in $ E $ we have
\begin{equation}\label{preLSinterpol}
 L_{p,S}(m,\chi,k) =  L_S^\ast(m,\chi\op^{m-1},k) 
\end{equation}
where $ L_S^\ast(m,\chi\op^{m-1},k) = L^\ast(m,\chi\op^{m-1},k) \prod_{v\in S} \Eul{v}^\ast(m,\chi\op^{m-1},k) $.

\begin{remark} \label{deuler-induction}
Suppose that $k'/k$ is a finite extension of number fields and
$\psi : G_{k'} \rightarrow E$ an Artin character.
For a finite prime $v$ of $k$ not in $P$ one has
$\dEul{v}(s,(\Ind_{k'}^k \psi) \op^{-1},k) = \prod_{w \mid v} \dEul{w}(s,\psi \op^{-1}, k')$,
because $\dEul{v}$ interpolates the values of $\Eul{v}^\ast$ as just after~\eqref{diamondinterpolation},
and for those an analogous result holds.
Therefore, if $k'$ is totally real and  $\chi: G_{k'} \to E$ is an even Artin character, then
$L_{p,S}(s,\Ind_{k'}^k \chi, k) = L_{p,S'}(s,\chi,k')$ where $S'$ is the set of primes of
$k'$ consisting of all the primes lying above those in~$S$.
\end{remark}

\begin{remark} \label{analytic-remark}
If $ f(T) \ne 0 $ is in $ \Oe[[T]] $ then the power series
$ g(s) = f(q_k^{1-s}-1)$  $ = \sum_{n\ge0} a_n s^n $ converges on $ \B_k $
and $ |a_n|_p \le |q_k-1|_p^n p^{n/(p-1)} $.  With $ E' = E(\delta) $ for some $\delta$
algebraic over $E$ with $ |\delta|_p = |q_k-1|_p^{-1} p^{-1/(p-1)} $,
we find that $ g(\delta \sp) $ is in $ \Oep[[\sp]] $, so can
be written as $ \pi_{E'}^l P(\sp) u(\sp) $ with $ \pi_{E'} $
a uniformizer of $ E' $, $ l $ a non-negative integer, $ P(\sp) $ a
distinguished polynomial in $ \Oep[\sp] $, and $ u(\sp) $ in $ \Oep[[\sp]]^\times $.
By the discussion following~\eqref{diamondinterpolation},
the same argument can be applied to a factor
$ \dEul{v}(s,\chi\op^{-1},k) = \Eulpol{v}(\langle \Nm(v)\rangle^{-s},\chi\op^{-1}) $
in~\eqref{pLS}.  From~\eqref{LpIw} and~\eqref{pLS} we 
have, for $ s $ in~$ \B_k $,
\begin{equation*}
 L_{p,S}(s,\chi,k) = \pi_{E'}^l \frac{P(s/\delta)}{Q(s/\delta)} u(s/\delta)
\end{equation*}
with $ l $ in $ \Z $, $ P(\sp) $ and $ Q(\sp) $ distinguished
polynomials in $ \Oep[\sp] $, and $ u(\sp) $ in $ \Oep[[\sp]]^\times $,
so that all the zeroes and poles here come from $ P(\sp)/Q(\sp) $.
Therefore $ L_{p,S}(s,\chi,k) $, which equals $ L_p(s,\chi,k) $ if $S = P$,
is analytic (i.e., given by a power series that converges on
$ \B_k $) if and only if it is bounded on $ \B_k(E) $ for all finite extensions
$ E $ of $ \Q_p $.
\end{remark}

\section{Iwasawa theory} \label{iwasawa-theory}

In this section we discuss the main conjecture of Iwasawa theory and prove some lemmas for later use.

Let $ k $ be a totally real number field, $ p $ a prime number, $E$ a finite extension of $\Q_p$,
$\chi: G_k \rightarrow E$ an even Artin character,
$K$ the (totally real) fixed field of the kernel of $\chi$, and $G=\Gal(K/k)$. 
As in Notation~\ref{notation} we let $k_\infty$ and $K_\infty$ be the cyclotomic $\Z_p$-extensions
of $k$ and $K$ respectively, $\gamma_0$ a topological generator of $\Gamma_0 = \Gal(k_{\infty}/k)$.
Let $L_{\infty}$ be the maximal Abelian pro-$p$-extension of $K_{\infty}$ that is unramified
outside of the primes above $p$ and the infinite primes.  Write $\M = \Gal(L_{\infty}/K_{\infty})$ and let
$\Gal(K_{\infty}/k)$ act on the right on $\M$ by conjugation.

Now assume that $K \cap k_\infty = k$, so that $ \Gal(K_{\infty}/k) \simeq G \times \Gamma_0$.
We let $\Gamma_0$ and $G$ act on $\M$ via this isomorphism,
and thus view $\M$ as a $\Z_p[G][[T]]$-module where $T$ acts as $\gamma_0 - 1$.

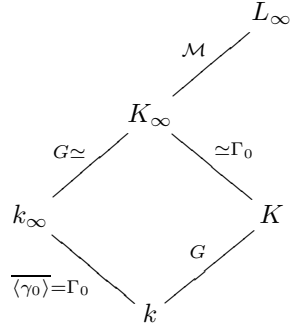
\begin{equation*}
\xymatrix{
&& L_{\infty} \ar@{-}[dl]_{\M} \\
& K_{\infty} \ar@{-}[dl]_{G \simeq} \ar@{-}[dr]^{\simeq \Gamma_0} & \\
k_{\infty} \ar@{-}[dr]_{\overline{\left< \gamma_0 \right>} = \Gamma_0} && K \ar@{-}[dl]_G \\
& k & 
}
\end{equation*}

In the remainder of this section we assume $\chi$ is $1$-dimensional.
Then \cite[Theorem 13.31]{wash} implies that $\M \otimes_{\Z_p[G]} \Oe$
is a finitely generated torsion $\Oe[[T]]$-module, where $T$ acts on $\M$, and $G$ acts on $\Oe$ via $\chi$.

\begin{definition} \label{structure}
Suppose $ \O $ is the valuation ring in a finite extension of $ \Q_p $.
If $Y$ is a finitely generated torsion $\O[[T]]$-module then
by \cite[Thm.~5 of \S 4.4]{Bou65} it is isogenous
to $A \oplus B$ (i.e., there is a morphism $ Y \to A\oplus B $
of $ \O[[T]] $-modules with finite kernel and cokernel), with
\begin{eqnarray*}
A = \bigoplus_{i=1}^r \frac{\O[[T]]}{(\pi^{\mu_i})}; &&
B = \bigoplus_{j=1}^s \frac{\O[[T]]}{(g_j(T))},
\end{eqnarray*}
where $\pi$ is a uniformizer of $\O$, $ r, s \ge 0 $, the $\mu_i$ are positive integers, and the $g_j(T)$ are
distinguished polynomials in $\O[T]$.
We let $\mu_Y = \sum_{i=1}^r \mu_i$ and $g_Y = \prod_{j=1}^s g_j$.
\end{definition}

We denote the $g_Y$ and $\mu_Y$ in Definition~\ref{structure}
for $Y = \M \otimes_{\Z_p[G]} \Oe$ and $ \O = \Oe $ 
by $g_\chi$ and $\mu_{\chi}$ respectively.
Then the following result is part of
the main conjecture of Iwasawa theory as proved by 
Wiles for $ p $ odd in \cite{Wiles90}.

\begin{theorem} \label{mainc}
If $p$ does not divide $[K:k]$ and $p \ne 2$ then, with notation as in~\eqref{pLpowerseries}, 
$ g_\chi(T) = \tg_\chi(T) $ and $ \mu_\chi = m_\chi $.
\end{theorem}

\begin{proof}
The first equality is proved in \cite[Theorem~1.3]{Wiles90}, which
is equivalent to \cite[Theorem~1.2]{Wiles90} via \cite[Proposition~3]{greenberg77}.
Note that in \cite{greenberg77}
Greenberg uses the direct limit of class groups in the cyclotomic tower,
but this is isogenous to the
Galois group of the maximal Abelian unramified extension of $K_\infty$ by \cite[Theorem~11]{Iwa73}.

The equality of the $\mu$-invariants follows from \cite[Theorem~1.4]{Wiles90}.  Note that
by \cite[Proposition 1]{greenberg77}, combined with \cite[Theorem~11]{Iwa73}, we have
that $\mu_\chi$ equals the $\mu$-invariant considered in \cite[Theorem~1.4]{Wiles90}.
\end{proof}

\begin{remark} \label{notSremark}
If $ p $ divides $ [K:k] $ but $K \cap k_\infty = k$, then
\begin{equation*}
 \M \otimes_{\Z_p[G]} E \simeq \bigoplus_{j=1}^s E[[T]]/(g_j(T)) \simeq \bigoplus_{j=1}^s E[T]/(g_j(T)) 
\end{equation*}
for some distinguished polynomials $ g_j(T) $ in $ \Oe[T] $,
and, when $ p \ne 2$, $ \tg_\chi = \prod_{j=1}^s g_j $ by \cite[Theorem 1.2,1.3]{Wiles90}.
\end{remark}

Our proofs in Section~\ref{EC} will also work for $ p=2 $ if we assume the following weaker version of the
main conjecture of Iwasawa theory.

\begin{assumption} \label{2-mainc}
If $p=2$ and $\chi$ is of odd order then $\tg_\chi(T) = g_\chi(T)$ and $m_{\chi} = \mu_\chi$.
\end{assumption}

\begin{remark} \label{itholds}
Let $ \chi $ be a 1-dimensional even Artin character of $G_k$ such that $ K \cap k_\infty = k $, and assume $ p=2 $.
If $ k=\Q $ then by \cite[Theorem~6.2]{Wiles90}
we have $ \tg_\chi(T) = g_\chi(T) $ as well.
In fact, this still holds for many more pairs $ (k,\chi) $ by \cite[Theorem~11.1]{Wiles90}.
We also note that $m_\chi \ge [k:\Q]$ by \cite[(4.8)]{Rib79}.
On the other hand, if $k$ is Abelian over $\Q$ and $\chi$ the trivial character,
then by combining the main theorem on \cite[p.377]{FeWa} with \cite[Section~6.4]{schmidt2002},
we have $\mu_\chi = [k:\Q]$, hence
$m_\chi \ge \mu_\chi$.  Equality holds when $k=\Q$ because
then $m_\chi = 1$ by \cite[Lemma~7.12]{wash}.
\end{remark}

We shall later use the following lemma about the structure of finitely generated torsion $\O[[T]]$-modules.

\begin{lemma} \label{isogeny_lemma}
Let $\O$ be the valuation ring in a finite extension of $\Q_p$, $Y$ a finitely generated
torsion $\O[[T]]$-module, so $ Y $ is isogenous to $ A \oplus B$ with
\begin{eqnarray*}
A = \bigoplus_{i=1}^r \frac{\O[[T]]}{(\pi^{\mu_i})}; &&
B = \bigoplus_{j=1}^s \frac{\O[[T]]}{(g_j(T))},
\end{eqnarray*}
where $\pi$ is a uniformizer of $\O$,
$ r, s \ge 0 $, the $\mu_i$ are positive integers, and the $g_j(T)$ are
distinguished polynomials in $\O[T]$.  Suppose that $ \Gamma \simeq \Z_p $ acts on
$Y$ such that a topological generator acts as multiplication by $u\cdot(1+T)$ for some $u $ in $ 1+ \pi \O$.  
Then the following are equivalent.
\begin{enumerate}
\item $Y^\Gamma$ is finite;
\item $Y_\Gamma$ is finite;
\item $g_Y(u^{-1} - 1) \neq 0$.
\end{enumerate}
In particular, these are satisfied for all but finitely many $u$ in $1 + \pi \O$.  Moreover,
if these hold, then
$ (\# Y_\Gamma)/(\# Y^\Gamma) = |\pi^{\mu_Y} g_Y(u^{-1} - 1)|_p^{-[\mathrm{Frac}(\O):\Q_p]}$.
\end{lemma}

\begin{proof}
For any $\mu \geq 1$, we have a short exact sequence
\begin{equation*}
\xymatrix{
0 \ar@{->}[r] & \dfrac{\O[[T]]}{(\pi^{\mu})} \ar@{->}[rr]^{u\cdot(1+T)-1} && \dfrac{\O[[T]]}{(\pi^{\mu})}
\ar@{->}[r] & \dfrac{\O}{(\pi^{\mu})} \ar@{->}[r] & 0.
}
\end{equation*}
This shows that $A^\Gamma = 0$ and that $A_\Gamma \simeq \bigoplus_{i=1}^r \O / (\pi^{\mu_i}) $.
Similarly we get
\begin{eqnarray} \label{b_to_g}
B_\Gamma \simeq \bigoplus_{j=1}^s \frac{\O}{(g_j(u^{-1} - 1))}.
\end{eqnarray}

Note that the $\O$-ranks of $B_\Gamma$ and $B^\Gamma$ both equal the number of $ g_j $ with $ g_j(u^{-1}-1)=0 $,
and that $ B^\Gamma $ has no non-trivial finite subgroup.
Therefore $B_\Gamma$ is finite if and only if $B^\Gamma$ is finite, hence trivial.
From (\ref{b_to_g}) it is clear that $B_\Gamma$ is finite if and only if $g_Y(u^{-1} - 1) \neq 0$.
Since $g_Y$ is a polynomial, this last condition holds for all but finitely many $u$ in $1 + \pi \O$.

Let $Y'$ be a quotient of $Y$ by a finite submodule such that $Y'$ injects into $A \oplus B$.
Note that finiteness of $Y^\Gamma$ (resp. $Y_\Gamma$) is equivalent to that of $(Y')^\Gamma$ (resp. $(Y')_\Gamma$).
We have an exact sequence of $ \O[[T]] $-modules,
$$ 0 \rightarrow Y' \rightarrow A \oplus B \rightarrow C \rightarrow 0,$$
where $C$ is a finite $\O[[T]]$-module.
Since $A_\Gamma$ is finite it follows that $(Y')_\Gamma$ is finite if and only if $B_\Gamma$ is finite.
Also, because $A^\Gamma$ is trivial, $(Y')^\Gamma$ is finite if and only if $B^\Gamma$ is finite,
hence~(1),~(2) and~(3) are equivalent.
Moreover, in this case note that $(Y')^\Gamma$ is trivial, $\# C^\Gamma = \# C_\Gamma$, and therefore
\begin{eqnarray*}
 \dfrac{\# Y_\Gamma}{\# Y^\Gamma } 
=
 \# (Y')_\Gamma = \# A_\Gamma \cdot \# B_\Gamma = |\pi^{\mu_Y} g_Y(u^{-1} - 1)|_p^{-[\mathrm{Frac}(\O):\Q_p]}
\,.
\end{eqnarray*}
\end{proof}

\section{Multiplicative Euler characteristics and $ p $-adic $ L $-functions} \label{EC}

In this section we prove Theorem~\ref{introtheorem}, formulate Conjecture~\ref{conjecture} (which would
imply Conjecture~\ref{introconjecture}) and briefly discuss the case of
a $1$-dimensional even Artin character in Example~\ref{Hexample}.  Note that Theorem~\ref{introtheorem}(3)
implies Proposition~\ref{contprop} applies for almost all $e$.

We use notation as in Notation~\ref{notation}.  Also, for any Artin
character $\eta$ of $ G_k $ we let $k_\eta$ denote the fixed field of the kernel of the
corresponding representation.

As seen in Remark~\ref{coeffremark}(2) and~(3),
the size of $H^i(\O_{k,S},\qlatmtt{E}{\chi}{1-e})$
for $ i=0 $ can depend on the choice of the lattice $ \lat{E}{\chi} $.
We shall see that a (modified) multiplicative
Euler characteristic can be defined independent of this choice when those cohomology
groups are finite for $ i=0,1 $ and~2.
By Proposition~\ref{corankprop}, this can only happen when
$ k $ is totally real and the character even, so it will be denoted by $\chi$.
For $ i=2 $ the finiteness of the cohomology group then implies
it is trivial, so that the sizes for $ i=0 $ and $ i=1 $ depend on
the choice of the lattice in the same way.

\begin{definition} \label{euler_char}
Let $ E $ be a finite extension of $ \Q_p $, $k$ a totally real number field,
$ \chi : G_k \to E $ an even Artin character realizable over~$ E $, and
$ e $ in $ \B_k(E) $.
If $H^i(\O_{k,S},\qlatmtt{E}{\chi}{1-e})$ is finite for $i = 0,1$ and $2$
then it is trivial for $  i=2 $ by Proposition~\ref{divprop}, and we define
\begin{alignat*}{1}
\EC_S(e,\chi,k) 
& = 
 \left( \prod_{i = 0}^2 \left( \# H^i(\O_{k,S},\qlatmtt{E}{\chi}{1-e}) \right)^{(-1)^i} \right)^{1/[E:\Q_p]}
\\
&  =
 \left( \frac{\# H^0(G_S, \qlatmtt{E}{\chi}{1-e})}{\# H^1(G_S, \qlatmtt{E}{\chi}{1-e})}\right)^{1/[E:\Q_p]}
\,.
\end{alignat*}
\end{definition}

\begin{remark}
By Remark~\ref{coh_dim}, if $ p \ne 2 $ then $ \EC_S(e,\chi,k) $ is a (modified) Euler characteristic,
but if $ p=2 $ then it is truncated.
For $ p=2 $, because every complex conjugation acts trivially,
$ H^i(k_v,\qlatmtt{E}{\chi}{1-e}) $
for $ i \ge 3 $ is trivial
if $ i $ is even since $ \qlatmtt{E}{\chi}{1-e}$ is 2-divisible, 
and isomorphic with $ \qlat{E}{\chi} [2] $ for $ i $ odd.
\end{remark}

\begin{remark} \label{independence}
(1) Definition~\ref{euler_char} for a fixed $E$ is independent of the choice of the lattice $M = \lat{E}{\chi}$.
Namely,
if $M'$ is another lattice for the same representation, then we
obtain an exact sequence~\eqref{sos}.
The $H^i(\O_{k,S},(M' \otimes_{\Z_p} \Q_p/\Z_p)\mtt{1-e})$ are finite for $i=0,1$ and $2$
by Remark~\ref{Mremark}(1), and Lemma~\ref{eul_ind} applied with $A = X_E\mtt{1-e}$ implies that
$\EC_S(e,\chi,k)$ is the same for $\qlat{E}{\chi} = M \otimes_{\Z_p} \Q_p/\Z_p$ or  $M' \otimes_{\Z_p} \Q_p/\Z_p$.

(2) Definition~\ref{euler_char} is independent of the field $E$ by Remark~\ref{Mremark}(2).

(3)
If $ \chi_j $ for $ j=1,2 $ are even Artin characters,
then by Remark~\ref{Mremark}(3), the $ \EC_S(e,\chi_j,k) $ are defined if and only if 
$ \EC_S(e,\chi_1+\chi_2,k) $ is defined, in which case 
$ \EC_S(e,\chi_1+\chi_2,k) = \EC_S(e,\chi_1,k) \cdot \EC_S(e,\chi_2,k) $.

(4)
Let $k'/k$ be a finite extension of totally real fields, $\chi':G_{k'} \to E$ be an even
Artin character realizable over $E$, and $\chi = \Ind_{k'}^k \chi'$.  If $\chi$ is unramified outside of $S$
then by part~(1) and Remark~\ref{Mremark}(4),
$ \EC_S(e,\chi',k') $ is defined if and only if $ \EC_S(e, \chi, k) $ is defined, in which case the two are equal.

(5)
Let $S'$ be a finite set of primes of $k$ with $S \subseteq S'$.
Note that from the definition of $\dEul{v}$ in Section~\ref{pL-functions}, we have that
\begin{equation*}
\dEul{v}(e,\chi \op^{-1},k) = \det(1 - \mtt{\Nm(v)}^{-e} \Fr_w | \vlat{E}{\chi \op^{-1}})
\end{equation*}
for any $v$ not
in $S$ and $w$ above $v$.  If this is non-zero then by Remark~\ref{Mremark}(5)
we have
\begin{equation*}
   \# H^0(\Gal(\Fbar_v/\F_v),\qlatmtt{E}{\chi \op^{-1}}{-e})
 = |\dEul{v}(e,\chi \op^{-1},k)|_p^{-[E:\Q_p]}
\,.
\end{equation*}
So if $\dEul{v}(e,\chi \op^{-1},k) \ne 0$ for $v $ in $ S' \setminus S$ then from Remark~\ref{Mremark}(5) we see
that $\EC_S(e,\chi,k)$ is defined if and only if $\EC_{S'}(e,\chi,k)$ is defined, in which case
\begin{alignat*}{1}
\EC_{S'}(e,\chi,k) = \EC_S(e,\chi,k) \cdot \prod_{v \in S' \setminus S}
|\dEul{v}(e,\chi \op^{-1},k)|_p
\,.
\end{alignat*}

\end{remark}

We now prove the finiteness statement in Theorem~\ref{introtheorem}.

\begin{lemma} \label{finitelemma}
Let $k$ be a totally real number field, $p$ a prime number, $E$ a finite extension of $\Q_p$,
and $ \chi : G_k \to E $ an even Artin character realizable over~$ E $.
Assume that $ S $ is a finite set of finite primes of $ k $ containing the primes above $ p $
as well as the finite primes at which $ \chi $ is ramified.
Then $H^i(\O_{k,S},\qlatmtt{E}{\chi}{1-e})$ is finite for all but finitely many $e$ in $\B_k(E)$.
\end{lemma}

\begin{proof}
The result for $i = 0$ and $i \ge 3$ follows from Theorem~\ref{cogentheorem}(3) and Remark~\ref{coh_dim}.
Hence by Proposition~\ref{corankprop}, it is enough to consider $i=1$.

Since $\chi$ occurs in $\Ind_{k_\chi}^k (\chi|_{k_\chi})$ and $\B_k(E) \subseteq \B_{k_\chi}(E)$,
parts~(1),~(3) and~(4) of Remark~\ref{Mremark} imply that it is enough to prove the result for the trivial character.
We therefore assume that $\chi$ is trivial and hence $\qlat{E}{\chi} = E/\Oe$.  By the last statement in
Proposition~\ref{corankprop}, we may also assume that $S = P$.

From the spectral sequence
\begin{equation*}
H^j(\Gamma_0,H^{j'}(\Gal(\Omega_P/k_\infty),E/\Oe \mtt{1-e})) \Rightarrow H^{j+j'}(\O_{k,P},E/\Oe \mtt{1-e})
\end{equation*}
it is enough to show that $H^j(\Gamma_0,H^{1-j}(\Gal(\Omega_P/k_\infty),E/\Oe \mtt{1-e}))$ is finite for $j=0,1$.
Since $\Gal(\Omega_P/k_\infty)$ acts trivially on $E/\Oe \mtt{1-e}$, for $j=1$ this group equals the cokernel
of multiplication by $q_k^{1-e} - 1$ on $E/\Oe$, which is finite for $e \ne 1$.  For $j=0$, this group is isomorphic
to $\Hom(\M, E/\Oe \mtt{1-e})^{\Gamma_0} \simeq \Hom((\M \otimes_{\Z_p} \Oe \mtt{e-1})_{\Gamma_0}, \Q_p/\Z_p)$,
where $\M$ is the Galois group of the maximal Abelian extension of $k_\infty$ that is unramified outside of the primes
above $p$ and $\infty$, on which $\Gamma_0$ acts by conjugation, and $\Hom$ is the functor of continuous homomorphisms.
Because $\M \otimes_{\Z_p} \Oe$ is a torsion $\Oe[[T]]$-module
as mentioned at the beginning of Section~\ref{iwasawa-theory}, the result follows
from Lemma~\ref{isogeny_lemma} with $Y = \M \otimes_{\Z_p} \Oe$, $\O = \Oe$, $ \Gamma = \Gamma_0 $ with topological
generator corresponding to $ \gamma_0 $, and $u = q_k^{e-1}$.
\end{proof}

In order to prove the rest of Theorem~\ref{introtheorem} we first prove two weaker results:
Theorem~\ref{ebn} if $ \chi $ is 1-dimensional, and Theorem~\ref{ebn2}
for general $ \chi $.

\begin{theorem} \label{ebn}
Let $ k $ be a totally real number field, $ p $ a prime number, $ E $ a finite extension of $ \Q_p $
with valuation ring $ \Oe $,
$ \chi : G_k \to E $ a 1-dimensional even Artin character, and let be $ e $ in $ \B_k(E) $.
Assume that $ S $ contains $ P $ as well as all the finite primes of $ k $ at which $ \chi $ is ramified.
Let $k \subseteq k' \subseteq k_{\chi}$ be such that $[k':k] = p^l$ for some $l \geq 0$ and 
$ p $ does not divide $ [k_{\chi}:k'] $.  Assume that $ e \ne 1 $ if $\chi_{|G_{k'}} $ is the trivial
character, so that $L_{p,S}(s, \chi_{|G_{k'}} , k')$ is defined at $s = e$.
If its value there is non-zero, then with Assumption~\ref{2-mainc} if $p=2$, we have
\begin{enumerate}
\item
$H^i(\O_{k,S},\qlatmtt{E}{\chi}{1-e})$
is finite for $i = 0, 1$ and trivial for $i = 2$;

\item
$ |L_{p,S}(e,\chi ,k)|_p = \EC_S(e,\chi,k) $.
\end{enumerate}
\end{theorem}

\begin{proof}
We begin by fixing some notation.
We write $K $ for $ k_{\chi}$, $G$ for $ \Gal(K/k)$, and also view $\chi$ as a character of $G$.
Also, for any 1-dimensional (even) Artin character $ \chi $ with values in $ E $ we choose
$ \lat{E}{\chi}$ to be $ \Oe $ with action given through multiplication by the character, 
for which we write $ \slat{\chi} $, and write $\sqlat{\chi} = \qlat{E}{\chi} = \slat{\chi} \otimes_{\Z_p} \Q_p/\Z_p$.

We first prove the result for $ l=0 $, so that $ k' = k $, and
we assume $L_{p,S}(e,\chi ,k) \ne 0$.
By~\eqref{pLS} and Remark~\ref{independence}(5), it suffices to prove the theorem when
$S$ is the union of the set of primes of $k$ where $\chi$ is ramified and $P$.
Since $\chi\op^{-1}$ is 1-dimensional
and ramified at all $ v $ in $ S \setminus P $ we then have
$L_{p,S}(e,\chi,k) = L_p(e,\chi,k)$.

Note that $\Gal(K_{\infty}/k) \simeq G \times \Gamma_0$ because $ p $ does not divide $ [K:k] $.
We shall consider $\gamma_0$ as a topological generator of $\Gal(K_{\infty}/K)$ as well 
via the isomorphism $\Gal(K_{\infty}/K) \simeq \Gal(k_{\infty}/k) = \Gamma_0$.

We have
\begin{equation} \label{H0}
\begin{aligned}
   H^0(\O_{k,S},\sqlatmtt{\chi}{1-e})
& \simeq
   H^0(\Gal(K_{\infty}/k), \sqlatmtt{\chi}{1-e})
\\
& = 
  \left( \sqlatmtt{\chi}{1-e}^G \right)^{\Gamma_0}
\,.
\end{aligned}
\end{equation}
If $ \chi $ is non-trivial then
$ \chi(g)-1 $ is in $ \Oe^\times $ for any non-trivial $ g $ in $ G $ 
as $ l=0 $, so $ \sqlatmtt{\chi}{1-e}^G $ and~\eqref{H0}
are trivial.
If $ \chi $ is trivial then~\eqref{H0} equals $ ( \sqlatmtt{\chi}{1-e})^{\Gamma_0} $, which
is isomorphic to the kernel of multiplication by
$ q_k^{1-e}-1 $ on $ E/\Oe $.
In this case $e \ne 1$ by assumption, so \eqref{H0} is isomorphic to $\Oe/(q_k^{1-e}-1)$.
In either case, we see from the definition of $ h_\chi $ in Section~\ref{pL-functions} that
\begin{equation} \label{h_0}
\# H^0(\O_{k,S},\sqlatmtt{\chi}{1-e})^{1/[E:\Q_p]} = |h_\chi(q_k^{1-e}-1)|_p^{-1}
\,.
\end{equation}

We now consider $H^1(\O_{k,S},\sqlatmtt{\chi}{1-e})$. 
Note that $H^i(G,T) = 0$ for any $\Z_p$-module $T$ if $i \geq 1$, because $ \# G  $ annihilates
this group and $ \# G$ is in $ \Z_p^\times $.  
Hence
taking $G$-invariants is exact
and $H^i(\O_{k,S},T) \simeq H^0(G,H^i(\O_{K,S},T))$.

With $Z = \Spec \O_{K,P} \setminus \Spec \O_{K,S}$ we have, for all $ j $, that
\begin{equation*}
 H^j(Z,\sqlatmtt{\chi \op^{-1}}{-e})
\simeq
 \prod_{w \in Z} H^j(\Gal(\Fbar_w/\F_w),\sqlatmtt{\chi \op^{-1}}{-e})
\,.
\end{equation*}
Let $G_w$ and $I_w$ denote the decomposition and inertia subgroups for $ w $ inside $G$.
Note that $I_w$ is non-trivial for each $w$ in $ Z$ since $\chi$ is ramified at $w$.  Hence there exists
$ g $ in $ I_w $ with $ \chi(g)^{-1} - 1 \ne 0 $, and this lies in $ \Oe^\times $ as $ l=0 $.
The $ I_w $-invariants, and hence $ G $-invariants of
$ H^j(\Gal(\Fbar_w/\F_w),\sqlatmtt{\chi \op^{-1}}{-e}) $, are therefore trivial
because the action of $I_w$ on this group is as on $\sqlat{\chi}$.
From the Gysin sequence as in Remark~\ref{Mremark}(5), but with $k$, $ S $ and $S'$ replaced with
$ K $, $ P$ and $S$, and with the current $ Z $, we find by taking $ G $-invariants that
\begin{alignat*}{1}
H^i(\O_{k,S},\sqlatmtt{\chi}{1-e}))
& \simeq 
H^i(\O_{K,S},\sqlatmtt{\chi}{1-e})^G
\\
& \simeq
H^i(\O_{K,P},\sqlatmtt{\chi}{1-e})^G
\end{alignat*}
for all $i \geq 0$.

For the remainder of the case $l=0$ we can now follow the setup for the proof of \cite[Theorem~6.1]{Ba-Ne}.
Let $\OKP$ denote the maximal extension of $K$ that is unramified outside of the primes above
$p$ or $ \infty $, so that by Remark~\ref{duality-remark} we have
\begin{alignat*}{1}
H^i(\O_{K,P},\sqlatmtt{\chi}{1-e})^G & \simeq H^i(\Gal(\OKP/K),\sqlatmtt{\chi}{1-e})^G
\\
& \simeq H^i(\Gal(\OKP/k),\sqlatmtt{\chi}{1-e})
\,.
\end{alignat*}
We obtain from the Hochschild-Serre spectral sequence for the normal subgroup $ \Gal(\OKP/K_\infty)$
of $ \Gal(\OKP/k)$ with quotient $ \Gal(K_\infty/k) \simeq \Gamma_0 \times G$
a 5-term exact sequence.  
Using that $ H^0(\Gal(\OKP/K_\infty), \sqlatmtt{\chi}{1-e}) =  \sqlatmtt{\chi}{1-e}$,
we find there is an exact sequence
\begin{equation} \label{fiveterm}
\begin{aligned}
 0
& \to
 H^1(\Gal(K_\infty/k),\sqlatmtt{\chi}{1-e})
\\
& \to
 H^1(\Gal(\OKP/k),\sqlatmtt{\chi}{1-e})
\\
& \to
 H^0(\Gal(K_\infty/k),H^1(\Gal(\OKP/K_{\infty}),\sqlatmtt{\chi}{1-e}))
\\
& \to
 H^2(\Gal(K_\infty/k),\sqlatmtt{\chi}{1-e})
\,.
\end{aligned}
\end{equation}
Note that $H^i(\Gal(K_\infty/k),\cdot) \simeq H^i(\Gal(K_\infty/K),\cdot)^G = 0$ for $i \geq 2$ as
the cohomological dimension of $\Gal(K_\infty/K)$ is 1, so the last term in~\eqref{fiveterm} is trivial.
As $H^1(\Gal(K_\infty/K),\sqlatmtt{\chi}{1-e}) $ is the cokernel of multiplication
by $ q_k^{1-e} - 1 $ on $ \sqlat{\chi}$, the first term in~\eqref{fiveterm} is trivial when $e \ne 1$
and equals $\sqlat{\chi}^G$ when $e = 1$.  In the last case $\chi$ is non-trivial by assumption, so
$ \chi(g)-1 $ is in $ \Oe^\times $ for some $ g $ in $ G $ as $ l=0 $.
Hence
$ H^1(\O_{k,S},\sqlatmtt{\chi}{1-e})) $ 
is isomorphic with $ H^0(\Gal(K_\infty/k),H^1(\Gal(\OKP/K_{\infty}),\sqlatmtt{\chi}{1-e}))$.
The action of $\Gal(\OKP/K_{\infty})$ on $\sqlatmtt{\chi}{1-e}$ is trivial, so
with $\M$ the Galois group of the maximal Abelian pro-$p$-extension $L_\infty$ of $K_{\infty}$ that is unramified
outside of the primes above $p$ and the infinite primes (so $ L_\infty \subseteq \OKP $),
we can rewrite this last group as
\begin{equation}\label{to_good_form_1}
\begin{aligned}
& \Hom(\Gal(\OKP/K_{\infty}),\sqlatmtt{\chi}{1-e})^{\Gal(K_{\infty}/k)}
\\
\simeq \, & 
 \Hom(\M,\sqlatmtt{\chi}{1-e})^{\Gal(K_{\infty}/k)} 
\\
\simeq \, &
 \Hom(\M\mtt{e-1}\otimes_{\Z_p} \slat{\dual{\chi}},\Q_p/\Z_p)^{\Gal(K_{\infty}/k)}
\\
\simeq \, &
 \Hom \left( \left( \M\mtt{e-1}\otimes_{\Z_p} \slat{\dual{\chi}} \right)_{\Gal(K_{\infty}/k)}, \Q_p/\Z_p \right)
\,,
\end{aligned}
\end{equation}
where $\Hom$ is again the functor of continuous homomorphisms.
Note that the action of $\Gal(K_{\infty}/k)$ on
$\M\mtt{e-1}\otimes_{\Z_p} \slat{\dual{\chi}}$ is diagonal. 
We therefore have
\begin{equation}\label{to_good_form_2}
\begin{aligned}
& \, \left( \M\mtt{e-1}\otimes_{\Z_p} \slat{\dual{\chi}} \right)_{\Gal(K_{\infty}/k)}
\\
\simeq \, &  \left( \left( \M\mtt{e-1}\otimes_{\Z_p} \slat{\dual{\chi}} \right)_{G} \right)_{\Gal(K_\infty/K)} 
\\
\simeq & \, \left( \left( \M\mtt{e-1}\otimes_{\Z_p[G]} \slat{\chi} \right) \right)_{\Gal(K_\infty/K)}
\end{aligned}
\end{equation}
because for a right (resp.\ left) $\Z_p[G]$-module $A$ (resp.\ $B$), with $G$
acting on $A \otimes_{\Z_p} B$ as $g(a \otimes b) = a\cdot g \otimes g^{-1} \cdot b$, we have a natural
isomorphism $(A \otimes B)_G \simeq A \otimes_{\Z_p[G]} B$ as $\Z_p$-modules.

Recall that, since $k$ is totally real and $\chi$ is even, $\M$ is a torsion $\Z_p[G][[T]]$-module with
the action of $T$ given by that of $\gamma_0 - 1$. 
By \cite[Theorem 18]{Iwa73} and \cite[Section 6.4]{schmidt2002}, $ \M $ has no non-trivial finite $ \Z_p[[T]] $-submodule,
hence the same holds for $ \M \otimes_{\Z_p[G]} \Oe[G] \simeq \M^{[E:\Q_p]} $.
Then the last group cannot contain a finite non-trivial $ \Oe[[T]] $-submodule,
and as we may view $ \slat{\chi}$ as a direct summand of $ \Oe[G] $ because $p$ does not divide $\# G$
since $ l=0 $, we find
that $\M \otimes_{\Z_p[G]} \slat{\chi}$ is a torsion $\Oe[[T]]$-module without non-trivial finite submodule.  

Observe that we are now in the situation of Lemma~\ref{isogeny_lemma} with
 $\O = \Oe$,
$Y = \M  \otimes_{\Z_p[G]} \slat{\chi}$, $ \Gamma = \Gal(K_\infty/K) $ with topological generator
corresponding to $ \gamma_0 $, and $u = q_k^{e-1}$.
As defined right after Definition~\ref{structure}, $\mu_\chi$ and $g_\chi$ are
the $ \mu_Y $ and $ g_Y $ in Definition~\ref{structure}
but the action of $ \gamma_0 $ on $ Y $ in Lemma~\ref{isogeny_lemma} for this $u$
is the action on $ \M\mtt{e-1}\otimes_{\Z_p[G]} \slat{\chi} $.

We assume that $L_{p,S}(e,\chi,k) \neq 0$, so by Theorem~\ref{mainc} when $ p \ne2 $
and Assumption~\ref{2-mainc} when $ p=2 $, we get $g_{\chi}(q_k^{1-e} - 1) \neq 0$.
Since $ \mu_\chi = \mu_Y $ and $ g_\chi = g_Y $ in Lemma~\ref{isogeny_lemma},
this lemma states that the group in~\eqref{to_good_form_2} is finite and has cardinality
$ \left| \pi^{\mu_{\chi}} \cdot g_{\chi} (q_k^{1-e} - 1) \right|_p^{-[E:\Q_p]} $, with $ \pi $ a uniformizer
of $ \Oe $.  The same holds for $H^1(\O_{k,S},\sqlatmtt{\chi}{1-e}) $ by~\eqref{to_good_form_1}.
Then $H^2(\O_{k,S},\sqlatmtt{\chi}{1-e})$ is also finite by Proposition~\ref{corankprop},
hence trivial by Proposition~\ref{divprop}.
By Theorem~\ref{mainc} or Assumption~\ref{2-mainc}, and~\eqref{h_0}, we then find, since
$L_{p,S}(e,\chi ,k) = L_p(e,\chi ,k)$,
\begin{alignat*}{1}
     &
   |L_{p,S}(e,\chi ,k)|_p = \frac{| \pi^{\mu_{\chi}} \cdot g_{\chi} (q_k^{1-e} - 1)|_p} {|h_{\chi} (q_k^{1-e} - 1)|_p} 
\\
= \, &
  \left( \frac{\# H^0(\O_{k,S},\sqlatmtt{\chi}{1-e})}{\#  H^1(\O_{k,S},\sqlatmtt{\chi}{1-e})} \right)^{1/[E:\Q_p]}
= 
  \EC_S(e,\chi,k) \,.
\end{alignat*}
This completes the proof of the theorem when $ l=0 $.

We now proceed by induction on $l$, so that we have $[K:k] = m p^l$ with $ (p,m) = 1 $ and $ l \ge 1 $.
Recall that $ k' $ is the subfield of $ K = k_\chi $ with $ [k':k] = p^l $.  Since $\chi$ is even,
$K$ and $k'$ are totally real.
With $\chi' = \chi_{|G_{k'}}$ we have
\begin{eqnarray} \label{induction}
\Ind_{k'}^{k} (\chi') = \sum_{\varphi \in I} \varphi,
\end{eqnarray}
where $I$ is the set of all 1-dimensional characters
$\varphi$ over $ \Qpbar $ of $\Gal(\kbar/k)$ factoring through $\Gal(K/k)$ that restrict to $ \chi' $
on $ G_{k'} $ (hence the quotient of two $ \varphi $'s is a character of $ \Gal(k'/k) \simeq \Z/p^l\Z$).
Note that each $ \varphi $ is a power of $ \chi $, hence is even, takes values in $E$, and $ \Ind_{k'}^{k} (\chi') $
is also unramified outside of $ S $.
Let $N_1 = \Ind_{k'}^{k} (\slat{\chi'}) $ and $N_2 = \oplus_{\varphi \in I} \slat{\varphi} $.
By~\eqref{induction} these two $\Oe[G_k]$-modules of finite $\Oe$-rank correspond to the same representation.
From Remark~\ref{Mremark}(4) we know that
\begin{equation} \label{induction_iso}
H^i(\O_{k,S},N_1 \otimes_{\Z_p} \Q_p/\Z_p\mtt{1-e})
\simeq H^i(\O_{k',S},\sqlatmtt{\chi'}{1-e})
\end{equation}
for each $i \geq 0$.

By our assumptions, $L_{p,S}(e,\chi',k') $ is defined and non-zero, so
by the case $l=0$ the theorem is true for $ \chi' $ and $ k' $.  Therefore
the group in \eqref{induction_iso} is finite for $i = 0,1$, trivial for $i=2 $, and
\begin{equation*}
\EC_S(e,\chi',k') = | L_{p,S}(e,\chi',k') |_p
\,.
\end{equation*}
From Remark~\ref{independence}(1) we conclude that
$H^i(\O_{k,S},N_2 \otimes_{\Z_p} \Q_p/\Z_p\mtt{1-e})$ is finite for $ i = 0,1,2 $, hence trivial for
$ i = 2 $ by Proposition~\ref{divprop}.  The same statement therefore holds for each summand $H^i(\O_{k,S},\sqlatmtt{\varphi}{1-e})$.
From parts~(3) and~(4) of Remark~\ref{independence} we obtain
\begin{equation}\label{ECLidentity}
\begin{aligned}
   & \, \prod_{\varphi \in I} \EC_S(e,\varphi,k)
 = \EC_S(e,\chi',k')
\\
 = & \, | L_{p,S}(e,\chi' ,k') |_p
 =  \prod_{\varphi \in I} | L_{p,S}(e,\varphi,k) |_p
\,,
\end{aligned}
\end{equation}
with the last equality by~\eqref{induction} and the compatibility of $L_p$ and $\dEul{v}$ with induction (see
Remark~\ref{deuler-induction}).
Note that the $ p $-adic $ L $-functions involved here are of Abelian characters so cannot have a pole
except possibly at~$ 1 $, hence by our assumptions on $e$ and $\chi'$,
$ L_{p,S}(e,\varphi,k)$ exists and is non-zero for every~$ \varphi $.

Let $I_0$ denote the subset of $I$ consisting of all $ \varphi $ with order divisible by~$ p^l $.  
If $\varphi \in I \setminus I_0$ then $[k_\varphi:k]$ divides $p^{l-1}m$, hence by the induction hypothesis
the theorem is true for $\varphi$.  Thus from~\eqref{ECLidentity} we obtain
\begin{equation}\label{I0id}
 \prod_{\varphi \in I_0} |L_{p,S}(e,\varphi,k)|_p = \prod_{\varphi \in I_0} \EC_S(e,\varphi,k)
 \,.
\end{equation}
Using Remark~\ref{independence}(2) we may replace $ E $ with its Galois closure over $ \Q_p $ if necessary.
All $ \varphi $ in $ I_0 $ are
in the orbit of $ \chi $ under the action of $ \Gal(E/\Q_p) $ since this group acts transitively on the
roots of unity of order $ p^l $.  But for $\sigma $ in $ \Gal(E/\Q_p)$ we have
$L_{p,S}(e,\chi^{\sigma},k) = L_{p,S}(e,\chi,k)^{\sigma}$ by~\cite[p.413]{BBdJR}, hence
$ |L_{p,S}(e,\chi^{\sigma},k) |_p = | L_{p,S}(e,\chi,k) |_p $.
On the other hand, $\sigma$ induces an
isomorphism $\slat{\chi} \simeq \slat{\chi^{\sigma}}$ of sheaves for
the \'etale topology on $\O_{k,S}$, so that
$\EC_S(e,\chi^{\sigma},k) = \EC_S(e,\chi,k)$.  Taking roots of both sides in~\eqref{I0id} we find
$ |L_{p,S}(e,\chi,k)|_p = \EC_S(e,\chi,k)$.
This finishes the proof of the induction step, and of the theorem.
\end{proof}

We now extend the statements of Theorem~\ref{ebn} to
arbitrary Artin characters for all but finitely many $ e $, which
is the last main ingredient needed for the proof of Theorem~\ref{introtheorem}.

\begin{theorem} \label{ebn2}
Let $ k $ be a totally real number field, $ p $ a prime number, $ E $ a finite extension of $ \Q_p $,
and $ \chi : G_k \to E $ an even Artin character realizable over $ E $.
Assume that $ S $ contains $ P $ as well as all the finite primes of
$ k $ where $ \chi $ is ramified.  The following hold for all but finitely many $ e $
in $ \B_k(E) $ where for part~(2) we make Assumption~\ref{2-mainc} if $ p = 2 $.
\begin{enumerate}
\item
$H^i(\O_{k,S},\qlatmtt{E}{\chi}{1-e})$
is finite for $i = 0, 1$ and trivial for $i = 2$.

\item
$ |L_{p,S}(e,\chi,k)|_p  = \EC_S(e,\chi,k) $.
\end{enumerate}
\end{theorem}

\begin{proof}
Part~(1) follows from Lemma~\ref{finitelemma} and Proposition~\ref{divprop}.  For part~(2),
note that we may enlarge $ E $ if we want because of Remark~\ref{independence}(2).
Using Brauer's Theorem we can then write 
\begin{equation*}
\chi+ \sum_i \Ind_{k_i}^k(\varphi_i) = \sum_j \Ind_{k_j'}^k(\varphi_j')
\end{equation*}
for 1-dimensional $ E $-valued Artin characters $ \varphi_i $ (resp.\ $ \varphi_j' $)
of $G_{k_i}$ (resp.\ $ G_{k_j'} $) with all $ k_i $, $ k_j' $ intermediate fields of the extension
$ k_\chi/k $,
so that $G_{k_{\chi}}$ is in the kernel of all $\varphi_i$ and $\varphi_j'$. 
Since $k_\chi$ is totally real, so are all $k_i$ and $k_j'$, and all
$ \varphi_i $ and $ \varphi_j' $ are even.
Because $ \B_k(E)$ is contained in all $\B_{k_i}(E) $ and $\B_{k_j'}(E)$,
we may apply Theorem~\ref{ebn} to each pair $ (\varphi_i,k_i) $ or $ (\varphi_j',k_j') $
for all but finitely many $ e $ in $ \B_k(E) $.
For all such $ e $, parts~(3) and~(4) of Remark~\ref{independence} give that
\begin{equation*}
 \EC_S(e,\chi,k) \prod_i \EC_S(e,\varphi_i,k_i) = \prod_j \EC_S(e,\varphi_j',k_j') 
\,,
\end{equation*}
whereas Remark~\ref{deuler-induction} implies
\begin{equation*}
L_{p,S}(e,\chi,k) \prod_i L_{p,S}(e,\varphi_i,k_i) = \prod_j L_{p,S}(e,\varphi_j',k_j') 
\,,
\end{equation*}
so that $ \EC_S(e,\chi,k) = | L_{p,S}(e,\chi,k) |_p $.
\end{proof}

We can now give the proof of Theorem~\ref{introtheorem}.
We make extensive use of a counting argument while approximating a
fixed $ e $ with suitable $ e' $.
In particular, this method enables us to treat the finitely many $ e $ excluded
in Theorem~\ref{ebn2}.

\begin{proof}[Proof of Theorem~\ref{introtheorem}]
The first statement is Theorem~\ref{ebn2}(1).  For the remaining statements we 
make Assumption~\ref{2-mainc} if $p=2$.

Let $r_i = r_{i,S}(1-e,\chi) = \corank_{\Oe} H^i(\O_{k,S},\qlatmtt{E}{\chi}{1-e})$.
By Theorem~\ref{cogentheorem}(3), $ \cork{0}{1-e}{\chi} $ is zero if $e \ne 1$
and equals the multiplicity of the trivial character in $ \chi $ otherwise.
Since $ r_2 \ge0 $, the statements about
analyticity and poles of $ L_{p,S}(s,\chi,k) $ in part~(2) follow from
the inequality $ r_2 - r_0 \le \nu $ in part~(1) and Remark~\ref{analytic-remark}.
The claims in part~(2) for $L_p(s,\chi,k)$ then follow from this
in turn by using Remark~\ref{analytic-remark}, \eqref{pLS},
Remark~\ref{deulzero}, and
Proposition~\ref{OPprop} (note that for $ e=0 $ we have $ r_0=0 $
and $ r_1=r_2 $ by parts~(2) and~(3) of Theorem~\ref{cogentheorem}).
It thus suffices to prove parts~(1) and~(3).

For the proof of~(1), let $ r= r_2-r_0 $.
We shall prove the inequalities $ \min(1-r_0,\nu) \le r \le \nu $
in~(1) when $ e=1 $ and $ \chi $ contains the trivial character later.
So let us fix $ e $ in $ \B_\chi(E) $, hence
$r_0 = 0$ by Theorem~\ref{cogentheorem}(3) and therefore $r = r_1 = r_2$ by Proposition~\ref{corankprop}.
Let $X$ be a finite $\Oe$-module of size $c_1$ such that
$ H^1(\O_{k,S}, \qlatmtt{E}{\chi}{1-e})^\vee \simeq \Oe^{r} \times X $ and
hence $  H^1(\O_{k,S}, \qlatmtt{E}{\chi}{1-e}) \simeq (E/\Oe)^{r} \times X^\vee $
as $ \Oe $-modules.
Using Theorem~\ref{cogentheorem}(3) for condition~(b) below, and Theorem~\ref{ebn2} for condition~(c),
we now only consider $ n>0 $ such that
\begin{enumerate}
\item[(a)]
$ c_1 $ divides $ p^n $;

\item[(b)]
$ \# H^0(\O_{k,S}, \qlatmtt{E}{\chi}{1-e'}) $ equals a constant $c_0$ if $ |e'-e|_p \leq p^{-n} $,
and $ c_0 $ divides $ p^n $;

\item[(c)]
if $ 0< |e'-e|_p \le p^{-n} $ then $e'$ is in $\B_\chi(E)$, $H^i(\O_{k,S},\qlatmtt{E}{\chi}{1-e'})$
is finite for $i=0,1,2$, $L_{p,S}(e',\chi,k) \ne 0$, and $ |L_{p,S}(e',\chi,k)|_p  = \EC_S(e',\chi,k)$.
\end{enumerate}
Recall that, for $a \ne 0 $ in $\Oe$,
$ \klat{a}{E}{\chi} =  \ker(\qlat{E}{\chi} \stackrel{a}{\to} \qlat{E}{\chi}) $.
For $ |e'-e|_p \le p^{-n}  $ the long exact sequence associated to \eqref{texact} gives
\begin{equation*}
\begin{aligned}
 \# H^1(\O_{k,S}, \qlatmtt{E}{\chi}{1-e'})[p^n]
=
 \frac{\# H^1(\O_{k,S}, \klatmtt{p^n}{E}{\chi}{1-e'})}{\# H^0(\O_{k,S}, \qlatmtt{E}{\chi}{1-e'})}
\,.
\end{aligned}
\end{equation*}
For $|e' - e|_p \le p^{-n}$ we have $ \klatmtt{p^n}{E}{\chi}{1-e'}\simeq \klatmtt{p^n}{E}{\chi}{1-e}$
as $ G_k $-modules, hence
\begin{equation}\label{eq0}
\begin{aligned}
      & \# H^1(\O_{k,S}, \qlatmtt{E}{\chi}{1-e'}) [p^n]
\\
 = \, & 
    \#  H^1(\O_{k,S}, \qlatmtt{E}{\chi}{1-e}) [p^n] 
\\
= \, &
         c_1 p^{nr[E:\Q_p]}
\,,
\end{aligned}
\end{equation}
so if also $ e'\ne e $ then by~(c) and~(b) we find
\begin{equation}\label{eq1}
\begin{aligned}
& \, 
  | L_{p,S}(e',\chi,k) |_p^{[E:\Q_p]}
=
 \EC_S(e',\chi,k)^{[E:\Q_p]} 
\\
= & \, 
 \frac{\# H^0(\O_{k,S}, \qlatmtt{E}{\chi}{1-e'})}{\# H^1(\O_{k,S}, \qlatmtt{E}{\chi}{1-e'})} 
\le
 c_0 c_1^{-1} p^{-nr[E:\Q_p]}
\,.
\end{aligned}
\end{equation}
For $ n \gg 0 $ the left-hand side equals $ c |e'-e|_p^{\nu [E:\Q_p]} $ for some $ c \ne 0 $ and so by
choosing $e'$ with $|e' -e|_p = p^{-n}$ we get $ \nu \ge r $.

For proving that $ \min(1,\nu) \le r $ we may take
$ r=0 $.  For $ |e'-e|_p \le p^{-n-1} $ we see
from~\eqref{eq0} that the part of $ H^1(\O_{k,S}, \qlatmtt{E}{\chi}{1-e'}) $ annihilated
by $ p^n $ is the same as that annihilated by $ p^{n+1} $.
Hence this is the entire group,
and $ \# H^1(\O_{k,S}, \qlatmtt{E}{\chi}{1-e'}) = c_1 $ for such $ e' $.
From~\eqref{eq1} we then see that $ | L_{p,S}(e',\chi,k) |_p^{[E:\Q_p]}  = c_0/c_1  $
for $ e'\ne e $ close enough to $ e $, so by continuity we have
$ | L_{p,S}(e,\chi,k) |_p^{[E:\Q_p]}  = c_0/c_1  \ne 0 $ and $ \nu=0 $.

We now assume $e=1$ and $\chi$ contains the trivial character $\chi_0$.
Note that the two sides of the inequality $r(e,\chi) \le \nu(e,\chi)$ are additive in $\chi$, therefore
it is enough to prove this for $\chi = \chi_0$.  In this case we may take $E = \Q_p$ and $M = \Z_p$.
Moreover, we may take $S=P$ by Theorem~\ref{cogentheorem}(2) because $ \dEul{v}(1,\chi_0 \op^{-1},k) \ne 0 $.
Clearly $H^0(\O_{k,P},\Q_p/\Z_p) = \Q_p/\Z_p$, hence $\cork{0}{0}{\chi_0} = 1$.
On the other hand, by the discussion after~\eqref{fiveterm} and~\eqref{to_good_form_1} we have
$\cork{1}{0}{\chi_0} = 1 + \rank_{\Z_p} \M$, where $\M$ is the Galois group of the maximal extension of $k_\infty$
unramified outside of the primes above $p$ and infinite primes. 
Since $\M$ is isogenous to
$ \bigoplus_{i=1}^r \Z_p/\pi^{m_i} \oplus \bigoplus_{j=1}^s \Z_p/(g_j(0))$,
we find $\rank_{\Z_p} \M \leq \ord_{T=0} (\tg_{\chi_0}(T))$
by Theorem~\ref{mainc} if $p \ne 2$, and Assumption~\ref{2-mainc} if $p = 2$.
As $h_{\chi_0}(T) = T$,
$\ord_{T=0} (\tg_{\chi_0}(T)) = \nu(0,\chi_0) + 1$.
Thus we have  $\cork{1}{0}{\chi_0} \le \nu(0,\chi_0) + 2$ and $\cork{2}{0}{\chi_0} \le \nu(0,\chi_0) + 1$
by Theorem~\ref{cogentheorem}(2).  Thus $r(0,\chi_0) \le \nu(0,\chi_0)$.  This proves
the inequality $r \le \nu $ in all the cases.

The inequality $\min(1-r_0,\nu) \le r = r_2 - r_0$ is trivial unless $r_2 = 0$.
Write $\chi = \chi' + s\chi_0$ with $\chi'$ not containing the trivial character and $s \ge 1$, then
$\cork{2}{0}{\chi} = \cork{2}{0}{\chi'} +s\cork{2}{0}{\chi_0}$.  This is zero only if
$ \cork{2}{0}{\chi'} = \cork{2}{0}{\chi_0} = 0$.  We have just proved Theorem~\ref{introtheorem}(1)
for $\chi'$, so using Theorem~\ref{cogentheorem}(3) for $\chi'$ we see $L_p(1,\chi',k) \ne 0$.
Moreover, $\cork{2}{0}{\chi_0} = 0$ implies that $\cork{0}{0}{\chi_0} = \cork{1}{0}{\chi_0} = 1$ by
parts~(2) and~(3) of Theorem~\ref{cogentheorem}.  Hence $\rank_{\Z_p} \M = 0$, $\tg_{\chi_0}(0) \ne 0$
and therefore $ \nu(0,\chi_0) = -1$.  This proves $ \nu(0,\chi) = -s$ and $\min(1-r_0,\nu) = -s = r_2 - r_0 $.
This completes the proof of part~(1) of the theorem.

For part~(3), we fix $ e $ in $ \B_\chi(E) $,
so $ r_0 = 0 $ by Theorem~\ref{cogentheorem}(3).
Then the equivalences follow from the case $r_2 = 0$ in part~(1)
and Proposition~\ref{divprop}.
For the very last formula, we observe that this is covered by the case $r=0$ in the paragraph
following~\eqref{eq1} since $ \# H^i(\O_{k,S}, \qlatmtt{E}{\chi}{1-e}) = c_i $ for $ i=0,1 $ and
$r_1 = r_2$ by Theorem~\ref{cogentheorem}(2).
This concludes the proof of Theorem~\ref{introtheorem}.
\end{proof}

Note that if the equality
\begin{equation} \label{ordequality}
 \nu_S(1-e,\chi) = \cork{2,S}{1-e}{\chi} - \cork{0,S}{1-e}{\chi} 
\end{equation}
holds in Theorem~\ref{introtheorem}(1) for some $ S $ then it holds
for every choice of $S$ by Proposition~\ref{gysinprop}(2), Theorem~\ref{cogentheorem}(2) and Remark~\ref{deulzero}.
If it holds for all totally real $k$ and all
$1 $-dimensional even Artin characters $\chi : G_k \rightarrow E$ with
$[k_\chi:k]$ not divisible by $ p $, then it holds for all even Artin characters of $G_k$ for all totally real $k$.
This is because such an equality is
preserved under the induction argument in the proof of
Theorem~\ref{ebn} and also under the arguments involving Brauer
induction in the proof of Theorem~\ref{ebn2}.  Since
$ \cork{0,S}{1-e}{\chi} $ is independent of $ S $ as stated
in Theorem~\ref{cogentheorem}(2), we see from Proposition~\ref{gysinprop}(2) and Remark~\ref{deulzero},
the proof of the case $l=0$ of Theorem~\ref{ebn}, and the
proof of Lemma~\ref{isogeny_lemma}, that the equality in~\eqref{ordequality}
for such an Artin character $ \chi $ is equivalent to all distinguished polynomials
$g_j(T) $ in $\Oe[T]$ having order at most~1 at $ T=q_k^{e-1} - 1 $, where
$ Y =\M \otimes_{\Z_p[G]} \slat{\chi} $ as in the application of
Lemma~\ref{isogeny_lemma} in the proof of Theorem~\ref{ebn} is isogenous to
$ \bigoplus_{i=1}^r \frac{\O[[T]]}{(\pi^{\mu_i})} \oplus \bigoplus_{j=1}^s \frac{\O[[T]]}{(g_j(T))}$.
We therefore formulate the following (partly folklore) conjecture.

\begin{conjecture} \label{conjecture}
Let $ k $ be a totally real number field, $ p $ a prime number, $ E $ a finite extension of $ \Q_p $,
and $ \chi : G_k \to E $ an even Artin character realizable over $ E $.  If $\chi$ is $1$-dimensional
then let $\slat{\chi} = \Oe$ on which $G_k$ acts through $\chi$.
Assume that $ S $ contains  the primes above $ p $ as well as the finite primes of $ k $ at which $ k_\chi/k $ is ramified.
Let $\M$ be the Galois group of the maximal Abelian extension of $(k_\chi)_\infty$ that is unramified outside
of the primes above $p$ and infinity.
Then for $ e $ in $ \B_k(E) $ we have the following.
\begin{enumerate}
\item
If $ \chi $ is 1-dimensional of order prime to $ p $, then the $ g_j $ 
corresponding to
$ Y =  \M \otimes_{\Z_p[G]} \slat{\chi} $ are square-free.

\item
If $ \chi $ is 1-dimensional of order prime to $ p $, then
the $ g_j $
corresponding to
$ Y =  \M \otimes_{\Z_p[G]} \slat{\chi} $
have no multiple roots $ z $ in $E$ for which $ |z-1|_p < p^{-1/(p-1)} $.

\item
If $ \chi $ is 1-dimensional of order prime to $ p $, then equality
holds in~\eqref{ordequality}.

\item
Equality holds in~\eqref{ordequality}.
\end{enumerate}
\end{conjecture}

For the same $ k $, $ e $ and $ \chi $, (1) implies (2),
and, as discussed before the statement
of the conjecture, (2) is equivalent to (3).
As also implied by the discussion there, if we fix $ e $, then
(3) for all $ k $, $ \chi $ and $ E $ implies (4), and the converse
is clear.

\begin{remark} \label{sisone}
In this remark we make Assumption~\ref{2-mainc} if $ p=2 $.

(1)
If $e=1$ and $\chi$ is the trivial character then parts~(2) and~(3)
of Theorem~\ref{cogentheorem} and Theorem~\ref{introtheorem}(1)
imply the following are equivalent:
\begin{enumerate}
\item[(a)] $\zeta_p(s,k)$ has a simple pole at $s=1$;
\item[(b)] $\corank_{\Z_p} H^1(\O_{k,S},\Q_p/\Z_p) = 1$;
\item[(c)] $H^2(\O_{k,S},\Q_p/\Z_p) = 0$.
\end{enumerate}
Those statements are all equivalent with the Leopoldt conjecture
for $ k $ (see \cite[Theorem~10.3.6]{NSW2nded}).

Clearly, an equality $ \nu_S(0,\chi) = - \cork{0,S}{0}{\chi} $ with
$ \chi $ the trivial character implies~(a) by Theorem~\ref{cogentheorem}(3), hence the Leopoldt
conjecture for~$ k $.  For other characters, 
recall that Brauer's theorem
$ \chi = \sum_i a_i \Ind_{H_i}^G (\chi_i) $
for some non-zero integers $ a_i $ and 1-dimensional characters $ \chi_i $ of certain solvable
subgroups $ H_i $.
If $H_i$ is nontrivial and $\chi_i = 1_H$, then
$ \Ind_{[H_i,H_i]}^G (1_{[H_i,H_i]}) = \Ind_{H_i}^G(1_{H_i}) + \sum_j \Ind_{H_i}^G(\chi_j) $
where the sum runs through the non-trivial 1-dimensional characters of $ H_i $.
Repeating this process if necessary,
we may assume that $\chi_i$ is trivial only if $H_i$ is trivial.
In that case, Frobenius reciprocity for $ 1_G $ implies that 
if $ \chi $ does not contain the trivial character of $ G $, then 
all $ \chi_i $ are non-trivial.
Therefore the Leopoldt conjecture for all totally real number fields
implies $L_p(s,\chi,k)$, for an even character $\chi$ of $ G_k $ 
not containing the trivial character, is defined and non-zero at $s=1$, so
by Theorem~\ref{cogentheorem}(3) and Theorem~\ref{introtheorem}(3),
$\nu_S(0,\chi) = - \cork{0,S}{0}{\chi} = 0$.  Since the Leopoldt
conjecture for $ k $
implies $\nu_S(0,\chi) = - \cork{0,S}{0}{\chi} $ if $ \chi $ is the trivial
character, the Leopoldt conjecture for all totally real $ k $
implies equality for all $ k $ and all even Artin characters
$ \chi $ of $ G_k $.

(2)
For $e$ in $B_\chi(E)$, $H^1(\O_{k,S},\qlatmtt{E}{\chi}{1-e})$ is
finite if and only if $ L_{p,S}(e,\chi,k) $ is defined and non-zero
by Theorem~\ref{introtheorem}(3).
If $ \chi $ contains the trivial character and $ e=1 $, then
$H^1(\O_{k,S},\qlatmtt{E}{\chi}{1-e})$ is infinite by parts~(2)
and~(3) of Theorem~\ref{cogentheorem}.
The Leopoldt conjecture (for all totally real $ k $) would imply that
then $ L_{p,S}(e,\chi,k) $ should not be defined at $ e=1 $,
extending this equivalence to $ e $ in $ B_k(E) $.
\end{remark}

To conclude this section, we briefly discuss the case of $ 1 $-dimensional~$ \chi $.

\begin{example}\label{Hexample}
With notation as in Theorem~\ref{introtheorem}, suppose that $k$ is totally real, $\chi$ is
a 1-dimensional even Artin character, and $e$ is in $\B_\chi(E)$ with $L_p(e,\chi,k) \ne 0$.
Then $h_\chi(T) = 1$ unless $k_\chi \subseteq k_\infty$, in which case $h_\chi(T) = \chi(\gamma_0)(T+1) - 1$.
Note that in the latter case
\begin{alignat*}{1}
H^0(\O_{k,S},\qlatmtt{E}{\chi}{1-e}) =
\ker \left( \qlatmtt{E}{\chi}{1-e}\stackrel{1 - \gamma_0}{\longrightarrow}
\qlatmtt{E}{\chi}{1-e}\right)
\,.
\end{alignat*}
Since $\gamma_0$ acts on $\qlatmtt{E}{\chi}{1-e}$ as multiplication by $q_k^{1-e} \chi(\gamma_0)$
in this case, we have 
\begin{equation*}
\# H^0(\O_{k,S},\qlatmtt{E}{\chi}{1-e}) = |h_\chi(q_k^{1-e}-1)|_p^{-[E:\Q_p]}
\,.
\end{equation*}
Hence for any 1-dimensional $\chi$ we always have
\begin{equation}\label{preciseH0}
 \# H^0(\O_{k,S},\qlatmtt{E}{\chi}{1-e})) \ge |h_\chi(q_k^{1-e}-1)|_p^{-[E:\Q_p]} 
 \,.
\end{equation}
So by~\eqref{pLpowerseries} and Theorem~\ref{introtheorem}(3) we also have
\begin{equation}\label{preciseH1}
 \# H^1(\O_{k,S},\qlatmtt{E}{\chi}{1-e})) \ge |\pi^{\mu_\chi} \tg_\chi(q_k^{1-e}-1)|_p^{-[E:\Q_p]} 
\,.
\end{equation}
We distinguish three cases, with the third motivating why one needs to use $ \EC_S(e,\chi,k) $
rather than the individual cohomology groups (and to some extent necessitating the complexity of our proofs).

{\rm(1)}
$ [k_\chi:k] $ is divisible by a prime other than $ p $: 
then $ h_\chi(T) = 1 $ and
$ H^0 $ is trivial because it is contained in the kernel of multiplication by $1-\xi_q$ where $ \xi_q $
is a root of unity of order a prime $ q \ne p $.
Therefore both entries in~\eqref{preciseH0} are~1, and equality holds in both~\eqref{preciseH0}
and~\eqref{preciseH1} because of Theorem~\ref{introtheorem}(3).

{\rm(2)}
$ [k_\chi:k] $ is a power of $ p $ and $ k_\chi \subseteq k_\infty $.
As already mentioned, equality holds in~\eqref{preciseH0}, hence it also holds in~\eqref{preciseH1},
but in~\eqref{preciseH0} both sides are bigger than~1.

{\rm(3)}
$ [k_\chi:k] $ is a power of $ p $ but $ k_\chi \nsubseteq k_\infty$.
In this case $ h_\chi(T) = 1 $.  Note that any $g$ in $G$ acts on $\qlatmtt{E}{\chi}{1-e}\simeq E/\Oe
\mtt{1-e}$ as multiplication by an element of $\Oe^\times$ that reduces to $1$ in the residue field of $ \Oe $,
hence $H^0(\O_{k,S}, \qlatmtt{E}{\chi}{1-e})$
is non-trivial, and the inequalities in~\eqref{preciseH0} and~\eqref{preciseH1} are strict.
\end{example}

\section{Numerical examples}\label{secexamples}

The approximations $ \tPol (T) $ of the distinguished polynomials $ P(T) $ and the triviality of the $ \mu $-invariants
in the following examples were kindly provided to us by X.-F. Roblot. 
In all cases we have a totally real field $ K $ such that $ K/\Q $ is Galois with dihedral
Galois group $ G $ of order~8.   If $ k $ is the fixed field of the cyclic subgroup $ H $ of order~4 then we consider
$ L_5(s) :=  L_{5}(s,\chi,k) = L_5(s,\Ind_k^\Q(\chi),\Q) $
where $ \chi $ is a 1-dimensional character of order 4 of $ H $ with values in $ \mu_4 \subset \Q_5 $.
Note that $ \Ind_{k}^\Q (\chi) $ is the irreducible 2-dimensional character of $ \Gal(K/\Q) $ for either
possibility for $ \chi $, so $ L_5(s) $ is the same for either $ \chi $.

With notation as in~\eqref{pLpowerseries} we have $ h_\chi(T)=1 $, and used $ q_k= 1+5 $.
In all cases below Roblot
found that
$ m_\chi = 0 $, so $ L_5(s) = P(6^{1-s}-1) u(6^{1-s}-1) $
for $ s $ in $ \B_k = \{ s \text{ in } \C_p \text{ with } |s|_5 < 5^{3/4}\} $,
where $ u(T) $ is in $ \Z_5[[T]]^\times $ and
$ P(T) =  \tg_{\chi}(T) $ is a distinguished polynomial in $ \Z_5[T] $.
So, in particular, $ | L_5(s) |_5 =  | P(6^{1-s}-1) |_5 $.

Note that if $ S $ is a finite set of primes of $ \Q $ that includes $ 5 $
and the primes where $ K/\Q $ is ramified, then
by Remark~\ref{Mremark}(3) we have, for $ E/\Q_5 $ a finite extension and $ e $ in $ \B_\chi(E) = \B_k(E) $,
with the appropriate choice of lattices,
\begin{equation*}
H^i(\O_{\Q,S},\qlatmtt{E}{\Ind_{k}^\Q (\chi)}{1-e})\simeq H^i(\O_{k,S}, \qlatmtt{E}{\chi}{1-e})
\,.
\end{equation*}
We recall from Remark~\ref{coeffremark} that the coefficients in the right-hand side are
unique but that this does not necessarily hold in the left-hand side.

In the examples, $ K/k $ is unramified outside of the primes
above 5. We discuss $  H^i(\O_{k,S}, \qlatmtt{E}{\chi}{1-e}) $ only when 
$ S $ consists of those primes.
The sizes of those groups for larger $ S $ can be calculated similarly by taking into account various
$ | \dEul{v}(e,\chi\o_5^{-1},k) |_5 $ as in~\eqref{pLS}.  We leave this to the interested reader.

\begin{example}
$ k = \Q(\sqrt{145}) $, $ K $ is the Hilbert class field of $ k $, where $ P(T)=1 $.
Taking $ e $ in $ \B_k(\Q_5) = \Z_5 $,
one sees from Case~(1) of Example~\ref{Hexample}
that $ H^i(\O_{k,S}, \qlatmtt{E}{\chi)}{1-e}) $ with $ E=\Q_5 $
is trivial for $ i=0 $, hence the same holds for $ i=1 $ and $ i=2 $
by Theorem~\ref{introtheorem}.
In fact, the same statements are true if we use any finite extension $ E $
of $ \Q_5 $ and take $ e $ in $ \B_k(E) $.
\end{example}

\begin{example}\label{2ndexample}
$ k = \Q(\sqrt{41})$, $ K $ is the ray class field of $ k $ modulo 5.  Here
$ \tPol(T)=T-(5+\varepsilon) $ with $ |\varepsilon|_5 \le 5^{-10} $.
But $ \Eul{P}^\ast(s, \chi\o_5^{-1}, k) = 1-5^{-s} $, so that $  L_5 (0,\chi, k) = 0 $
by~\eqref{interpol}, and $ P(T) = T - 5 $.
Fix a finite extension $ E $ of $ \Q_5 $ and $ e $ in $ \B_k(E) $.
Again by Example~\ref{Hexample} we have that $ H^0(\O_{k,S}, \qlatmtt{E}{\chi)}{1-e}) = 0 $.
Taking $ e \ne 0 $, so that $ L_5(e) \ne 0 $, by Theorem~\ref{ebn} we have
$ \# H^1(\O_{k,S}, \qlatmtt{E}{\chi)}{1-e}) = \# \Oe/(6^{1-e}-6) = \# \Oe/(6^{-e}-1) = \# \Oe/(5e) $.

In fact, any $ \Z_5[[T]] $-submodule of finite index in $ \Z_5[[T]]/(T-5) \simeq \Z_5 $ is isomorphic with $ \Z_5 $
as $ \Z_5[[T]] $-module, so that, in the proof of Theorem~\ref{ebn} for $ l=0 $,
the part of $ \M $ on which $ H $ acts through $ \chi $ is $ \M \otimes_{\Z_5[H]}\Z_5(\chi) \simeq \Z_5 $
with the action of $ \gamma_0 $ given by multiplication by $ 6 $.
So by the proof of Lemma~\ref{isogeny_lemma} (in particular~\eqref{b_to_g}) we have
$ H^1(\O_{k,S}, \qlatmtt{E}{\chi)}{1-e}) \simeq \Oe/(5e) $, also
when $ e=0 $.
\end{example}

\begin{remark}\label{funny}
Note that the calculation in Example~\ref{2ndexample} for each of the two possibilities for $ \chi $
implies the existence of $ L_\chi/K_\infty $, such that $ L_\chi/k $ is Galois with Galois group isomorphic to
$ (H\times (1+5\Z_5)) \ltimes \Z_5 $, where the action of $ (h,u) $ on $ \Z_5 $ (the part of $ \M $ on
which $ H $ acts through $ \chi $) is given by multiplication
by $ \chi(h) u $.

Working over $ \Q $ we find that the two $ L_\chi $ together
give rise to
extensions $ L_\infty/K_\infty/\Q $
with $ \Gal(L_\infty/\Q) $ isomorphic to
$ (G\times (1+5\Z_5)) \ltimes \Z_5^2 $ where $ G $ acts on $ \Z_5^2 $ as on the irreducible $ 2 $-dimensional
representation of $ G $, and $ 1+5\Z_5 $ by multiplication.
\end{remark}

\begin{example}\label{3rdexample}
$ k = \Q(\sqrt{793})$, $ K $ is the Hilbert class field of $ k $, and
\begin{equation*}
 \tPol(T) =  T^2 + ((0.341430342)_5 + \varepsilon_1) \, T + ((0.103034211)_5 +  \varepsilon_2)\,,
\end{equation*}
with $ |\varepsilon_i|_5 \le 5^{-10} $.
Here $ (0.a_1\cdots a_m)_5 = \sum_{i=1}^m a_i 5^m $.
Therefore $ P(T) $ is Eisenstein, hence is irreducible over $ \Z_5 $ and $ |L_5(s) |_5 = 1 $ for $ s $ in
$ \Z_5 $.  But $ \Z_5[[T]]/(P(T)) \simeq \Z_5[\sqrt5] = \O_{\Q_5(\sqrt5)} $ so any $ \Z_5[[T]] $-submodule
of finite index is isomorphic to $ \Z_5[[T]]/(P(T)) $ with $ \gamma_0 $ acting as multiplication by $ 1+T $.
As in Example~\ref{2ndexample} we find that
$ H^1(\O_{k,S}, \qlatmtt{E}{\chi)}{1-e}) \simeq \Oe/(P(6^{1-e}-1)) $
for each finite extension $ E $ of $ \Q_5 $ and $ e $ in $ \B_k(E) $ as
$ H^0(\O_{k,S}, \qlatmtt{E}{\chi)}{1-e}) $	is trivial.
Unless $ | 6^{1-e}-1 - \alpha|_5 \leq 5^{-19/2} $ for a root
\begin{equation*}
\alpha= (0.104202323)_5 \pm (3.41423 114)_5 \sqrt5 + \varepsilon 
\end{equation*} 
of $ P(T) $ (with $ |\varepsilon|_5 \le 5^{-19/2} $), we may replace $ P(6^{1-e}-1) $ with $ \tPol(6^{1-e}-1) $.
\end{example}

\begin{remark}\label{funny2}
Just as in Remark~\ref{funny}, the calculation in Example~\ref{3rdexample} for each of the two possibilities
implies the existence of $ L_\chi/K_\infty $, such that $ L_\chi/k $ is Galois with Galois group isomorphic to
$ (H\times (1+5\Z_5)) \ltimes \Z_5[\sqrt{5}] $,
where the action of $ (h,u) $ on $ \Z_5[\sqrt{5}] $ (the part of $ \M $ on
which $ H $ acts through $ \chi $) is given by multiplication
by $ \chi(h) (1 + \alpha)^{\log(u)/\log(6)} $
for $ \alpha $ in $ \Z_5[\sqrt5] $ a root of $ P(T) $. 

Working over $ \Q $ we find that the two $ L_\chi $ together
give rise to
extensions $ L_\infty/K_\infty/\Q $
with $ \Gal(L_\infty/\Q) $ isomorphic to
$ (G\times (1+5\Z_5)) \ltimes \Z_5[\sqrt5]^2  $ where $ G $ acts as on $ \Z_5[\sqrt{5}]^2 $ as on the irreducible $ 2 $-dimensional
representation of $ G $, and $ u $ in $ 1+5\Z_5 $ as multiplication by $ (1 + \alpha)^{\log(u)/\log(6)} $
for $ \alpha $ a root of $ P(T) $ in $ \Z_5[\sqrt5] $.
\end{remark}

\section{The equivariant Tamagawa number conjecture} \label{etnc_section}

Let $k$ be a number field, $S$ a finite set of finite primes in $k$, $\Sig$ the set of infinite places
of $k$ and $G_S$ the Galois group of the maximal extension of $k$ that is unramified outside of $S \cup \Sig$.
For any place $v$ of $k$ let $G_{w_v}$ be the decomposition group in $G_k$
of a prime $ w_v $ of $ \kbar $ lying above $ v $.
For a finitely generated $\Z_p$-module $A$ with continuous $G_S$-action, define
\begin{equation*}
R\Gamma_c(\O_{k,S},A) := \mathrm{Cone} \left( C^{\bullet}(G_S,A) \rightarrow \oplus_{v \in S \cup \Sig}
C^{\bullet}(G_{w_v},A) \right)[-1],
\end{equation*}
where $C^\bullet$ denotes the standard complex of (continuous) cochains, and
the morphism is induced by the natural maps $G_{w_v} \rightarrow G_k \rightarrow G_S$.
We denote the cohomology of $R\Gamma_c(\O_{k,S},A)$ by $\H_c^\ast(\O_{k,S},A)$.
Applying inverse limits to the resulting long exact sequence of cohomology groups
with coefficients $A/p^nA$
is exact because the $H^\ast(G_S,A/p^n A)$ and $H^\ast(G_{w_v},A/p^n A)$ are finite.
Therefore
$\H_c^\ast(\O_{k,S},A) \simeq \varprojlim_n \H_c^\ast(\O_{k,S},A/p^n A)$
by the five lemma.

If $R$ is a commutative ring and $\Lambda$ a perfect complex of $R$-modules then
$H^i(\Lambda)$ is trivial for all but finitely many $i$ and $\Det_R \Lambda$ is defined
(see \cite{KnMu76} for the definition and properties of the determinant functor).
Moreover, if $H^i(\Lambda)$ is projective for all $i$, then by \cite[Proposition 2.1(e)]{bufl99} there is
a canonical isomorphism $\Det_R \Lambda \to \otimes_i \Det_R^{(-1)^i} H^i(\Lambda)$.
If $R$ is the valuation ring $\Oe$ in a finite extension $E$ of $\Q_p$, $H^i(\Lambda)$ is finite for all $i$, and trivial for all but
finitely many $i$, then we have a canonical composition
\begin{equation} \label{theta}
 \Det_{\Oe} \Lambda
\simeq
 \otimes_i \Det_{\Oe}^{(-1)^i} H^i(\Lambda)
\to
 \otimes_i \Det_E^{(-1)^i} (H^i(\Lambda)\otimes_{\Oe} E)
\simeq
 E
\,.
\end{equation}

\begin{theorem} \label{etnc}
Let $ k $ be a totally real number field, $ p $ a prime number, $ E $ a finite extension of $ \Q_p $
with valuation ring $ \Oe $, $m$ a negative integer, $\eta$ an Artin character of $G_k$ realizable
over $E$ such that $\eta(c) = (-1)^{m-1} \eta(\id_{\kbar})$ for every complex conjugation $c$ in $G_k$,
$\vlat{E}{\eta^\vee}$ an Artin representation of $G_k$
over $E$ with character $\eta^\vee$, $\lat{E}{\eta^\vee}$ an $\Oe$-lattice for $\eta^\vee$,
and $ S $ a finite set of primes of $k$ containing the primes above $p$
as well as all the finite primes of $ k $ at which $ \eta $ is ramified.
Then the following hold, where for $p=2$ we make Assumption~\ref{2-mainc}.
\begin{enumerate}
\item
$R\Gamma_c(\O_{k,S},\lat{E}{\eta^\vee}(m))$ and $R\Gamma_c(\O_{k,S},\vlat{E}{\eta^\vee}(m))$
are perfect complexes.

\item
$\H_c^i(\O_{k,S},\lat{E}{\eta^\vee}(m))$ is finite for all $i$ and trivial if $i \ne 2,3$.

\item
In $E$ we have
\begin{equation} \label{etnc_eq}
L_S^\ast(m,\eta,k) \cdot \Oe = \theta (\Det_{\Oe} R\Gamma_c(\O_{k,S},\lat{E}{\eta^\vee}(m)))
\,,
\end{equation}
with $L_S^\ast(m,\eta,k)$ as in Section~\ref{pL-functions}, and
$\theta$ the composition~\eqref{theta} for $\Lambda = R\Gamma_c(\O_{k,S},\lat{E}{\eta^\vee}(m))$.

\end{enumerate}
\end{theorem}
\begin{proof}
Part~(1) follows from \cite[Theorem~5.1]{Flach00}.  For parts~(2) and~(3), we
abbreviate $\lat{E}{\eta^\vee}(m)$ to $M$, and begin by computing $\H_c^*(\O_{k,S},M)$.
We have $\H_c^0(\O_{k,S},M) = 0$
because $H^0(G_S,M) \rightarrow \oplus_{v \in S} H^0(G_{w_v},M)$ is injective.
For a real place $v$ in $S$, note that the complex conjugation $c_v$ in $\Gal(\overline{k_v}/k_v)$
acts on $M$ as multiplication by $-1$, so that
$H^0(G_{w_v},M/p^nM) \simeq \widehat{H}^0(G_{w_v},M/p^nM)$, where the right-hand side denotes
Tate cohomology.  Therefore, for $i \ge 1$, $\H_c^i(\O_{k,S},M/p^nM) \simeq H_c^i(\O_{k,S},M/p^nM)$
where the right-hand side denotes the cohomology with compact support as in \cite[Section~II \S~2, p.203]{Mil86}.
Taking inverse limits and using~\eqref{alpha_duality} we have
\begin{equation*}
\H_c^i(\O_{k,S},M) \simeq \left( H^{3-i}(\O_{k,S},\qlatmtt{E}{\eta \op^{1-m}}{1-m}) \right)^\vee
\end{equation*}
for $i \ge 1$,
where $\qlat{E}{\eta}$ and $\qlat{E}{\eta \op^{1-m}}$ are obtained from $\lat{E}{\eta^\vee}$.
We see as just after~\eqref{CL-statement} that $ L_S^\ast(m,\eta,k) \ne 0$
because $m < 0$, and hence $L_{p,S}(m,\eta \op^{1-m},k) \ne 0$ by~\eqref{preLSinterpol}.
Then according to Theorem~\ref{introtheorem}(3),
$H^j(\O_{k,S}, \qlatmtt{E}{\eta \op^{1-m}}{1-m})$ is finite for
$j = 0, 1$ and trivial for $j=2$, completing the proof of part~(2).

For part~(3), we note that for any finite $\Oe$-module $A$, $I^{[E:\Q_p]} = (\# A) \cdot \Oe$ with
$I$ the image of the composition $ \Det_{\Oe} A \to \Det_E (A \otimes_{\Oe} E) \simeq E $.
Therefore
\begin{alignat*}{1}
\theta \left( \Det_{\Oe} R\Gamma_c(\O_{k,S},M) \right)^{[E:\Q_p]}
& =
\frac{\# \H_c^2(\O_{k,S},M)}{\# \H_c^3(\O_{k,S},M)} \cdot \Oe \\
& =
\frac{\# H^1(\O_{k,S},\qlatmtt{E}{\eta \op^{1-m}}{1-m})}{\# H^0(\O_{k,S},\qlatmtt{E}{\eta \op^{1-m}}{1-m})} \cdot \Oe \\
& =
\EC_S(m,\eta\op^{1-m},k)^{-[E:\Q_p]} \cdot \Oe \\
& =
L_{p,S}(m,\eta\op^{1-m},k)^{[E:\Q_p]} \cdot \Oe \\
& =
 L_S^*(m,\eta,k)^{[E:\Q_p]} \cdot \Oe
\,,
\end{alignat*}
where the last two equalities follow from Theorem~\ref{introtheorem}(3) and the interpolation
formula~\eqref{preLSinterpol}.
\end{proof}

\begin{corollary} \label{etnccorollary}
Let $k$ be a totally real number field, $K/k$ a finite Galois extension with Galois group $G$ and
$m$ a negative integer.  Assume that $K$ is totally real if $m$ is odd and $K$ is a CM field if $m$ is even.
Let $\mathfrak{M}$ denote a maximal $\Z[G]$-order inside $\Q[G]$.  Let $\pi_m = 1$ if $m$ is odd and
$\pi_m = (1 - c)/2$ if $m$ is even, where $c$ is the unique complex conjugation in $G$.
Let $p$ be a prime number and if $p = 2$ make Assumption~\ref{2-mainc}.
Then the $p$-part of the equivariant Tamagawa number conjecture $($\cite[Conjecture 6, p.535]{bufl99}$)$ holds for the
pair $(\pi_m h^0(\Spec K)(m),\pi_m\mathfrak{M})$.
\end{corollary}

\begin{proof}
Let $F$ be a finite extension of $\Q$ such that all the irreducible
$ \Qbar $-valued characters of $G$
can be realized over $F$, and let $\mathfrak{M}_F$ be a maximal $\O_F[G]$-order in $F[G]$.
By \cite[Theorem~4.1]{bufl99}, the equivariant
Tamagawa number conjecture for the pair $(\pi_m h^0(\Spec K)(m),\pi_m\mathfrak{M})$ is then equivalent to
that for the pair $(\pi_m h^0(\Spec K)(m),\pi_m\mathfrak{M}_F)$, which can be decomposed according to
all irreducible characters $\psi$ of $G$ if $ K $ is totally
real, and those $ \psi $ satisfying $\psi(c) = -\psi(\id_{\kbar})$
if $K$ is a CM field.
For such a character $\psi$,
if $\pi_\psi = (\#G)^{-1} \sum_{g \in G} \psi(g) g^{-1}$ is the standard idempotent
in $F[G]$ corresponding to $\psi$, then
the special value of the motivic $L$-function associated to
$\pi_\psi h^0(\Spec K)(m)$
 is $L(m,\psi,k)$ in $ F $.  If we fix an embedding $\sigma: F \to E$, with
$E$ a finite extension of $\Q_p$, then $\sigma (L(m,\psi,k)) = L^\ast(m,\sigma \circ \psi,k)$, so the
$p$-part of the equivariant Tamagawa number conjecture
states that~\eqref{etnc_eq} holds with $\eta = \sigma \circ \psi$:
in the formulation of the equivariant Tamagawa number conjecture in \cite{bufl99}, 
the canonical composition
$\theta$ in~\eqref{theta} is multiplied by $\prod_{v \in S \setminus P} \Eul{v}^\ast(m,\eta,k)$
(see \cite[Section 3]{bufl99} for details).
\end{proof}

\begin{remark}
Note that if the interpolation formula~\eqref{preLSinterpol} holds for $m=0$ and $L_S^\ast(0,\eta,k) \ne 0$
then the statements of Theorem~\ref{etnc} are also true for $m=0$.  Moreover, Corollary~\ref{etnccorollary}
holds in this case if we further assume that the reciprocal Euler factors $\Eul{v}^\ast(0,\eta,k)$ are
non-trivial for all primes $v$ in $S$.
\end{remark}

\begin{remark} \label{finalremark}
(1)
Assumption~\ref{2-mainc} holds if $p=2$, $k = \Q$ and $\eta$ is the trivial character (see Remark~\ref{itholds}),
so Corollary~\ref{etnccorollary} holds without any assumptions when $m$ is odd and $k = K = \Q$.

(2)
A stronger version of Corollary~\ref{etnccorollary}, with the maximal
order replaced by $\Z[G]$ but $p \ne 2$, is proven by Burns in \cite[Corollary~2.10]{Burns10pre}
under the assumption that certain Iwasawa-theoretic $p$-adic $\mu$-invariants of $K$ are trivial.
(If $G$ is Abelian then this was already proved under similar
assumptions in \cite[Theorem~3.3]{BaBu}.)
This is also proved without assumptions if $k = \Q$ and $K/\Q$ is Abelian by
Burns and Greither in \cite[Corollary~8.1]{BuGr} for $p \ne 2$, and by Flach in \cite[Theorem~5.1]{Flach08}
and \cite[Theorem~1.2]{Flach11} for $p = 2$.
Huber and Kings \cite[Theorem~1.3.1]{Hu-Ki03} have also proved Corollary~\ref{etnccorollary}
in the case $k = \Q$, $K/\Q$ Abelian and $p$ odd.
In fact, in each case the results hold at every integer $m$,
and in \cite{BuGr,Flach08,Flach11,Hu-Ki03} the results hold for 
the pair
$(h^0(\Spec K)(m),\mathfrak{M})$ instead of $(\pi_m h^0(\Spec K)(m),\pi_m \mathfrak{M})$.
\end{remark}

\nocite{*}
\bibliographystyle{plain}
\bibliography{references}

\end{document}